\documentclass{article}
\usepackage{authblk}
\usepackage{amssymb,amsfonts,units,nicefrac,amsthm,stmaryrd,multicol,cmll}

\DeclareFontFamily{U}{MnSymbolC}{}
\DeclareSymbolFont{MnSyC}{U}{MnSymbolC}{m}{n}
\DeclareMathSymbol{\boxdot}{\mathbin}{MnSyC}{"76}
\DeclareMathSymbol{\diamonddot}{\mathbin}{MnSyC}{"7E}
\DeclareFontShape{U}{MnSymbolC}{m}{n}{
    <-6>  MnSymbolC5
   <6-7>  MnSymbolC6
   <7-8>  MnSymbolC7
   <8-9>  MnSymbolC8
   <9-10> MnSymbolC9
  <10-12> MnSymbolC10
  <12->   MnSymbolC12}{}

\usepackage{subcaption} 

\usepackage{xcolor,bm,tikz,stackrel}
\usepackage[inline]{enumitem}

\newcommand{\lbl}[1]{\mathrm {Lab}_{#1}}

\DeclareFontFamily{U}{mathb}{\hyphenchar\font45}
\DeclareFontShape{U}{mathb}{m}{n}{
      <5> <6> <7> <8> <9> <10> gen * mathb
      <10.95> mathb10 <12> <14.4> <17.28> <20.74> <24.88> mathb12
      }{}
\DeclareSymbolFont{mathb}{U}{mathb}{m}{n}
\DeclareFontSubstitution{U}{mathb}{m}{n}

\DeclareMathSymbol{\sqsubset}{3}{mathb}{"80}
\DeclareMathSymbol{\sqsupset}{3}{mathb}{"81}

\DeclareMathSymbol{\sqSubset}       {3}{mathb}{"94}
\DeclareMathSymbol{\sqSupset}       {3}{mathb}{"95}

\DeclareMathSymbol{\sqsubseteq}{3}{mathb}{"84}
\DeclareMathSymbol{\sqsupseteq}{3}{mathb}{"85}

\DeclareMathSymbol{\sqsubsetneq}{3}{mathb}{"88}
\DeclareMathSymbol{\sqsupsetneq}{3}{mathb}{"89}

\usepackage[all]{xy}
\usepackage{tikz-cd}
\usepackage{xcolor}
\definecolor{darkolivegreen}{rgb}{0.33, 0.42, 0.18}
\definecolor{darkblue}{rgb}{0.0, 0.0, 0.55}
\definecolor{falured}{rgb}{0.5, 0.09, 0.09}	
\definecolor{tan}{rgb}{0.82, 0.71, 0.55}
\definecolor{turquoise}{rgb}{0.19, 0.84, 0.78}
\definecolor{lightgray}{rgb}{0.83, 0.83, 0.83}
\definecolor{babyblue}{rgb}{0.54, 0.81, 0.94}
\definecolor{applegreen}{rgb}{0.55, 0.71, 0.0}
\definecolor{amber}{rgb}{1.0, 0.75, 0.0}
\definecolor{atomictangerine}{rgb}{1.0, 0.6, 0.4}

\usepackage[inline]{enumitem}
\usepackage{tasks}
\usetikzlibrary{arrows,fit}
\usepackage{rotating}
\usetikzlibrary{trees,decorations.pathmorphing}
\usetikzlibrary{arrows.meta,shapes,positioning,calc,automata,quotes,patterns}
\usepackage{mathtools}
\usepackage{csquotes}
\usepackage{centernot}

\newlength{\eparindent}

\newlist{Cases}{enumerate}{9}
\setlist[Cases,1]{label={\sc Case {\rm \arabic*}},wide, labelwidth=!, labelindent=0pt,listparindent=\eparindent}
\setlist[Cases,2]{label*= {\rm .\arabic*},wide, labelwidth=!, labelindent=0pt,listparindent=\eparindent}
\setlist[Cases,3]{label*= {\rm .\arabic*},wide, labelwidth=!, labelindent=0pt,listparindent=\eparindent}
\setlist[Cases,4]{label*= {\rm .\arabic*},wide, labelwidth=!, labelindent=0pt,listparindent=\eparindent}
\setlist[Cases,5]{label*= {\rm .\arabic*},wide, labelwidth=!, labelindent=0pt,listparindent=\eparindent}
\setlist[Cases,6]{label*= {\rm .\arabic*},wide, labelwidth=!, labelindent=0pt,listparindent=\eparindent}
\setlist[Cases,7]{label*= {\rm .\arabic*},wide, labelwidth=!, labelindent=0pt,listparindent=\eparindent}
\setlist[Cases,8]{label*= {\rm .\arabic*},wide, labelwidth=!, labelindent=0pt,listparindent=\eparindent}
\setlist[Cases,9]{label*= {\rm .\arabic*},wide, labelwidth=!, labelindent=0pt,listparindent=\eparindent}

\newcommand{\martin}[1]{}
\newcommand{\define}[1]{{\em #1}}
\newcommand{\term}[1]{{#1}}
\newcommand{\dom}[1]{|#1|}
\newcommand{\altdefine}[1]{}
\newcommand{\altterm}[1]{}
\newcommand{\isf}{\mathsf{IS4}}

\newcommand{\simfor}[1]{{\text{\sc sim}(#1)}}
\newcommand{\BN}{{\mathfrak C}_{\sf IGL}}
\newcommand{\CIKF}{{\mathfrak C}_{\sf IK4}}

\newcommand{\langfull}{{\mathbb L}^{2}}

\newcommand{\ikf}{\mathsf{IK4}}

\newcommand{\csf}{\mathsf{CS4}}

\newcommand{\restr}[2]{{#1}{\uparrow} {#2} }

\usepackage{hyperref}
\hypersetup{
    colorlinks=true,
    linkcolor=blue,
    citecolor=blue,
    urlcolor=blue
}

\usepackage{doi}

\usepackage[numbers]{natbib}
\usepackage{url}  

\def\simu{\mathrel{\vcenter{\offinterlineskip \vskip 0ex\hbox{\rotatebox{30}{$\shortrightarrow$}} \vskip -.9ex\hbox{\hskip -.2ex\rotatebox{-30}{$\shortrightarrow$}}}}}

\newcommand{\notnec}{{\overline\nec}}
\newcommand{\notnecd}{{\overline\boxdot}}

\def\llsim{\mathrel{\vcenter{\offinterlineskip
\hbox{$\ll$}\vskip.1ex\hbox{\hskip.25ex $\sim$}}}}
\def\ggsim{\mathrel{\vcenter{\offinterlineskip
\hbox{$\gg$}\vskip.1ex\hbox{\hskip.25ex $\sim$}}}}

\def\lleq{\mathrel{\vcenter{\offinterlineskip\hbox{$\ll$} \vskip -1.4ex\hbox{\hskip0.25ex$\underline{\phantom{<}}$}}}}
\def\ggeq{\mathrel{\vcenter{\offinterlineskip\hbox{$\gg$} \vskip -1.4ex\hbox{\hskip0.25ex$\underline{\phantom{<}}$}}}}

\def\lb{\left\llbracket}
\def\rb{\right\rrbracket}
\def\<{\left (}

\def\>{\right )}
\def\({\left (}
\def\){\right )}

\newcommand{\simrel}{\mathrel E}

\newcommand{\om}{\omega}

\newcommand{\al}{\alpha}

\newcommand{\be}{\beta}

\newcommand{\ga}{\gamma}

\newcommand{\de}{\delta}

\newcommand{\fA}{{\cl A}}
\newcommand{\fB}{{\cl B}}
\newcommand{\fC}{{\cl C}}
\newcommand{\fD}{{\cl D}}

\newcommand{\peq}{\preccurlyeq}
\newcommand{\sbs}{\sqsubset}
\newcommand{\sbeq}{\sqsubseteq}
\newcommand{\sps}{\sqsupset}

\newcommand{\diam}{\lozenge}
\newcommand{\val}[1]{\lb #1 \rb}

\newcommand{\ignore}[1]{}
\newcommand{\peqT}{\peq_{\Sigma}}

\newcommand{\subT}{\subseteq_{\Sigma}}
\newcommand{\sqsT}{\sqsubset_{\Sigma}}

\newcommand{\type}[1]{\mathsf T_{ #1 }}

\newcommand{\lanfull}{{\mathbb L}}

\newcommand{\landi}{\cl L_\diam}

\newcommand{\cl}{\mathcal}
\newcommand{\imp}{\mathop \to}
\newcommand{\seq}{\succcurlyeq}

\newcommand{\forrefs}[1]{}

\newcommand{\bfrm}[1]{#1^\nec}

\renewcommand\>{\right )}
\renewcommand\({\left (}
\renewcommand\){\right )}

\newcommand{\rel}{\sqsubset}
\newcommand{\ler}{\sqsupset}
\newcommand{\srel}{\sqsubsetneq}
\newcommand{\sler}{\sqsupsetneq}
\newcommand{\rrel}{\sqsubseteq}
\newcommand{\rler}{\sqsupseteq}

\newcommand{\nec}{\Box}
\newcommand{\ps}{\Diamond}

\makeatletter
\newcommand{\mlabel}[2]{#2\def\@currentlabel{#2}\label{#1}}
\makeatother

\DeclareSymbolFont{AMSb}{U}{msb}{m}{n}

\begin{document}

\newtheorem{theorem}{Theorem}[section]
\newtheorem{lemma}[theorem]{Lemma}
\newtheorem{proposition}[theorem]{Proposition}

\theoremstyle{definition}
\newtheorem{claim}[theorem]{Claim}
\newtheorem{corollary}[theorem]{Corollary}
\newtheorem{definition}[theorem]{Definition}
\newtheorem{example}[theorem]{Example}
\newtheorem{remark}[theorem]{Remark}
\newtheorem{question}[theorem]{Question}

  \title{{Fine Selection for Intuitionistic Modal Logic:\\
  $\sf IK4$ and $\sf IGL$}}
\author{David Fern\'andez-Duque}
\affil{Department of Philosophy, University of Barcelona\\{\tt fernandez-duque@ub.edu}}
\maketitle

\begin{abstract}
We extend Fine's selection method to the setting of intuitionistic modal logic and use it to provide a model-theoretic proof that Fischer Servi-style $\sf IK4$ has the finite model property and $\sf IGL$ has a sound and complete Hilbert calculus.
\end{abstract}


\maketitle
\section{Introduction}

In the setting of classical modal logic, Fine selection is a clever and powerful technique, useful for e.g.~proving that many transitive modal logics have the finite model property or logics with converse well-foundedness conditions such as $\sf GL$ are complete~\cite{Fin74c,Fine1974}.
In this paper, we will adapt this method to the intuitionistic setting and use it for proving analogous results in the context of intuitionistic modal logic.
 
Much as classical modal logics are traditionally interpreted over structures with one binary relation, intuitionistic modal logics may be naturally interpreted over birelational structures, with a preorder $\peq$ being used to interpret the intuitionistic implication and a second relation $\rel$ being used to interpret the modal diamond.
The precise interaction between the two relations gives rise to various intuitionistic variants of modal logics, with problems such as completeness or the finite model property often leading to challenging problems, especially when the modal relation is transitive.
In this context, some of the more prominent transitive modal logics that have been studied are the `constructive' $\csf$ of Alechina et al.~\cite{AlechinaMPR01}, and the `intuitionistic' $\isf$ and $\ikf$ of Fischer Servi~\cite{servi1977modal,servi1984axiomatizations}.

Whether each of these logics enjoys the finite model property had been a longstanding open question, in the case of {$\csf$} since at least 2001 \cite{AlechinaMPR01}, and of {$\isf$} and $\sf IK4$, at least since 1994 \cite{Simpson94}.
The finite model property was recently established for \altterm{def:csf}{$\csf$}~\cite{balbiani2026} and \altterm{def:isf}{$\isf$}~\cite{GKMMS23}.
A related problem is the decidability of these logics, which can be shown by proof-theoretic techniques without appealing to the finite model property~\cite{piazza}.

Another transitive intuitionistic modal logic that has recently gained attention is the intuitionistic Gödel-Löb logic $\sf IGL$ of Das et al.~\cite{DasIGL}.
Semantics for $\sf IGL$ is defined over a natural class of birelational structures with a converse well-foundedness condition, but syntactically, $\sf IGL $ is defined via an infinitary calculus (with a variant given by Aguilera and Pacheco~\cite{Aguilera2025IntuitionisticGW}).
The decidability and finite model property has recently been established using a variant of the techniques we present here~\cite{FDIGL}, but whether a natural Hilbert calculus exists remained open, with the possible exception of a proof of completeness of the diamond-free fragment by Becker, Das, and Marin via proof-theoretic means, reported in private communication.

In parallel to the classical case, we show that Fine selection can also be used to establish the finite model property for $\sf IK4$ and the completeness of $\sf IGL$ for a natural Hilbert calculus.
It does, however, require some essential modifications.
Kruskal's theorem~\cite{krusty} is present in all known proofs of decidability or finite model property for $\sf IK4$ or $\sf IS4$ and is likely to be required in some sense. 
Our notion of Fine selection operates not only on isolated formulas, but on finite, tree-like models more globally, and Kruskal's theorem is used as the foundation of a transfinite induction argument that an iterative selection procedure terminates, producing a finite model.

\subsection{How to read this paper}
This work would have best been presented as two separate papers, respectively dealing with $\ikf$ and $\sf IGL$, and this is still how I plan to publish it.
Unfortunately, they were apparently too similar and I had the dubious honour of having a preprint rejected by the moderators.
Thus I was forced to create this monstrosity, but we can use our imagination and pretend that it is only about one of $\ikf$ or $\sf IGL$ at a time:

\begin{itemize}

\item Readers only interested in $\ikf$ can safely skip Section~\ref{secStrict} and ignore the notion of strict Finality, which is referenced later in the text.

\item Readers only interested in $\sf IGL$ can skip Sections~\ref{sectBloom} and \ref{sectHatch}.
Blooming and hatching defects can be ignored throughout.

\end{itemize}

Aside from this, there are a handful of results that are only relevant for one of the two logics, but this is mentioned when it happens.
Even for readers interested in both logics, I would recommend focusing on $\sf IGL$ first.

\subsection{AI use statement}
Well, I guess we have to put these now, so here goes.
This is an organic, eco-friendly paper, and the only neural network harmed in producing it was my brain.
Or at least I would have liked to write this, except that Ian Shillito sent me an AI-generated report and I made a handful of changes based on it: most notably, it remarked that the ordinal interpretation in the original draft was identically zero (always check your base cases!).
I am very grateful to Ian for his interest, but ultimately, I chose not to follow the AI revision route.

\section{The logics $\sf IK4$ and $\sf IGL$}

With all that out of the way, let us prove some theorems, while we're still allowed to.
The logics we are interested in share the language of intuitionistic modal logic.
Fix a countably infinite set $\mathbb P$ of propositional variables.
Then, the \define{intuitionistic modal language} $\lanfull$ is defined by the grammar (in Backus--Naur form)
\[\varphi,\psi \coloneqq  \   p \  | \   \bot  \ |  \ \left(\varphi\wedge\psi\right) \  |  \ \left(\varphi\vee\psi\right)  \ |  \ \left(\varphi\imp \psi\right)    \  | \  \ps\varphi \  |  \ \nec\varphi  , \]
where $p\in \mathbb P$.
We may also use the abbreviations $\neg \varphi = \varphi\to\bot$, $\boxdot \varphi = \varphi\wedge\nec\varphi$, and $\diamonddot\varphi = \varphi\vee\ps\varphi$.

\begin{definition}
The logic $\ikf$ is the smallest logic containing all intuitionistic tautologies and closed under the following axioms and rules.
\smallskip

\noindent\fbox{\begin{minipage}{\textwidth}
\begin{tasks}[label={}, label-width=3em,label-align = {right}](2)
\task[\mlabel{ax:k:box}{\ensuremath{\bm{ \mathrm K_{\nec}}}}] $\nec (p \imp q) \imp(\nec p \imp \nec q)$
\task[\mlabel{ax:k:dia}{\ensuremath{\bm{\mathrm K_{\ps}}}}] $\nec (p \imp q) \imp( \ps p \imp \ps q)$
\task[\mlabel{ax:trans:box}{\ensuremath{\bm{\mathrm 4_{\nec}}}}] $\nec p \imp \nec \nec p$
\task[\mlabel{ax:trans:dia}{\ensuremath{\bm{\mathrm 4_{\ps}}}}] $\ps \ps p \imp \ps p$
\task[\mlabel{ax:dp}{\ensuremath{\bm{\mathrm{DP}}}}] $\,\ps(p \vee q) \imp \ps p \vee \ps q$ 
\task[\mlabel{ax:fs}{\ensuremath{\bm{\mathrm{FS2}}}}] $\,(\ps p \imp \nec q ) \rightarrow \nec (p \imp q )$ 
\task[\mlabel{ax:null}{\ensuremath{\bm{\mathrm{N}}}}] $\neg \ps \bot$
\task[]
\task[\mlabel{rl:mp}{\ensuremath{\bm{\mathrm{MP}}}}] $\dfrac{\varphi \imp \psi \hspace{10pt} \varphi}{\psi}$
\task[\mlabel{rl:nec}{\ensuremath{\bm{\mathrm{Nec}}}}] $\dfrac{\varphi}{\nec \varphi}$
\end{tasks}
\end{minipage}}

The logic $\sf IGL$ is obtained by adding the following axioms to $\ikf$.

\smallskip

\noindent\fbox{\begin{minipage}{\textwidth}
\begin{tasks}[label={}, label-width=3em,label-align = {right}](2)
\task[\mlabel{ax:gl:box}{\ensuremath{{\text{\bf L\"ob}_{\nec}}}}] $\nec (\nec p\to p) \imp \nec p$
\task[\mlabel{ax:gl:dia}{\ensuremath{{\text{\bf L\"ob}_{\ps}}}}] $\nec (p \imp q\vee \ps p) \imp( \ps p \imp \ps q)$
\end{tasks}
\end{minipage}}
\end{definition}

Aside from the more familiar names, \ref{ax:dp} stands for {\em disjunctive possibility} and \ref{ax:fs} is the second of Fischer Servi's connecting axioms~\cite{servi1977modal}.

Semantics for these logics are based on birelational frames satisfying some combinatorial conditions.
As these conditions will pop up in various contexts, let us begin by studying them abstractly.

\begin{definition}
A {\em transitive frame} is a pair $\fA=(\dom{\fA},\sqsubset_\fA)$, where $\dom{\fA}$ is any set and $\sbs_\fA$ is a transitive relation on $\dom{\fA}$.
We write:
\begin{itemize}

\item $w\sps _\fA v$  if $v\sbs_\fA w$, $w\sbeq_\fA v $ if $w\sbs_\fA v $ or $w=v $,

\item $w\equiv_\fA   v $  if $w\sbeq_\fA v $ and $v\sbeq_\fA w $,

\item $w \srel_\fA  v $  if $w\sbeq_\fA v $ but $v \not \sbeq_\fA w $, and

\item $w\sbs_\fA^1  v $  if $v$ is an immediate $\sbs_\fA$-successor of $w$, in the sense that $w\srel_\fA v $ and there is no $u$ such that $w \srel_\fA  u \srel _\fA v$.

\end{itemize}
\end{definition}

We will often drop the subindex and write e.g.~$\sbs$ for $\sbs_\fA$.
Note that $\equiv$ is an equivalence relation, and the equivalence class of $w$ is denoted $[w]$.
Any set  $ C\subseteq [w] $ for some $w$ is a {\em $\sbs$-cluster,} and, if equality holds, $C$ is a {\em full cluster.}
While $\rel$ usually denotes a transitive relation, $\peq$ usually denotes a partial order or preorder, in which case $x\prec y$ means that $x\peq y$ but $x \neq y$.
Note that this is equivalent to $x \peq y$ but $y \not\peq  x$ with $\peq$ is a partial order, but {\em not } when it is a preorder (in fact, in this case, $\prec$ may not even be transitive).
Note that bracket notations for clusters and intervals will exclusively refer to relations denoted $\rel$.

If $R\subseteq X\times Y$, $A\subseteq X$ and $B\subseteq Y$, we write  $A \mathrel R B$ if for all $a\in A$, there exists $b\in B$ such that $a \mathrel R b$.
If one of the two sets is a singleton $\{x\}$, we may write e.g.~$x\mathrel R B$ instead of $\{x\}\mathrel R B$.
In the case of a transitive relation $\rel$, the set of full $\rel$-clusters thus forms an antisymmetric transitive set under $\rel$.
Note that it is possible for a cluster to be irreflexive, but in this case, it must be a singleton $\{a\}$, where $a\not\rel a$.

Next, we define two confluence properties that we will demand from the intuitionistic order $\peq$.
However, we will also consider other relations with these properties, so we define them more generally.

\begin{definition}
Let $\fA$ and $\fB$ be transitive frames.
A relation $R\subseteq \dom{\fA}\times\dom{\fB}$ is
\begin{enumerate}

\item {\em forward confluent} if whenever $y \sps_\fA x \mathrel R x'$, there is $y'$ such that $y  \mathrel R y' \sps_\fB x'$,

\item {\em backward confluent} if whenever $x \sbs_\fA y \mathrel R y'$, there is $x'$ such that $x  \mathrel R x' \sbs_\fB y'$, and

\item {\em expansive} if it is both forward and backward confluent.

\end{enumerate}
If $R$ is a total, expansive relation between $\dom{\fA}$ and $\dom{\fB}$, we write $R\colon \fA \simu \fB$.
An expansive relation which is a function is a {\em homomorphism.}
\end{definition}

When not clear from context, we may instead write e.g.~{\em expansive from $\sbs_\fA$ to $\sbs_\fB$,} or {\em expansive over $\sbs_\fA$} when $\fA=\fB$.
Note that a homomorphism is equivalently a function $h$ such that $x\rel_\fA y$ implies that $h(x) \rel_\fB h(y)$, as expected.

Expansive relations are the bisimulation-invariant analogue of a homomorphism, arguably making them the correct notion of `morphism' when studying modal properties on frames.
Because, unlike homomorphisms, they are multivalued, we have access to several operations on them that would normally break functionality.

\begin{proposition}\label{propExpandingClos}
Let $\fA$, $\fB$, $\fC$ be transitive frames.

\begin{enumerate}

\item If $\mathfrak R$ is a family of relations $R\subseteq \dom{\fA}\times \dom{\cl B}$ and every $R\in \mathfrak R$ is expansive, then so is $\bigcup \mathfrak R$.

\item If $R \subseteq \dom{\fA}\times\dom{\fB}$ and $S\subseteq \dom{\fB}\times\dom{\fC}$ are both expansive, then $S\circ R \subseteq \dom{\fA}\times\dom{\fC} $ is expansive.

\item If $R \subseteq \dom{\fA}\times\dom{\fA}$ is expansive, then so are its transitive closure  $R^+$ and its transitive, reflexive closure $R^*$.

\end{enumerate}

\end{proposition}

\begin{proof}
It suffices to show that forward confluent relations are closed under the above three operations, since $R$ is backward confluent with respect to $\rel$ if and only if $R $ is forward confluent with respect to $\ler$, and we can apply the claim to $\ler$.
Moreover, $R^+$ and $R^*$ are unions of iterated compositions of $R$, so the third claim follows from the other two.

For the first claim, if $a\rel_\fA a'$ and $a\mathrel{\bigcup \mathfrak R} b$, then $a\mathrel R b$ for some $R\in\mathfrak R$.
Since $R$ is forward confluent, there is $b'\ler_\fB b$ such that $a' \mathrel R b'$, hence $a' \mathrel{\bigcup \mathfrak R} b'$, as required.

For the second claim, if $a\rel_\fA a'$ and $a \mathrel{S\circ R} c$, then there is $b$ so that $a\mathrel R b$ and $b\mathrel Sc$.
By forward confluence for $R$, there is $ b ' \ler_\fB  b$ such that $a'\mathrel R b'$, then by forward confluence for $S$, there is $ c ' \ler _\fC c$ such that $b'\mathrel S c'$.
We then have that $c\rel_\fC c'$ and $a' \mathrel{S\circ R} c'$, as required.
\end{proof}

In view of Proposition~\ref{propExpandingClos}, we can take joins of expansive preorders.

\begin{lemma}\label{lemmExpJoin}
If $\cal F$ is a transitive frame and $\mathfrak R$ is a family of expansive relations on $\cal F$, then there is a least expansive preorder containing $\bigcup \mathfrak R$.
We denote this relation by $\bigcurlyvee \mathfrak R$.
\end{lemma}

We write $R\curlyvee S$ instead of $\bigcurlyvee \{R,S\}$.



\begin{definition}\label{DefSem}
An  {\em expansive transitive frame} or {\em $\ikf$ frame} is a triple $\fA=(\dom{\fA}, \peq_{\fA}  ,\sqsubset_{\fA})$, where $\sbs_{\fA}$ is a transitive relation on $\dom{\fA}$ and $\peq_{\fA}$ is a preorder on $\dom{\fA}$ which is expansive over $\sbs_\fA$.
We denote the class of expansive transitive frames by $\CIKF$.
\end{definition}

When clear from context, we may omit subindexes and write e.g.~$\peq$ instead of $\peq_\fA$.
We define ${\sqSubset_\fA} = {\rel_\fA\circ \peq_\fA}$ and ${\llsim_\fA} = {\sqSubset_\fA \cup \peq_\fA}$.
If $w\llsim_\fA v$ and $v \llsim_\fA w $, we write $w\cong_\fA v$.
The following is easily checked using backward confluence (see e.g.~\cite{santiago2026}).

\begin{lemma}\label{lemLL}
If $\fA $ is an expansive transitive frame, then ${\sqSubset_\fA} $ is transitive and ${\llsim_\fA}$ is a preorder.
Moreover, ${\rel_\fA}\subseteq {\sqSubset_\fA}$ and
\[{\sqSubset_\fA} = {\sqSubset_\fA}\circ {\llsim_\fA} = {\llsim_\fA}\circ {\sqSubset_\fA}. \]
\end{lemma}

Note that, contrary to common usage, double symbols such as $\sqSubset$ indicate a larger relation (as a set), rather than a stricter one.

In order for an expansive transitive frame to validate $\sf IGL$, it requires one additional condition.

\begin{definition}
A preordered set $(A,{\rel})$ is {\em converse well founded} if there are no infinite chains $(a_i)_{i<\omega} \subseteq A$ such that $a_i \rel a_{i+1}$ for all $i<\omega$.

An {\em $\sf IGL$ frame} is an expansive transitive frame $\fA$ such that $\sqSubset_\fA$ is converse well founded.
We denote the $\sf IGL$ frames by $\BN$.
\end{definition}

Note that any infinite $\rel$-chain is also a $\sqSubset$-chain, so if $\sqSubset$ is converse well founded, then so is $\rel$.
Note also that expansive transitive frames may be seen as first order structures with one predicate $p(x)$ for each atom $p$ and two relations $\peq$, $\rel$.
In order to appeal to properties of first order logic, instead of defining semantics directly on the modal language, we will define them via the standard translation into first order logic.
The notational conventions we have introduced can be defined within the first order language, e.g.~$x\seq y := y\peq x$ and $x\sqSubset y := \exists z (x\peq z\wedge z\rel y)$, and we use other standard abbreviations, including writing $\exists v \seq  w \ \varphi$ instead of $\exists v (w\peq v\wedge \varphi)$.

\begin{definition}
Given a formula $\varphi\in\lanfull$, we define its {\em first order translation}, denoted  $\varphi(x)$, where $x$ is a first order variable, inductively on $\varphi$ as follows.
\begin{itemize}

\item If $p$ is a propositional atom, then $p(x)$ is an atomic predicate,

\item $\bot(x) = x\neq x$,

\item $(\varphi\wedge\psi)(x) = \varphi(x)\wedge\psi(x)$,

\item $(\varphi\vee\psi)(x) = \varphi(x)\vee\psi(x)$,

\item $(\varphi\to\psi)(x) = \forall y  \seq  x (  \neg \varphi(y) \vee \psi(y))$,

\item $\ps\varphi (x) = \exists y  \ler  x \  \varphi(y) $, and

\item $\nec \varphi (x) = \forall y  \sqSupset  x \   \varphi(y)  $.

\end{itemize}
\end{definition}
Note that the displayed negation is classical.
Below, if $\varphi(x)$ is a first order formula with one free variable and $w$ is an element of a model $\cl M$, we write $\cl M\models \varphi(w) $ to mean that $\varphi(x)$ is true on $\cl M$ under the assignment that maps $x$ to $w$.

\begin{definition}
An {\em intuitionistic transitive model} is a tuple
\[\cl M=(\dom{\cl M},\peq_\cl M,\rel_\cl M,\val\cdot_\cl M),\]
consisting of an intuitionistic transitive frame equipped with a {\em valuation,} i.e., a function $\val\cdot_\cl M \colon \mathbb P \to 2^{\dom{\cl M}}$ which is {\em persistent,} in the sense that $v \seq w\in \val p_\cl M$ implies that $v\in \val p_\cl M$.

Given an intuitionistic transitive model $\mathcal M $ and $w\in \dom{\cl M}$, we extend $\val\cdot _\cl M$ to $\lanfull$ by setting
\[\val \varphi_{\cl M} = \{w\in \dom{\cl M}:  \cl M \models\varphi (w)\}.\]
\end{definition}

It can be easily proven by induction on $\varphi \in \lanfull$ that \emph{persistence} (or \emph{monotonicity}) holds: for all $w,v \in W$, if $w \peq v$ and $ \mathcal M \models \varphi(w)$, then $ \mathcal M \models \varphi(v)$.

\section{Labelled and typed structures}\label{SecNDQ}

Our proof of the finite model property is based on various combinatorial manipulations and well-foundedness properties stemming from Kruskal's theorem~\cite{krusty}.
These techniques are best developed in terms of {\em quasimodels,} and {\em labelled structures} in general.
As we will see, such structures generalise standard models, thus allowing us to work exclusively with labelled structures going forward.

These labels will record syntactic information, including Gentzen sequents meant to record true and false formulas at a point.
It will thus be convenient to establish some conventions to allow us to employ Gentzen-style reasoning.

\begin{definition}\label{defGentzen}
If $\Gamma$ is a set of formulas, a {\em conjunction from $\Gamma$} is a formula of the form $\bigwedge \Gamma'$ and a {\em disjunction from $\Gamma$} is a formula of the form $\bigvee \Gamma'$, where $\Gamma'\subseteq \Gamma$ is finite, with the convention that $\bigwedge \varnothing =\top$ and $\bigvee\varnothing = \bot$.
Given a logic $\sf L$ (which for our purposes is either $\ikf$ or $\sf IGL$), $\Gamma \vdash_{\sf L}  \Delta$ means that there is a conjunction $\gamma$ from $\Gamma$ and a disjunction $\delta$ from $\delta$ such that $ \gamma \imp \delta $ is a theorem of $\sf L$.
We call $\vdash$ the \define{syntactic consequence} relation, and if $\Gamma\not\vdash_{\sf L}\Delta$, we say that $\Gamma$ is {\em $\Delta$-consistent.}
	We omit mention of $\Delta$ if $\Delta = \varnothing$ and may omit the subscript $\sf L$ when clear from context.
	\end{definition}
	
These conventions apply below when defining {\em types,} which are essentially non-derivable sequents.
Intuitively, types   represent the true and false formulas at a point, respectively, be it partially or completely.

\begin{definition}
Let ${\sf L} \in \{\isf,{\sf IGL}\}$, $\Sigma$ be a set of formulas, and $\Phi=(\Phi^+,\Phi^-)$ be a pair of elements of $2^\Sigma$.

We say that $\Phi$ is a {\em $\Sigma$-type (over $\sf L$)} if $ \Phi ^+$ is $\Phi^- $-consistent over $\sf L$ (i.e., if $\Phi^+ \not\vdash \Phi^-$).
We denote the set of $\Sigma$-types by $\type \Sigma$.
If, moreover, $\Phi^+\cup \Phi^- = \Sigma$, we say that $\Phi$ is {\em complete.}
We may   write {\em type} instead of {\em $\lanfull$-type.}
\end{definition}

Each point $w$ in our structures will be labelled by triples of sets of formulas $\Phi=(\Phi^+,\Phi^-,\Phi^{\notnec})$, where, intuitively, $\Phi^+$ records formulas true on $w$, $\Phi^-$ records false formulas, and $\Phi^{\notnec}$ records formulas that are false on $\rel$-accessible worlds: those formulas of the form $\nec\varphi$ which are made false `directly' by a $\rel$-successor instead of a $\sqSubset$-successor.
We likewise define other derived sets which contain useful modal information.

\begin{definition}\label{defDerived}
Let $\Phi = (\Phi^+,\Phi^-,\Phi^\notnec)$ be a triple of sets of formulas and $\Sigma\subseteq\lanfull$.
Define
\begin{align*}
\Phi^\nec & = \{\varphi:\nec\varphi\in \Phi^+\}&\Phi^\ps & =\{\varphi:\ps\varphi\in \Phi^+\} \\
\Phi^\boxdot & = \Phi^+\cap \Phi^\nec& 
\Phi^\diamonddot & = \Phi^+\cup \Phi^\ps \\
\Phi^\notnecd & = \Phi^-\cup \Phi^\notnec & \Phi^{\overline \ps} & =  \{\varphi: \ps\varphi\in \Phi^-\} \\
\Phi{{\upharpoonright}} \Sigma & = (\Phi^+ \cap \Sigma, \Phi^- \cap \Sigma, \Phi^{\notnec}\cap \Sigma ) .
\end{align*}
Fix ${\sf L} \in \{\ikf,{\sf IGL}\}$. 
We say that $\Phi$ is a {\em $\Sigma$-label (for $\sf L$)} if $( \Phi ^+, \Phi^- )$ is a $\Sigma$-type and $\Phi^\Box \cap \Phi^{\notnec} =\varnothing $.
If, moreover, $( \Phi ^+, \Phi^- )$ is a complete $\Sigma$-type, we say that $\Phi$ is {\em $\Sigma$-complete.}

We denote the set of $\Sigma$-labels for $\sf L$ by $\lbl\Sigma$.
We write $\lbl\infty$ instead of $\lbl\lanfull$ and {\em label} instead of {\em $\lanfull$-label}.
\end{definition}

Note that $\lbl\Sigma$ in principle depends on $\sf L$, but we can for the most part assume that ${\sf L} \in \{\ikf,{\sf IGL}\}$ is fixed, so there is no point in indicating it in the notation.
In the sequel, we only mention the choice of $\sf L$ when it is relevant for the result at hand.

Note that $\Sigma\subseteq\Delta$ implies that $\lbl\Sigma\subseteq\lbl\Delta$.
In particular, $\lbl\infty$ always contains $\lbl\Sigma$, although being $\Sigma$-complete does not imply being $\lanfull$-complete.
The idea is that
$\Sigma$-labels extend $\Sigma$-types with additional information, and we may tacitly view a $\Sigma$-type $\Phi$ as a $\Sigma$-label by setting $\Phi^\notnec =\varnothing$.

Labels themselves can be endowed with various binary relations, which will be useful later in order to incorporate them into larger structures.
Below, if $\Phi,\Psi,\Sigma$ are sets of formulas, write $\Phi \subseteq _\Sigma \Psi$ if $\Phi \cap\Sigma \subseteq  \Psi$  and define $\supseteq_\Sigma$ analogously.

\begin{definition}
For $\Sigma\subseteq\lanfull$ and $\Phi,\Psi\in \lbl\infty$, define
\begin{enumerate}[label=(\alph*)]

\item  $\Phi \peqT \Psi$ if $\Phi^+   \subseteq_\Sigma \Psi ^+$ and $\Phi^-  \supseteq_\Sigma \Psi ^ -  $;

\item  $\Phi \sqsT \Psi$ if $\Phi^\nec   \subseteq _\Sigma \Psi ^\boxdot$, $\Phi^\ps  \supseteq_\Sigma \Psi^\diamonddot $, and $\Phi^{\notnec}\supseteq_\Sigma   \Psi^{\notnecd} $;

\item   $\Phi \subT \Psi $ if $\Phi ^+\subT \Psi^+$, $\Phi ^-\subT \Psi^-$, and $\Phi ^\notnec\subT \Psi^\notnec$.

\end{enumerate}
\end{definition}

If $\Phi$ is a  label, we may write $\varphi\in \Phi$ instead of $\varphi\in \Phi^+$, $\overline \varphi\in \Phi$ instead of $\varphi\in \Phi^-$, and $\ps\overline \varphi\in \Phi$ instead of $\varphi\in \Phi^{\notnec}$.
We can think of $\overline \varphi$ as a `pseudo-negation' of $\varphi$, although this notation will only be used in the metalanguage.

Often (but not always), we will want $\Sigma$ to be finite.
To this end, given $\Delta\subseteq \landi$, we write $\Sigma \Subset\Delta$ if $\Sigma$ is a finite subset of $\Delta$ closed under subformulas.

\begin{definition}
Let $\Sigma\subseteq \lanfull$.
A {\em $\Sigma $-frame} is a triple $\fA=(\dom{\fA},\rel_\fA,\ell_\fA)$, consisting of a transitive frame equipped with a function $\ell_\fA\colon \dom{\fA}\to \lbl\Sigma$ which is monotone on $\rel$, in the sense that $w\rel_\fA v$ implies that $\ell_\fA(w) \sqsT \ell_\fA(v)$.
If $\rel_\fA$ is converse well founded, then $\fA$ is also a {\em $\sf GL$-frame.}

If $|\fA|$ is finite and {\em tree-like,} in the sense that the predecessors of any node are totally ordered by $\rrel_\fA$, then $\fA$ is a {\em $\Sigma$-sprout.}
If $\ell_\fA(w)  $ is complete for all $w\in\dom{\fA}$, we say that $\fA $ is {\em completely labelled}.

If $\ell_\fA(w) = (\Phi^+,\Phi^-,\Phi^{\notnec})$, we write $\ell_\fA^\circ(w) :=  \Phi^\circ $ for $\circ\in \{+,-,\notnec\}$.
\end{definition}

As before, we may drop subindexes when clear from context.
If $\fA$ is a $\Sigma$-sprout, the {\em height} of $\fA$, ${\rm hgt}(\fA)$, is the maximum $n$ so that there exists a sequence $w_0\srel_\fA w_1 \srel_\fA \ldots \srel_\fA w_n $ and its {\em width,} ${\rm wdt}(\fA)$, is the maximum $n$ such that there are $w\in\dom{\fA}$ and $V=\{v_1,\ldots,v_n\}$ with $w\rel^1_\fA v_i$ and $v_i \not \rrel_\fA v_j$ whenever $i\neq j$.
Such a maximal set $V$ is a {\em representative set of successors} for $w$.
The {\em density} of $\fA$, ${\rm dns}(\fA)$, is the supremum of cardinalities of clusters of $\fA$.

We will use ${{\upharpoonright}}$ to indicate domain restrictions in the standard way, so that for example if $\fA$ is a $\Sigma$-sprout and $X \subseteq \dom{\fA}$, then ${\rel_\fA} {{\upharpoonright}} X = {{\rel_\fA}}\cap ( X \times X) $, and $\fA {{\upharpoonright}} X = (X ,{\rel_\fA} {{\upharpoonright}} X, \ell_\fA {{\upharpoonright}} X)$.
If $R$ is a binary relation and $a$ is in its domain, then, as usual, $R(a) = \{b: a\mathrel R b\}$. In particular, for $w\in \dom{\fA}$, ${\rrel_\fA}(w) = \{v\in\dom{\fA}: w\rrel_\fA v\}$.
Then, we define  $\restr{\fA}{w} := \fA {{\upharpoonright}} {\rrel_\fA}(w)$.
The {\em depth} of $w\in \dom{\fA}$, denoted ${\rm dpt}_\fA(w)$, is defined to be the height of $\restr{\fA}{w}$.

In the general case, $\Sigma$-frames contain positive and negative information about their points, but on occasion we are interested only on positive information.
Following Stra{\ss}burger~\cite{Strassburger} and Lyon~\cite{Lyon}, we also consider sprouts that have a singular point with negative information.

\begin{definition}\label{defPosPart}
Let $\Sigma\subseteq \lanfull$ and $\fA$ be {\em $\Sigma $-frame} for either $\ikf$ or $\sf IGL$.
\begin{enumerate}

\item 
We define a new $\Sigma$-frame $\fA^+$ which is identical to $\fA$ except that each   $w \in\dom\fA$ is now labelled by $(\ell^+_\fA(w),\varnothing,\varnothing)$.
If $\fA=\fA^+$, we say that $\fA$ is {\em positive}.

\item
If $w_0\in \dom\fA$, we similarly define $\fA^+(w_0)$ to be identical to $\fA$, except that $w_0$ is labelled by $(\ell^+_\fA(w_0),\ell^-_\fA(w_0),\varnothing)$ and any other $w\in\dom\fA$ is labelled by $(\ell^+_\fA(w),\varnothing,\varnothing)$ as before.
We say that $\fA$ is {\em singular} if there is $w_0\in \dom{\fA}$ such that $\fA=\fA^+(w_0)$.

\end{enumerate}
\end{definition}

\begin{remark}
Singular sprouts are needed for Stra{\ss}burger formulas to be defined (see Definition~\ref{defYF}) and these formulas are used only in the completeness proof of $\sf IGL$.
The finite model property for $\ikf$ can be carried out without them.
I will comment further when they are used.
\end{remark}

As we have defined them, $\Sigma$-sprouts are only partially required to satisfy the semantic clauses for modalities or implication.
For example, they may not contain the required witnesses for a formula $\ps\psi$.
We refer to such situations as defects, as defined below.

\begin{definition}
A {\em potential defect} of a $\Sigma$-frame $\fA$ is a pair $\delta= (v,\varphi)$, where $v\in \dom{\fA}$, $\varphi \in \ell (v) $, and one of the following occurs.
\begin{enumerate}[label=(\alph*)]

\item $\varphi = \ps\psi$ or $\varphi=\ps\overline\psi$, in which case, we say that $ \delta $ is {\em modal,}

\item $\varphi = \overline{\psi\to \theta} $ and $ \psi \notin \ell  (v)$ or $\varphi = \overline{\nec \psi}$ and $ \ps \overline \psi \notin \ell (v)$, in which case we say that $\delta$ is {\em intuitionistic.}

\end{enumerate}

\end{definition}

We cannot expect to treat intuitionistic defects with the structures we have introduced so far, but modal defects can indeed can be resolved within a $\Sigma$-frame.

\begin{definition}\label{frame}    
Let $\fA$ be a $\Sigma$-frame and $\delta=(v,\ps \psi)$ be a potential modal defect of $\fA$.
We say that $\delta$ is {\em resolved} if there is $v'\ler_\fA v$ such that $\psi \in\ell_\fA(v')$.

A $\Sigma$-frame $\fA$ is {\em modal} if it is completely labelled and all of its modal defects are resolved.
A modal $\Sigma$-sprout is a {\em $\Sigma$-tree.}
\end{definition}

Thus, the machinery we have defined so far allows us to resolve modal defects.
However, we will need additional structure to deal with intuitionistic defects: in order to obtain a model from a modal $\Sigma$-frame, we need to equip it with an intuitionistic preorder.

\begin{definition}
Let $\fA$ be a  labelled frame and $\Sigma\subseteq \lanfull$.
A preorder ${\peq} \subseteq \dom{\fA}\times\dom{\fA}$ is a {\em $\Sigma$-intuitionistic order} if it is expansive over $\sbs_{\fA}$ and, whenever $ w\peq v$, it follows that $\ell(w)  \peqT \ell(v) $.
\end{definition}

Note that we in general allow $\peq$ to be a preorder, simply because our constructions may not preserve antisymmetry.
We often simply write {\em intuitionistic relation} when $\Sigma$ is either clear from context or not too relevant for the discussion at hand.

As their name implies, intuitionistic relations should resolve intuitionistic defects.

\begin{definition}
Let $\fA$ be a modal $\Sigma$-frame.
An intuitionistic relation $\peq$ on $\fA$ {\em resolves} a defect $(w,\varphi)$ of $\fA$ if there is $v \seq w$ such that $\overline \varphi \in \ell  (v)$ but  $(v,\varphi)$ is not a defect.
We say that $v$ is the {\em resolution link.}

If $\varphi$ is of the form $\psi\to\theta$, then $v$ is also the {\em point of resolution}.
Otherwise, $\varphi$ is of the form $\nec \psi $ and a  {\em point of resolution} is any $v' \ler v$ such that  $ \overline\psi \in\ell(v)$.
\end{definition}

By equipping modal $\Sigma$-frames with intuitionistic relations, we obtain quasimodels.
 
\begin{definition}
Fix $\Sigma$ closed under subformulas.
A {\em $\Sigma$-quasimodel} is a tuple $\cl Q=(\dom{\cl Q},\peq_\cl Q,\sbs_\cl Q,\ell_\cl Q)$, consisting of a modal $\Sigma$-tree equipped with an intuitionistic relation.
If $\peq_\cl Q$ resolves all intuitionistic defects, then $\cl Q$ is a {\em $\Sigma$-model.}
\end{definition}
 
As we are interested in the validity problem, falsifiability will be more relevant to us than satisfiability.
As usual, a model $\cl M$ {\em falsifies} $\varphi\in\lanfull$ if $\val\varphi_{\cl M}\neq \dom{\cl M}$.
Likewise, a labelled structure $\fA$ {\em falsifies} $\varphi $ if there is $w\in\dom\fA$ such that $\overline\varphi\in\ell(w)$.
Below, if $\mathfrak C$ is a class of expansive transitive frames, then a model or labelled frame is {\em from $\mathfrak C$} if its underlying frame belongs to $\mathfrak C$.

\begin{proposition}\label{propQMSound}
Given $\varphi\in\lanfull$, ${\sf L} \in \{\ikf,{\sf IGL}\}$, and a class of expansive transitive frames $\mathfrak C$, the following are equivalent:
\begin{enumerate}

\item There is a model $\cl M$ from $\mathfrak C$ falsifying $\varphi$.

\item There is an $\lanfull$-model $\cl Q$ from $\mathfrak C$ falsifying $\varphi$.

\item There are $\Sigma\Subset\lanfull$ and a $\Sigma$-model $\cl Q$ from $\mathfrak C$ falsifying $\varphi$.

\end{enumerate}
In all cases, $\cl Q$ is finite if and only if $\cl M$ is.
\end{proposition}

\begin{proof}
Given a model $\cl M$ and a set $\Sigma$ closed under subformulas, we can define $\ell_\cl M(w) = \{\varphi\in\Sigma:(\cl M,w)\models \varphi\}$ and easily check that $(\dom{\cl M},\peq_\cl M,\sbs_\cl M,\ell_\cl M)$ is a $\Sigma$-model.
By setting $\Sigma=\lanfull$, we obtain the second item from the first, and by letting $\Sigma$ be the set of subformulas of $\varphi$, we obtain the third.

Conversely, if $\Sigma$ is closed under subformulas, given a $\Sigma$-model $\cl Q$, we may define $\val p_\cl Q =\{w\in\dom{\cl Q}:p\in\ell(w)\}$. 
Then, for $\varphi\in \Sigma$, a standard induction shows that $\varphi\in\ell(w)$ if and only $ \cl Q \models\varphi(w) $, so the first item is obtained from the second when $\Sigma=\lanfull$ and from the third when $\Sigma$ is finite.
\end{proof}

Thus, in the sequel, we will work exclusively with labelled structures instead of models.
In particular, we will refer to $\lanfull$-models simply as {\em models.}

We have defined various binary relations on $\Sigma$-labels (e.g.~$\peqT $), which have helped us incorporate them into larger structures.
It will be useful to extend such orders to labelled structures, as this will allow us to perform various operations with them, such as amalgamating them into even larger structures.
Let us begin by defining an analogue of the intuitionistic relation.

\begin{definition}
Let $\fA $ and $\fB$ be $\Sigma$-frames.
We write $\fA \peqT \fB$ if there is $R\colon \fA\simu \fB $ such that, whenever $w\mathrel R v$, it follows that $\ell_\fA(w) \peqT \ell_\fB(v)$.
Given $w\in \dom{\fA}$ and $v\in \dom{\fB}$, if such an $R$ exists with $w\mathrel R v$, we write $(\fA,w) \peqT (\fB,v)$.
\end{definition}

It can easily be checked using Proposition~\ref{propExpandingClos} that $\peqT$ defines a quasiorder on the class of $\Sigma$-labelled frames.
One can imagine building a quasimodel by amalgamating states where, whenever $\fA$ is a state with some defect $\delta = (w,\varphi)$, we graft a sprout $\fB $ with some world $v$ such that $(\fA,w) \peqT (\fB,v)$ and $v$ resolves $\delta$.
Our strategy will fall along these lines, with the caveat that the `grafting' must be done very carefully to ensure termination.

Next, we turn to the appropriate notion of `substructure' for labelled frames, given by {\em simulations.}
Below, recall that if $\Phi,\Psi\in \lbl\infty$, $\Phi \subT \Psi$ means that $\Phi^\circ\cap\Sigma \subseteq  \Psi ^\circ$ for ${\circ} \in \{+,-,\notnec\}$.

\begin{definition}
Let $\Sigma \subseteq  \landi$  and $\cl X$, $\cl Y$ be labelled frames.
A relation
${\simrel} \colon  {\cl X} \simu {\cl Y} $
is a {\em $\Sigma$-simulation} if, whenever $x\simrel y$, it follows that $\ell_{\cl X} (x) \subT  \ell_{\cl Y}(y)$.
If $E$ is a function, we say that $E$ is a {\em homomorphism.}

 If there exists a simulation $ E \colon \cl X \simu \cl Y$, we write $ \cl X  \simu_\Sigma  \cl Y $.
Given $x\in \dom{\cl X}$ and $y\in {\cl Y}$, if $E$ can moreover be chosen such that $x\mathrel E y$, we write $(\cl X , x) \simu_\Sigma (\cl Y, y)$.

 We write $ \cl X  \subT  \cl Y $ if $ \cl X  \simu_\Sigma  \cl Y $ via a homomorphism $h\colon\dom{\cl X}\to\dom{\cl Y} $ and, as above, if $y=h(x)$, we write $(\cl X , x) \subT (\cl Y, y)$.
\end{definition}

In view of Proposition~\ref{propExpandingClos}, $\subT$ defines a quasiorder on the class of $\Sigma$-labelled frames.
As we will see, this quasiorder will turn out to be rather interesting and useful.
But first, let us review the standard canonical model, in the light of labelled structures.

\section{The Canonical Model}\label{secCanonical}

Our finite models will be extracted from the canonical model for $\ikf$.
In our presentation, it would be more precise to describe it as the canonical {\em $\lanfull$-model,} but in view of Proposition~\ref{propQMSound}, the difference is inessential.
The construction is essentially that of Fischer Servi~\cite{servi1977modal,servi1984axiomatizations}.

Fix ${\sf L} \in \{\ikf,{\sf IGL}\}$.
A set $\Phi \subseteq \lanfull$ is a {\em theory over $\sf L$} if it is deductively closed in $\sf L$, and {\em prime} if, whenever $\varphi\vee\psi\in \Phi$, it follows that $\varphi\in \Phi$ or $\psi\in \Phi$.
We say that $\Psi$ \define{extends} $\Phi$ if $\Phi \subseteq\Psi $.
In the intuitionistic setting, the classical Lindenbaum lemma produces prime sets.
Below, we follow the conventions of Definition~\ref{defGentzen}; in particular, $\Phi $ is $\Xi$-consistent if $\Phi\not\vdash\Xi$.
	
\begin{lemma}[Lindenbaum lemma~\cite{servi1977modal,servi1984axiomatizations}]\label{lem:lindenbaum}
Let $\Sigma \subseteq \lanfull$ be closed under subformulas.
Any $\Xi$\term{-consistent} set $\Phi\subseteq\Sigma$ of formulas can be extended to a prime $\Xi$-\term{consistent} theory $ \Phi _* $.
\end{lemma}	

Prime theories can be regarded as $\lanfull$-types by appending them with their complement.

\begin{lemma}\label{lemCompleteType}
Let $\Phi$ be a prime theory.
Then, $(\Phi,\lanfull \setminus \Phi)$ is a complete $\lanfull$-type.
\end{lemma}

\begin{proof}
It suffices to check that $\Phi$ is $(\lanfull \setminus \Phi)$-consistent, since, in this case, the pair is clearly complete.
Otherwise, let $  \bigvee_i \psi_i $ be a finite disjunction of formulas of 
$ \lanfull \setminus \Phi $ such that $\Phi\vdash  \bigvee_i \psi_i$.
Since $\Phi$ is prime, there is $i$ such that $\Phi\vdash \psi_i$.
But $\Phi$ is a theory, i.e., deductively closed, and hence $\psi_i \in \Phi$, a contradiction.\end{proof}

Now that we know that prime theories can be regarded as types, we can define $\Phi^\nec$, $\Phi^\ps$, $\Phi^\boxdot$, and $\Phi^\diamonddot$ as in Definition~\ref{defDerived}.

\begin{definition}
Let ${\sf L} \in \{\ikf,{\sf IGL}\}$.
We define the {\em canonical model} for $\sf L$ to be the $\lanfull$-labelled structure $\mathcal M_{\mathrm c} =( W _{\mathrm c} ,{\peq_{\mathrm c} },{\rel_{\mathrm c} }, \ell _{\mathrm c} )$, where
\begin{enumerate}[label=\alph*)]
	\item $W_{\mathrm c} $ is the set of prime theories over $\sf L$,
	
\item 	${\peq_{\mathrm c}} \subseteq W_{\mathrm c} \times W_{\mathrm c} $ is defined by $\Phi \peq_{\mathrm c}  \Psi$ if and only if $\Phi  \subseteq \Psi  $,

\item ${\rel_{\mathrm c}}  \subseteq W_{\mathrm c} \times W_{\mathrm c} $ is defined by $\Phi \rel_{\mathrm c}  \Psi$ if and only if $  \bfrm\Phi \subseteq  \Psi $ and $    \Phi^\ps \supseteq\Psi $, and

\item $\ell_{\rm c} (\Phi) = (\Phi^+,\Phi^-,\Phi^\notnec) $, where
\[\Phi^\notnec =\{\varphi:\exists \Psi \ler_{\rm c} \Phi: \varphi\in \Phi^-\}.\]
\end{enumerate}
\end{definition}

As mentioned, this is essentially the canonical model construction of Fischer Servi~\cite{servi1977modal,servi1984axiomatizations}, who (adjusted for our terminology) proved the following.

\begin{proposition}\label{lem:truth-lemma:regular}
$\cl M_{\rm c}$ is an $\lanfull$-model and $\peq_{\rm c}$ is a partial order.
\end{proposition}

Recall that we generally allow models where $\peq$ is a preorder, so the claim that $\peq_{\rm c}$ is a partial order is not redundant here.
Note that, when identifying theories with their associated types, $\peq_{\rm c}$ and $\peq_\lanfull$ coincide as we have defined them, but, a priori, $\rel_{\rm c}$ and $\rel_\lanfull$ may be different.
However, they do coincide modulo the $\sf IK4$ axioms.

\begin{lemma}\label{lemmCanRel}
For all $\Phi,\Psi\in W_{\rm c}$, $\Phi \rel_{\rm c} \Psi$ if and only if $\ell_{\rm c}(\Phi) \rel_{\lanfull} \ell_{\rm c}(\Psi)$.
\end{lemma}

\begin{proof}
That $\Phi \rel_{\rm c} \Psi$ follows tautologically from $\Phi\rel_\lanfull \Psi$, since if $\nec\varphi\in \Phi$ then from $\varphi,\nec\varphi\in \Phi$ we in particular deduce $\varphi\in \Psi$, and likewise if $ \varphi\in \Psi $ then this is sufficient to obtain $ \varphi\in \Psi^\diamonddot$  and thus $\ps\varphi\in\Phi$.

Conversely, assume that $\Phi \rel_{\rm c} \Psi$. Let $\nec\varphi \in \ell_{\rm c}(\Phi)$, i.e., $\nec\varphi\in \Phi$.
Then, $\varphi\in \Psi$ and using \ref{ax:trans:box} and modus ponens, $\nec\nec\varphi\in \Phi$, so $\nec\varphi\in \Psi$ and thus $\varphi\in \Psi^{\boxdot}$.
We can reason analogously to see that $\Phi^\ps \supseteq \Psi^{\diamonddot} $, and thus $\Phi \rel_\lanfull \Psi$.
\end{proof}

Note that we have defined the canonical model as a labelled structure rather than a standard model, mainly because this will be the more useful way to think about it.
However, we can as usual define $ \val p _{\mathrm c} := \{\Phi\in W_{\mathrm c} \mid p\in \Phi\}$ to view it as a proper model.

Recall that we defined $\llsim$ to be the transitive closure of $\peq\cup \rel$.
In the case of the canonical model, this relation has a simple syntactic characterisation.

\begin{lemma}\label{lemmBoxLL}
Given $\Theta,\Phi \in W_{\rm c}$, $ \Theta \llsim_{\rm c} \Phi $ if and only if $\Theta^\boxdot \subseteq \Phi^\boxdot$.
\end{lemma}

\begin{proof}
Let us omit the subindex $\rm c$.
First assume that $ \Theta \llsim \Phi $.
If $\Theta\peq \Phi $, then clearly $\Theta^\boxdot \subseteq \Phi^\boxdot$.
Otherwise, there is $\Psi$ such that $\Theta\peq\Psi\rel \Phi$.
Suppose that $\varphi\in \Theta^\boxdot $, i.e., $\varphi,\Box\varphi \in \Theta$.
By the definitions of the canonical relations, $\nec\varphi\in \Psi$ and hence by Lemma~\ref{lemmCanRel}, $\varphi,\nec\varphi\in \Phi$.

Now, assume that $\Theta^\boxdot \subseteq \Phi^\boxdot$.
We may moreover assume that $\Theta  \not\peq \Phi $, so that there is $\delta \in \Theta ^+\cap\Phi^- $.
We must find $\Psi$ such that $\Theta\peq\Psi\rel \Phi$; in other words, $\Theta^+\subseteq \Psi^+$, $   \Phi^+\subseteq \Psi^\ps   $, and $   \Psi^\nec  \subseteq \Phi^+ $ (or, $ \Psi ^ \notnec  \supseteq \Phi^- $).

It thus suffices to show that $\Theta ^+ \cup \ps \Phi^+ $ is $\nec \Phi ^ -$-consistent.
Otherwise, there are finite conjunctions $\theta$ from $\Theta^+$ and $\varphi$ from  $\Phi^+ $ and a finite disjunction $\psi$ from $\Phi^-$ such that $ \theta\wedge \ps \varphi \vdash \nec \psi$ (where we are using $\ps\bigwedge_i\varphi_i \vdash \bigwedge_i\ps \varphi_i$ and $\bigvee_i\nec\psi_i\vdash \nec\bigvee_i\psi_i$, which are easily shown to be derivable) and, hence, $\theta\vdash \ps \varphi \to \nec\psi$.
By \ref{ax:fs}, $\theta \vdash \nec(\varphi \to\psi)$ and, therefore, $\nec(\varphi \to\psi )\in\Theta^+$, so that $\boxdot (\varphi\to\psi\vee\delta) \in \Theta^+ $.
By assumption, $\boxdot (\varphi\to\psi\vee\delta) \in \Phi^+ $, which in particular implies that $ \varphi\to\psi\vee\delta  \in \Phi^+ $.
But $\varphi\in \Phi^+$ yet neither $\psi$ nor $\delta$ belong to $\Phi^+ $, a contradiction.
\end{proof}

\section{Finality}

In classical modal logic, a maximal, or {\em Final}, world, is one satisfying a given formula $\varphi$ which has no strict successors also satsifying $\varphi$.
Fine~\cite{Fin74c} showed that Final worlds can be used to extract finite models from the canonical model of a (classical) transitive logic.
We wish to extend this notion to the intuitionistic setting, where we immediately note that models have two primitive relations, $\peq$ and $\rel$, which leads to at least two notions of Finality.
We can also imagine worlds that are Final with respect to both relations simultaneously, which amounts to making them Final with respect to $\llsim $.
All of these will be interesting to us, so we define Finality in a general setting.
Below, recall that a {\em $\Sigma$-sprout} is a $\Sigma$-labelled tree-like frame, where the labels are not necessarily complete, i.e., some formulas may be left undecided and some modal defects may be unresolved.

\begin{definition}\label{defFinal}
Let $\Sigma \subseteq \Delta \subseteq\lanfull$, $\fA$ be a $\Sigma$-sprout, $\fB$ be a $\Delta$-frame, $w\in \dom{\fA}$ and $v\in \dom{\fB}$.
Let $\unlhd$ be a preorder on $\dom{\fB}$.
We say that $v$ is {\em $\unlhd$-Final for $(\fA,w)$} if 
\begin{enumerate}[label=(\alph*)]

\item there is $v'$ such that $v \unlhd v'$ and $(\fA,w) \subT (\fB,v') $, and

\item whenever $v'$ is such that $v \unlhd v'$ and $(\fA,w) \subT (\fB,v') $, it follows that $v' \unlhd v$.

\end{enumerate}

Similarly, $v$ is {\em simulably $\unlhd$-Final for $(\fA,w)$} if 
\begin{enumerate}[label=(\alph*')]

\item there is $v'$ such that $v \unlhd v'$ and $(\fA,w) \simu_\Sigma (\fB,v') $, and

\item whenever $v'$ is such that $v \unlhd v'$ and $(\fA,w) \simu_\Sigma (\fB,v') $, it follows that $v' \unlhd v$.

\end{enumerate}
A set $C\subseteq \dom\fB$ is {\em $\unlhd$-Final} if it contains at least one $\unlhd$-Final element for $(\fA,w)$.
We say that a world or a set is {\em $\unlhd$-Final for $\Sigma\subseteq \lanfull$} if it is $\unlhd$-Final for some pair $(\fA,w)$, where $\fA$ is a $\Sigma$-sprout, and {\em $\unlhd$-Final} if it is $\unlhd$-Final for some $\Sigma$. 
A world or set is {\em $\unlhd$-Final for $ \fA $} if it is $\unlhd$-Final for $(\fA,r)$, where $r$ is a root of $\fA$.
We adopt the analogous conventions for simulable $\unlhd$-finality.
\end{definition}

If ${\unlhd} \in \{\rel,\llsim\}$, it is easy to check that if a $\unlhd$-cluster $C$ is $\unlhd$-Final, then every element of $C$ is $\unlhd$-Final, and moreover that if a world or cluster is $\unlhd$-Final for $\fA$, then it is $\unlhd$-Final for $(\fA,r)$ for {\em any} root $r$.

Our notion of Finality generalises the classical one, since individual formulas can be viewed as sprouts; to be precise, a formula $\varphi$ is identified with a sprout $\fA_\varphi$ which has only one irreflexive world labelled by $(\{\varphi\},\varnothing,\varnothing)$.
Likewise, a pseudo-negation $\overline\varphi$ may be identified with the sprout $\fA_{\overline \varphi}$ consisting of an irreflexive singleton  labelled by $(\varnothing,\{\varphi\},\varnothing)$.
In such cases, we will talk about $R$-Finality for $\varphi$ or for $\overline\varphi$, instead of $R$-Finality for $\fA_\varphi$ or for $\fA_{\overline\varphi}$.

We are particularly interested in the case where $\fB$ is the canonical model.
Below, we write $\Theta$ for $(\cl M_{\rm c},\Theta)$, as per our convention.
Note that $\Sigma$ need not be finite.

\begin{proposition}\label{propFinal}
Let $ \unlhd \in \{\peq_{\rm c},\rrel_{\rm c},\rler_{\rm c},\llsim_{\rm c}\}$, $\fA $ be a $\Sigma$-sprout, and $r$ a root of $\fA$.
\begin{enumerate}[label=(\roman*)]

\item\label{itFinalHom} Suppose that $\Theta \in W_{\rm c}$ is such that $(\fA,r)\subT \Theta  $.
Then, there is $\Theta^* $ such that $\Theta \unlhd \Theta^*$ and $\Theta^*$ is $\unlhd $-Final for $ \fA $.

\item\label{itFinalSim} Suppose that $\Theta \in W_{\rm c}$ is such that $(\fA,r)\simu_\Sigma \Theta  $.
Then, there is $\Theta^* $ such that $\Theta \mathrel R \Theta^*$ and $\Theta^*$ is simulably $\unlhd $-Final for $ \fA $.

\end{enumerate}
\end{proposition}

\begin{proof}
Both claims follow from compactness for classical first order logic, using the standard translation.
Let $\tt IK4$ be a first order formula which holds on a model $\cl M$ if and only if $\cl M$ is based on an intuitionistic transitive frame and $\forall x\peq y(p(x)\to p(y))$ for all atoms $p$ occurring in $\Sigma$.

Let us define
\begin{align*}
\Xi_{\peq}  & = \{\varphi(x):\varphi\in\lanfull\},\\
\Xi_{\rrel}  & = \{ \boxdot \varphi(x) :\varphi\in\lanfull\} \cup \{ \neg {\diamonddot\varphi(x)}:\varphi\in\lanfull\}\\
\Xi_{\rler}  & = \{ \neg {\boxdot} \varphi(x)  :\varphi\in\lanfull\} \cup \{  {\diamonddot\varphi(x)}:\varphi\in\lanfull\}\\
\Xi_{\llsim}  & = \{\boxdot \varphi(x):\varphi\in\lanfull\} .\\ 
\end{align*}
Note that the negations are classical, i.e.~they are applied {\em after} the standard translation into classical first order logic.

For the first claim, we regard each $w\in \dom{\fA} $, including the root $r$, as a new constant.
Let ${\tt Acc}(\Theta) $ be the set of all formulas of the form $ \varphi(r)  $  such that $\varphi(x) \in \Xi_\unlhd  $ and $\cl M_{\rm c} \models \varphi(\Theta)$.
For each $w\in \fA$, we add  all instances of  $ \varphi(w)$, where $\varphi\in \ell^+ (w)$, all instances of $  \neg \varphi(w)$ where $\varphi\in \ell^- (w)$, all instances of $  \exists x \ler  w \ \neg \varphi( x )$ with $\varphi\in \ell^\notnec (w)$, as well as all instances of $  w\rel v$ with $w\rel_\fA v$.
Let the collection of these formulas be ${\tt Hom}(\fA)$.

Let $\Gamma = {\tt IK4}\cup {\tt Acc}(\Theta) \cup {\tt Hom}(\fA)$ and enumerate the elements of $\Xi_\unlhd  $ by $\{\xi_i\}_{i<\omega}$.
Define $\Delta_i$ inductively so that $\Delta_{0}=\varnothing$ and $\Delta_{i+1} = \Delta_i\cup \{   \xi_i(r)\}$ if $\Gamma \cup \Delta_i\cup \{  \xi_i(r)\}$ is consistent, otherwise, $\Delta_{i+1} = \Delta_i$.

Since $(\fA,r)\subT \Theta$, there is a homomorphism $h \colon \dom{\fA}\to W_{\rm c}$ with $h(r) =\Theta$.
Interpreting $w$ as $h(w)$ shows that $\Gamma = \Gamma\cup \Delta_0 $ is consistent as a set of first order formulas, hence an easy induction shows that $\Gamma \cup \Delta_i$ is consistent for all $i$ and thus compactness for first order logic shows that $\Gamma\cup\Delta_\omega$ is consistent, where $\Delta_\omega:=\bigcup_{i<\omega}\Delta_i$.

Letting $\cl N$ be any model of $\Gamma\cup\Delta_\omega$ and $ w \in\dom{\fA}$, we take $h'(w):= \{\varphi\in\lanfull: \cl N \models \varphi(w)\}$.
Each $h' (w)$ is an element of the canonical model, and hence $h' \colon \dom{\cl A} \to W_{\rm c} $.
It is not hard to check that $h$ is a homomorphism and, by our choice of $\Delta_\omega$, $h'(r)$ is $\unlhd $-Final.
Finally, by ${\tt Acc}(\Theta)$, $\Theta \unlhd  \Phi_r$, as desired.

The second claim is proven similarly, but we instead regard each $w\in \dom{\fA} $, including the root $r$, as a new predicate $w(x)$, whose intended interpretation is that $ w \simu_\Sigma x$.
For each $w\in \dom{\fA} $, we add $\exists x \ w(x)$, all instances of  $\forall x(w(x) \to \varphi(x))$ where $\varphi\in \ell^+ (w)$, all instances of $\forall x(w(x) \to \neg \varphi(x))$ where $\varphi\in \ell^- (w)$, and all instances of $\forall x(w(x) \to \exists y \ler  x \ \neg \varphi(y))$ with $\varphi\in \ell^\notnec (w)$.
We also add all instances of $\forall x (w(x) \to \exists y \  v(y) \wedge x\rel y)$ with $w\rel_\fA v$ and $\forall x (w(x) \to \exists y  \ v(y) \wedge x\ler y)$ with $w\ler_\fA v$.
Let the collection of these formulas be ${\tt Sim}(\fA)$.

Let $\Gamma = {\tt IK4}\cup {\tt Acc}(\Theta) \cup {\tt Sim}(\fA)$ and define $\Delta_i$ inductively so that $\Delta_{0}=\varnothing$ and $\Delta_{i+1} = \Delta_i\cup \{\forall x(r(x) \to \xi_i(x))\}$ if $\Gamma \cup \Delta_i\cup \{\forall x(r(x) \to \xi_i(x))\}$ is consistent, otherwise, $\Delta_{i+1} = \Delta_i$.
Reasoning as above, $\Gamma\cup\Delta_\omega$ is consistent.
Letting $\cl N$ be any model of $\Gamma\cup\Delta_\omega$ and $x\in\dom{\cl N}$, we take $\Phi_x:=\{\varphi\in\lanfull: \cl N \models \varphi(x)\}$.
Then, for $w\in  \dom{\fA}$, set $E'(w) = \{\Phi_x : \cl N\models w(x) \}$.
It is not hard to check that $E'$ is a simulation and thus, by the same reasoning as above, $\Theta \unlhd  \Phi_r$ is simulably $\unlhd $-Final, as desired.
\end{proof}

\section{Strict Finality}\label{secStrict}

The notion of Finality we have discussed so far is mostly suitable for $\sf IK4$, since for example a $\rrel$-Final point may belong to a $\rrel$-cluster.
For example,  point that is Final for $\nec (\ps p\wedge \ps \neg p)$ necessarily belongs to a cluster with at least two points.
However, $\sf IGL$ does not admit such clusters and, indeed, this formula cannot be satisfied on an $\sf IGL$ model.
In this context, we thus need to work with a stronger notion of Finality we call {\em strict Finality.}
This section can be safely skipped by readers only interested in $\ikf$.

\begin{definition}
Let $\Sigma \subseteq \Delta$ be closed under subformulas, $\fA$ be a $\Sigma$-frame, $\fB$ be a $\Delta$-frame, $w\in \dom{\fA}$ and $v\in \dom{\fB}$.
Let $\lhd$ be a transitive relation on $\dom{\fB}$.
We say that $v$ is {\em strictly $\lhd$-Final for $(\fA,w)$} if 
\begin{enumerate}

\item $(\fA,w) \subT (\fB,v') $, and

\item there is no $v'$ such that $v \rhd v'$ and $(\fA,w) \subT (\fB,v') $.

\end{enumerate}
\end{definition}

We adopt the same conventions as in Definition~\ref{defFinal} regarding e.g.~strict Finality of clusters.
For the sake of brevity, we exceptionally write {\em strictly $\ll_{\rm c}$-Final} instead of {\em strictly $({\prec_{\rm c} }\cup {\sqSubset_{\rm c}})$-Final.}
To justify and clarify this convention, recall that we defined ${\llsim_{\rm c}} = ({\peq_{\rm c} }\cup {\sqSubset_{\rm c}})$ and we defined $\ll_{\rm c}$ to be the irreflexive part of $ \llsim_{\rm c}$.
Thus it is possible for $\Phi$ to be strictly $\ll_{\rm c}$-Final according to the definition, yet have $\Phi \sqSubset_{\rm c} \Phi$; such a $\Phi$ would not be useful in an $\sf IGL$ model, as such models must be irreflexive both with respect to $\rel$ and to $\sqSubset$.
Our convention thus avoids notational clutter and will not be an issue since we never use strict $\ll_{\rm c}$-Finality in the original sense.

One major roadblock in all completeness proofs for logics with converse wellfoundedness conditions (even the classical $\sf GL$) is that the canonical model is not converse well founded, and there really is no way to control for this without omitting some maximal consistent sets, as $\sf GL$ is not strongly complete.
The method of Fine selection addresses this: arbitrary subsets of the canonical model may not have terminal (or Final) elements, but, due to the L\"ob axiom, {\em definable} subsets of the canonical model do.
We are specifically interested in subsets definable by a variant of the formulas Stra{\ss}burger assigned to certain nested sequents that are very similar to singular $\sf GL$ trees \cite{Strassburger}.

\begin{definition}\label{defYF}
Let $\Sigma\subseteq \lanfull$ be finite and $\fA$ be a singular or positive irreflexive $\Sigma$-sprout. 
Given $w\in\dom\fA$, we define a formula $\simfor{\fA,w}$ by induction on the depth of $w$.
We consider three cases.

\begin{Cases}
\item ($\fA$ is positive).
Then,
  \[
  \simfor{\fA,w} := \bigwedge \ell^+ _\fA(w) \wedge \bigwedge_{v \ler w} \ps \simfor{{\fA,v}}.
  \]

\item ($\fA$ is singular.)  

\begin{Cases}
  
\item ($\ell^-(w) \neq \varnothing$).
Then, $w$ is the unique singularity of $\fA$, so $\restr\fA v$ is positive for all $v\ler w$.
In this case, we set
    \[
  \simfor{\fA,w} := \simfor{\fA^+,w} \to \bigvee \ell^-(w) ,
  \]
  where $\simfor{\fA^+,w}$ is defined as above.
  
\item ($\ell^-(w) = \varnothing$).
There is a unique $v \ler w$ such that $\restr{\fA}{v }$ is singular.
Let $U$ be the set of all $u\ler w$ such that $u \neq v $ and define
    \[
  \simfor{\fA,w} := \bigwedge \ell^+(w) \wedge \bigwedge_{ u\in U} \ps \simfor{\fA,u} \to \nec \simfor{\fA,v} .
  \]

    \end{Cases}
  \end{Cases}
\end{definition}

We define $\simfor\fA=\simfor {\fA,r}$, where $r$ is a root of $\fA$.

\begin{theorem}\label{thmYF}
Let $\Sigma\Subset\langfull$, $\fA$ be a $\Sigma$-sprout, and $\Phi\in W_{\rm c}$.
\begin{enumerate}
\item If $\fA$ is positive, then the following are equivalent.
\begin{enumerate}
    \item  $ \fA \subT \Phi$.
    \item  $\Phi \in \lb \simfor\fA \rb_{\rm c}$.
    \item  $\simfor\fA \in \Phi$. 
\end{enumerate}
\item If $\fA$ is singular, then the following are equivalent.
\begin{enumerate}
    \item\label{YFsim} There is $\Psi \seq_{\rm c} \Phi $ such that $ \fA \subT \Phi$.
    \item\label{YFSimSat} $\Phi \notin \lb \simfor\fA \rb_{\rm c}$.
    \item\label{YFSimTruthLemma} $\simfor\fA \notin \Phi$. 
\end{enumerate}

\end{enumerate}
\end{theorem}
 
\begin{proof}[Proof sketch.]
First note that since $\mathcal M_{\rm c}$ is a model, (b) and (c) are equivalent (in both items).
Then, the first item is identical to the well-known classical case: see e.g.~\cite{me:simulability}.

We thus focus on the second, by induction on the height of $\fA$.
Let $r$ be the root of $\fA$ and assume first that $\ell^-(r) \neq \varnothing$.
Assume that \eqref{YFSimTruthLemma} holds; by the semantics of implication, there is $\Psi \seq_{\rm c} \Phi $ such that $ \simfor{\fA^+} \in  \Psi $ but $ \bigvee \ell^-(w) \notin \Psi $.
By the first item, $ \fA^+ \subT \Psi$ and, since $\bigvee \ell^-(w) \notin \Psi$, it readily follows that $ \fA  \subT \Psi$.
All steps in this argument are invertible, so we also obtain that \eqref{YFsim} implies \eqref{YFSimTruthLemma}.

The more interesting case is where $\ell^-(r) \neq \varnothing$.
Let $v$ be the unique daughter of $r$ such that $\restr{\fA}v$ is singular (i.e., the unique daughter below the singular point) and $U$ be the set of all other daughters of $r$.
Once again, assume \eqref{YFSimTruthLemma} and let $\Psi' \seq_{\rm c} \Phi$ be such that $\bigwedge \ell^+(w) \wedge \bigwedge_{ u\in U} \ps \simfor{\fA,u} \in \Psi' $ but $ \nec \simfor{\fA,v} \notin \Psi'$.
By the semantics of $\nec$, let $\Theta' \sqSupset_{\rm c} \Psi'$ be such that $\simfor{\fA,v} \notin \Theta'$.
By the induction hypothesis applied to $\restr\fA v$, there is $\Theta\seq_{\rm c}\Theta$ such that $\restr\fA v \subT \Theta$ via some homomorphism $h_v$.
By backward confluence,\footnote{This is precisely the step where it is essential that $\fA$ be singular; for, any negative information stored in $\Psi'$ may be lost when passing to $\Psi$.} there is $\Psi\seq_{\rm c} \Psi' $ such that $\Psi\rel_{\rm c}\Theta$.

Note that we still have $\bigwedge \ell^+(w) \wedge \bigwedge_{ u\in U} \ps \simfor{\fA,u} \in \Psi $, so for each $u\in U$ we can find $\Psi_u \ler_{\rm c} \Psi $ with $\simfor{\fA,u} \in \Psi_u $, which by the first claim yields $\restr\fA u\subT \Psi_u $ via some homomorphism $h_u$.

Finally, we can paste the homomorphisms $h_x$ we have found to the ordered pair $(r,\Psi)$ to obtain the desired homomorphism $h$ witnessing that $\fA \subT \Psi$, establishing \eqref{YFsim}.
Finally, we can follow similar reasoning to prove that \eqref{YFsim} implies \eqref{YFSimSat}.
\end{proof}
 
With our Stra{\ss}burger formulas at hand, we are ready to prove that the canonical model for $\sf IGL$ has sufficient strictly Final points.

\begin{proposition}\label{propStrictFinal}
Let $\fA $ be a $\sf GL$ $\Sigma$-sprout and suppose that $\Theta \in W_{\rm c}$ is such that $\fA  \subT \Theta  $.
\begin{enumerate}

\item 
There is $\Theta^* $ such that $\Theta \peq_{\rm c} \Theta^*$ and $\Theta$ is strictly $\prec_{\rm c}$-Final for $ \fA $.

\item 
If $\fA$ is positive, there is $\Theta^* $ such that $\Theta \rrel_{\rm c} \Theta^*$ and $\Theta^*$ is strictly $\rel_{\rm c}$-Final for $ \fA $. 

\item \label{itStrictFour}
If $\fA$ is singular, there is $\Theta^* $ such that $\Theta \lleq_{\rm c} \Theta^*$ and $\Theta^*$ is strictly $\ll_{\rm c}$-Final for $ \fA $.

\end{enumerate}
\end{proposition}

\begin{proof}
The claim for $ \prec_{\rm c}$ already follows from Proposition~\ref{propFinal}, given that $\peq_{\rm c}$ is a partial order.

For $\rel_{\rm c}$, we may assume that $\Theta$ is not already strictly $\rel_{\rm c}$-Final for $\fA$.
Letting $\sigma=\simfor\fA$, this amounts to assuming that $\ps\sigma\in\Theta$.
The idea is to show that if $(\psi_i)_{i<n} \subseteq \Theta^{\overline\ps}$, then there is $\tilde \Theta$ such that $\Theta^\nec\subseteq \tilde \Theta ^ \nec$, $ \sigma \in\tilde\Theta$, $\ps \sigma \notin\tilde\Theta$, and $\bigvee_{i<n}  \ps \psi_i \notin \tilde \Theta $.
By compactness, we obtain the desired $\Theta^* \ler_{\rm c}\Theta$.
Moreover, by \ref{ax:dp}, we have that $\bigvee_{i<n}  \ps \psi_i \equiv \ps \bigvee_{i<n}  \psi_i$, so we may assume that $n=1 $ and work with a single formula  $ \psi  \in \Theta^{\overline\ps}$.

With this in mind, we claim that $ \nec( \sigma\to \ps \psi\vee\ps\sigma ) \notin \Theta$.
For if it were, we would have by \ref{ax:gl:dia} that $\ps \sigma \to \ps \ps \psi \in \Theta$.
But, from the assumption that $\ps \sigma \in \Theta$, this would yield $ \ps \ps \psi \in \Theta$ and, by \ref{ax:trans:dia},  $  \ps \psi \in \Theta$, a contradiction.

By the truth lemma, there is $\tilde\Theta \sqSupset_{\rm c} \Theta$ such that $\sigma\to \ps \psi\vee\ps\sigma \notin \tilde\Theta $, and by the semantics of implication, we may assume that $\sigma\in \tilde \Theta$ but $\ps \psi\vee\ps\sigma \notin \tilde\Theta$.
Since $\psi \in \Theta^{\overline\ps}$ was arbitrary, we then use compactness to find $\Theta^*$ with $\Theta^* \sqSupset_{\rm c} \Theta$, $\sigma\in  \Theta^*$, $ \ps\sigma \notin \tilde\Theta$ and, whenever $\ps \psi\notin \Theta$, then $\ps \psi\notin \Theta^*$.
This readily yields using Theorem~\ref{thmYF} that $\Theta \rel _{\rm c} \Theta^*$ and $\Theta^*$ is strictly $\rel_{\rm c}$-Final for $\fA$.

For $\lleq_{\rm c}$, we may assume by the first item that $\Theta$ is strictly $\prec_{\rm c}$-Final for $\fA$.
If $\Theta$ is not already strictly $\ll_{\rm c}$-final, we use \ref{ax:gl:box} in the form $\nec (\nec \simfor\fA \to   \simfor\fA) \to \nec \simfor\fA$ to find  $\Theta_0 \sqSupset \Theta$ satisfying $\nec \simfor\fA $ but not $\simfor\fA$.
Then, let $\Theta^*\seq \Theta_0$ be strictly $\prec$-Final for $\fA$.
The choice of $\Theta_0$ ensures that no $\Theta'\sqSupset \Theta^*$ satisfies $\fA \subT \Theta'$ (as such $\Theta'$ would also satisfy $\Theta'\sqSupset\Theta_0$), and the choice of $\Theta^*$ ensures that no $\Theta' \succ \Theta^*$ satisfies $\fA \subT \Theta'$.
By definition, no $\Theta' \gg \Theta^*$ satisfies $\fA \subT \Theta'$.
\end{proof}

\begin{corollary}\label{corStrictFin}
Let $\varphi$ be any formula and suppose that $  \Theta \in W_{\rm c}$ is strictly $\prec_{\rm c}$-Final for $\overline\varphi$.
Then, there is $\Theta^* $ such that $\Theta \rel_{\rm c} \Theta^*$ and $\Theta^*$ is strictly $\rel_{\rm c}$-Final for $ \fA $.
\end{corollary}

\begin{proof}
Use Proposition~\ref{propStrictFinal}.\ref{itStrictFour} to find $\Theta \gg_{\rm c} \Theta^*$ and $\Theta^*$ is strictly $\ll_{\rm c}$-Final for $ \fA $.
Note that $\Theta^*$ is also strictly $\rel_{\rm c}$-Final for $ \fA $ and we must have that $\Theta \rel_{\rm c} \Theta^*$ since by definition of $\ll_{\rm c}$, there is $\Theta'\seq_{\rm c}\Theta$ with $\Theta' \rel_{\rm c} \Theta^* $, but by strict $\prec_{\rm c}$-Finality of $\Theta$, we must have $\Theta' = \Theta$.
\end{proof}

\section{Buds and Regions}

Fine's selection method proceeds by restricting the canonical model to $\rrel$-Final worlds for some finite $\Sigma$.
While the set of Final worlds is not finite, it does have the `bounded chain property', in that any chain $ w_1 \srel w_2 \srel \ldots \srel w_n$ has length bounded by $\#\Sigma$.
This is sufficient for its bisimulation quotient to be finite.

In our setting, we will likewise focus on Final worlds, although now we use Finality with respect to sprouts and this no longer leads to the bounded chain property.
We will, however, have the `finite chain property', i.e., the set of Final worlds is {\em Noetherian} (see Theorem~\ref{theoNoetherian}) as a consequence of Kruskal's theorem.
Even this will not immediately yield a finite bisimulation quotient, but it will serve as a starting point for a terminating selection procedure.
At the heart of this termination proof are $\llsim$-Final clusters, which we call {\em regions.}

\begin{remark}
In the proof of Proposition~\ref{propFinal}, $\fA$ need not be tree-like, and thus we can extend our notion of Finality even further to include arbitrary labelled frames.
However, we need to restrict to sprouts to allow for the above-mentioned application of Kruskal's theorem.
\end{remark}

\begin{definition}
Fix $\Sigma \subseteq \lanfull$.
A set $\mathfrak r \subseteq W_{\rm c}$ is a {\em $\Sigma$-region} if it is a full $\llsim_{\rm c}$-Final cluster for $\Sigma$, and it is a {\em region} if it is a $\Sigma$-region for some $\Sigma$.
The {\em canonical bud} of $\mathfrak r$ is its set of $\peq_{\rm c}$-maximal elements and is denoted $\hat{\mathfrak r}$.
Elements of $\hat {\mathfrak r}$ will also be referred to as {\em canonical buds.}

If $\mathfrak r$ contains a strictly $\ll_{\rm c}$-Final point, we say that $\mathfrak r$ is {\em irreflexive.}
\end{definition}

We want to imagine the bud of a region as lying at its top above all other elements, as justified by the following lemma.

\begin{lemma}\label{lemmBudNonEmpty}
If $\Sigma\subseteq\lanfull$ and $\mathfrak r$ is any $\Sigma$-region, then for every $\Phi\in \mathfrak r$, there exists  $\hat \Phi\in \hat {\mathfrak r}  $ such that $\Phi\peq \hat \Phi$.
\end{lemma}

\begin{proof}
Suppose that $\mathfrak r$ is $\llsim$-Final for $ \fA $ and let $\Phi\in {\mathfrak r}$ and $a$ be a root of $\fA$.

If there is $  \Phi_0 \succ \Phi$ such that $(\fA,a) \subT  \Phi_0 $, apply Proposition~\ref{propFinal} to find $\tilde\Phi\seq \Phi_0$ which is $\peq$-Final for $(\fA,a)$.
If there is no such $\Phi_0$, set $\tilde \Phi = \Phi$.
In either case, $\tilde\Phi \seq \Phi$ has the property that there is no $  \Phi_0 \succ \tilde \Phi$ such that $(\fA,a) \subT  \Phi_0 $.

If $\tilde\Phi$ is not already a $\peq$-maximal element of $\mathfrak r$, choose $\Phi_1\succ \tilde\Phi $ with $\Phi_1 \in \mathfrak r$.
From the latter, we obtain $ \Phi_2 \ggsim \Phi_1$ with $(\fA,a) \subT  \Phi_2 $ but, since we cannot have $\Phi_2 \seq \Phi_1 $ by choice of $\tilde\Phi$, we must have that $ \Phi_2 \sqSupset\Phi_1$.
Thus there is $ \Phi_3$ with $\Phi_1\peq \Phi_3 \rel \Phi_2$.

Let $\fA'$ be the result of appending a new irreflexive root $a' $ to $\fA$ with empty label, and note that $ (\fA',a')  \subT \Phi_3$ since $ (\fA ,a )  \subT \Phi_2$ and we can assign $\Phi_3$ to $a' $.
Hence, by Proposition~\ref{propFinal}, there exists $\Phi_*\seq \Phi_3 \seq \tilde \Phi$ which is $\peq$-Final for $\fA'$.
Since $\Phi_* \ggsim \Phi$ and $\fA\subT\Psi  $ for some $\Psi \ler \Phi_*$ (hence $\Psi \ggsim \Phi_*$), we must have that $\Phi_*\in {\mathfrak r}$.
If $\Theta\succ \Phi_*$, there can be no $\Theta'\seq \Theta$ with $(\fA,a) \subT \Theta'$ by choice of $\tilde\Phi $, and there can be no $\Theta'\sqSupset \Theta$ with $(\fA,a) \subT \Theta'$ by choice of $\Phi_*$.
Hence, $\Theta\notin \mathfrak r$, and thus  $\Phi_*$ is a $\peq$-maximal element of ${\mathfrak r}$.
\end{proof}

Buds are particularly convenient as witnesses of backward confluence, since they are not only $\peq$-maximal within their Final cluster, they are also $\rrel$-minimal.

\begin{lemma}\label{lemmregionProps}
Let ${\mathfrak r}$ be a region.
\begin{enumerate}

\item\label{itregionOne} If  $\Phi  \in {\mathfrak r}$, $   \Psi \in \hat {\mathfrak r}$, and $\mathfrak r$ has at least two points, then $  \Psi \rel \Phi$.

\item\label{itregionSingle} If $\mathfrak r$ is irreflexive, then $\mathfrak r = \{\Phi\}$.

\item\label{itregionTwo} If $\Phi ,\Psi \in  {\mathfrak r}$ and $  \Psi\seq \Phi \rel \Theta $, then $\Psi \rel \Theta$.

\end{enumerate}
\end{lemma}

\begin{proof}
For the first claim, assume first that $\Phi\neq\Psi$.
Since ${\mathfrak r}$ is a $\llsim$-cluster, $ \Psi \llsim \Phi$, i.e.~$ \Psi \peq \Phi$ or $ \Psi \sqSubset \Phi$.
But the first case is impossible, since $\peq$ is a partial order and $\Psi$ is maximal, so $ \Psi \sqSubset \Phi$, and hence there is $\Psi'$ such that $ \Psi \peq \Psi' \rel \Phi$.
But then, $\Psi' \in {\mathfrak r}$, so by maximality of $ \Psi$, $\Psi ' = \Psi$ and $\Psi' \rel \Phi$, as desired.

If $\Phi = \Psi$, using the assumption that $\mathfrak r$ has at least two points, we can pick $\Psi'\neq \Psi $ with $\Psi'\in \mathfrak r$.
By definition of $\llsim$, either $\Phi\prec \Psi'$ or $\Phi \sqSubset \Psi'$, but the first is impossible by $\peq$-maximality of $\Phi$.
Thus $\Phi \sqSubset \Psi'$ and, since $\Psi'\in\mathfrak r$, $\Psi'\llsim \Phi$.
In view of Lemma~\ref{lemLL}, we conclude that $\Phi \sqSubset \Phi$.
Then, we can reason as in the case $\Phi\neq\Psi$ to conclude that $\Phi \sqsubset \Phi$.

For the second claim, recall that $\mathfrak r$ is irreflexive if it contains a stricty $\ll_{\rm c}$-Final element $\Phi$.
We also recall our convention that by {\em stricty $\ll_{\rm c}$-Final,} we really mean {\em stricty $({\prec}_{\rm c} \cup {\sqSubset_{\rm c}})$-Final.}
Suppose $\Psi$ is a second point of $\mathfrak r$.
We cannot have $\Psi \succ_{\rm c}\Phi$ with $\Psi\in\mathfrak r$, since it would contradict strict $ {\prec}_{\rm c} $-Finality, which is implied by strict $\ll_{\rm c}$-Finality, of $\Phi$.
Thus by the first claim, $\Phi\rel_{\rm c}\Phi$, which now contradicts strict ${\sqSubset_{\rm c}}$-Finality of $\Phi$.

For the third, if $\ps \varphi\in \Phi$, from $\Phi\peq \Psi$, we obtain $\ps\varphi \in \Psi$.
Using Lemma~\ref{lemmBudNonEmpty}, let $\hat \Phi \in \hat {\mathfrak r}$ be such that $\Phi \peq \hat \Phi$.
Then, if $\nec\varphi\in \Phi$, it follows that $\nec\varphi\in \hat \Phi$.
We then have by the first item that $ \hat\Phi \rel \Psi   $, so $\nec\varphi\in \Psi  $.
Thus, $\Phi^\ps\subseteq \Psi^\ps$ and $\Phi^\nec\supseteq \Psi^\nec$, from which the claim follows by the definition of $\rel_{\rm c}$.
\end{proof}

Regions are partially ordered by $\llsim$, which we will denote $\lleq$ when comparing regions to stress that it is a partial order and not a preorder.
It will be convenient for there to be a `bottom' region, but this is not automatically given by the definitions.
Instead, we will work within a generated substructure of the canonical model which does have a bottom element.

To this end, fix $\Sigma\Subset\lanfull$ and let $p_{\Perp}$ be a fresh atom not included in $\Sigma$.
Note that validity is invariant under renaming variables, so this is a harmless restriction.

\begin{definition}
A {\em pseudo-root} is an element ${\Perp} \in W_{\rm c}$ such that
\begin{enumerate}

\item $  {p_{\Perp}} \wedge \nec {\sim} {p_{\Perp}}  \in {\Perp} $,

\item if $\Phi\in W_{\rm c}$ is such that $\boxdot {\sim} p_{\Perp} \in \Phi$, then ${\Perp}\rel \Phi $.

\end{enumerate}

\end{definition}

\begin{proposition}\label{propCanRoot}
A pseudo-root exists.
\end{proposition}

\begin{proof}
Let $\cl M_{\Perp} = (W_{\Perp},\peq_{\Perp},\rel_{\Perp},\val\cdot_{\Perp})$ be obtained by performing the following two operations on $\cl M_{\rm c}$:
\begin{enumerate}

\item Setting $\val {p_\Perp}_{\Perp} = \varnothing $ and $ \val {p }_{\Perp} = \val {p }_{\rm c} $ for $p\neq p_\Perp $, and 

\item adding a new point, $ {\Perp}$, and setting
\begin{align*}
\peq_{\Perp} & : = { \peq_{\rm c}} \cup { (w_{\Perp},w_{\Perp})} ,\\
\rel_{\Perp} & : = {\rel_{\rm c}} \cup {(\{{\Perp} \} \times W_{\rm c})} .
\end{align*}

\end{enumerate}

It is not hard to check that the resulting structure is an intuitionistic transitive model, so that $\{\varphi\in\lanfull : \cl M_\Perp \models \varphi({\Perp})\}$ is a prime theory, and hence an element of $\cl M_{\rm c}$.
It is then easy to check that $\Perp$ has the desired properties.
\end{proof}

We then note that ${\Perp}$ is $\llsim$-Final for $p_{\Perp}$ and hence $\{{\Perp}\}$ is a region.
By working exclusively with the generated submodel of $\Perp$, we can treat $\{{\Perp}\}$ as the `root region'.
 In the sequel, if $\Sigma\Subset\lanfull$, we will further assume that $p_{\Perp} \in \Sigma$, but $p_{\Perp}$ does not occur in any other formulas of $\Sigma$.

\section{Canonical Trees}

In the classical setting, Fine selection can be set up by restricting the canonical model to substructures consisting of Final worlds, but we will need a bit more flexibility, so instead of proper substructures, we will work with structures that homomorphically map into the canonical model.
These structures are very similar to the sprouts and $\Sigma$-trees we have worked with, except that they are labelled by prime theories.

\begin{definition}
Let ${\sf L} \in\{\ikf,{\sf IGL}\}$ and $\mathcal M_{\mathrm c} =( W _{\mathrm c} ,{\peq_{\mathrm c} },{\rel_{\mathrm c} }, \ell _{\mathrm c} )$ be the canonical model for $\sf L$.

\begin{enumerate}

\item For ${\sf L} = \ikf$, {\em canonical sprout for $\sf L$} is a tuple $\fA=(\dom{\fA}, \rel _\fA, h_\fA )$, consisting of a finite, transitive, tree-like frame equipped with a $\rel$-homomorphism $h\colon \dom{\fA}\to W_{\rm c}$.

\item For ${\sf L} = \sf IGL$, {\em a canonical sprout for $\sf L$} is defined similarly, except that we demand that $\rel _\fA$ be irreflexive. 

\end{enumerate}
If $\Sigma \Subset\lanfull$, we assign to $\fA$ a $\Sigma$-sprout $\fA {{\upharpoonright}} \Sigma$ given by
\[\fA {{\upharpoonright}} \Sigma = (\dom{\fA}, \rel _\fA, \ell_\fA ),\]
where $ \ell_\fA(w) = (\ell^+_\fA(w),\ell^-_\fA(w),\ell^\notnec_\fA(w) )$ is given by
\begin{itemize}

\item 
$ \ell^+_\fA(w) = \Sigma\cap  h_\fA (w)$,  $ \ell^-_\fA(w) = \Sigma\setminus h_\fA (w) $, and

\item  $\ell^\notnec _\fA(w) $ is the set of $\varphi\in\Sigma$ such that either there is  $v \ler_\fA w $ with $\varphi\notin h_\fA(v)$, or $h_\fA(v)$ is $\peq_{\rm c}$-Final for $\overline{\nec\varphi}$.

\end{itemize}

As usual, we assume $\sf L$ fixed and only mention it when the choice of logic is relevant.
We say that a canonical sprout $\fA$ is a {\em canonical $\Sigma$-tree} if
$\fA{{\upharpoonright}} \Sigma$ is a $\Sigma$-tree, and that it is {\em rooted} if it has a unique root $r$ with $h(r) = {\Perp}$.
The canonical sprout is {\em irreflexive} or {\em $\sf GL$} if $\rel_\fA$ is irreflexive.
\end{definition}

\begin{remark}\label{remNegDia}
Our choice of $\ell^\notnec _\fA(w) $ might need some explanation.
This is specifically tailored for $\sf IGL$; for $\ikf$ the definition could be simplified, but in any case, this particular choice will not hurt.
The intuition is that we avoid putting $\ps\overline\varphi \in \ell_\fA(w)$ if we can `realise $\overline{\nec\varphi}$ later'.
If $\ps\overline\varphi$ already {\em has} a witness, then it's already realised, so we cannot postpone this.
Otherwise, if $h_\fA(w)$ is not $\peq_{\rm c}$-Final, then we can realise $\overline{\nec\varphi}$ via some $\Psi \seq_{\rm c } h_\fA(w)$ which occurs later in our construction.
On the other hand, if $h_\fA(w)$ is $\peq_{\rm c}$-Final, there will not be any other opportunity to realise $\overline{\nec\varphi}$.
The upshot is that, in this situation, we have access to Corollary~\ref{corStrictFin}, so we have a strictly $\rel$-Final witness for $\ps\overline\varphi$, which will in particular imply that this witness is irreflexive, as required by $\sf IGL$.
\end{remark}

As a special case, any set $A \subseteq W_{\rm c} $ such that ${\rel_{\rm c}}{{\upharpoonright}} A$ is tree-like gives rise to a canonical sprout $\cl A = \cl M_{\rm c}{{\upharpoonright}} A$ in the obvious way.
In this case, we will identify $\cl A$ with $A$ and think of $A$ itself as a canonical sprout.
However, not all canonical sprouts are of this form.
If $A$ is not tree-like, we may still have a sprout mapping onto $A$, essentially via a  tree unwinding.

\begin{definition}
If $\fA$ and $\fB$ are canonical sprouts, we write $\fA\subseteq_{\rm c} \fB$ if $\fA$ is an {\em initial} substructure of $\fB$, i.e., $\dom{\fA} $ is a $\rrel_\fB$-downward closed subset of $\dom{\fB}$, ${\rel_{\fA}} = {\rel_{\fB}}{{\upharpoonright}} \dom{\fA}$, and ${h_{\fA}} = {h_{\fB}}{{\upharpoonright}} \dom{\fA}$.

If $w\in \dom{\fA}$ and for all $v\in \dom{\fB}\setminus \dom{\fA}$, $w\srel_\fB v$, we say that $\fB$ is a {\em local extension at $w$.}
\end{definition}

Note that the order $\subseteq_{\rm c}$ on canonical sprouts is much more `rigid'  than the relation $\subT$ on $\Sigma$-sprouts; this is because our construction will proceed by iteratively adding worlds to a $\Sigma$-sprout until we eventually obtain a $\Sigma$-model.
The weaker relation $\subT$ will mainly be used for treating Finality.

Of course, if we want to produce $\Sigma$-models, we need to equip our structures with an intuitionistic preorder.

\begin{definition}
Let $\Sigma\Subset\lanfull$ and $\fA$ be a canonical $\Sigma$-tree.
We say that ${\peq} \subseteq \dom{\fA} \times\dom{\fA}$ is a {\em $\Sigma$-intuitionistic relation} if it is an intuitionistic relation on $\fA{{\upharpoonright}} \Sigma$.

Let us write $\Phi \peq^\Sigma_{\rm c} \Psi $ if there is $\Psi' \equiv_{\rm c} \Psi$ such that $\Psi' \upharpoonright \Sigma = \Psi\upharpoonright \Sigma $ and $\Phi \peq_{\rm c} \Psi' $.
If $\peq $ is a $\Sigma$-intuitionistic relation such that $h$ is a homomorphism with respect to $\peq^\Sigma_{\rm c}$ in the sense that $w \peq  v$ implies that $h(w) \peq^\Sigma_{\rm c} h(v)$, we say that $\peq$ is {\em $\Sigma$-honest,} or {\em honest} if $\Sigma$ is clear from context.

A {\em canonical $\Sigma$-quasimodel} is a tuple $\fA=(\dom{\fA},\peq_\fA, \rel _\fA, h_\fA )$, consisting of a canonical $\Sigma$-tree equipped with a $\Sigma$-intuitionistic relation.
\end{definition}

\section{Pivotal Relations and Grafting}

The na\"ive strategy for resolving defects is to resolve them as soon as they arise, but this will have the issue that a resolution for one defect may lead to undoing a resolution for another.
We thus have to be careful regarding exactly {\em how} defects are resolved.
In particular, it will be convenient for them to be resolved `as locally as possible', which formally will be when they are {\em pivotal,} as defined below.

Since any sprout $\fA$ is tree-like, we have an operation $ \sqcap \colon \dom{\fA}^2\to 2^{\dom{\fA}}$ such that $w \sqcap v$ is the {\em meet} of $w$ and $v$ in the standard way, with the caveat that $w \sqcap v$ is a cluster and not a single element.
However, since the choice of an element of this cluster will not be essential, we will write $u=w \sqcap v $ instead of $u\in w\sqcap v$.
Below, recall that we defined $x\prec y $ if $x\peq y$ and $x\neq y$, which in principle allows for $y\peq x$ since $\peq$ is only assumed to be a preorder.

\begin{definition}\label{defPivot}
Let $\fA$ be a canonical sprout and $\Sigma\Subset\lanfull$.
A $\Sigma$-intuitionistic preorder $\peq$ is {\em pivotal} if there is a constant $w$ such that, whenever $u\prec u'$, it follows that $w\rrel u\sqcap u' $.
Both $w$ and $[w]$ are {\em pivots} of $\peq$.
If the stronger condition that $w \srel u $ and $w\srel u' $ is always satisfied, then $\peq$ is {\em strictly pivotal.}

A preorder $\peq$ is {\em (strictly) pointwise pivotal} if, whenever $ x \prec y$, there is a pivotal ${\peq'}\subseteq {\peq}$ with pivot $x\sqcap y$ such that $x \peq' y$.
\end{definition}

Pivotal relations are useful for $\ikf$ and strictly pivotal ones for $\sf IGL$, so readers interested in only one of the two logics may ignore the other definition.

\begin{definition}
Let $\fA $ and $\fB$ be canonical sprouts, $w\in \dom{\fA}$ and $r$ be a root of $\fB$ such that $h_\fA(w) \rel _{\rm c} h_\fB (r) $.
We define $\fA \sqcup _ w \fB = (W,\rel,h)$, where
\begin{itemize}

\item $W=\dom{\fA} \amalg \dom{\fB}$,

\item $x \rel y$ if and only if either $ x\rel_\fA y$, $x\rel _\fB y$, or $x\rrel _\fA w$ and $y\in\dom\fB$, and

\item $h = h_\fA \amalg  h_\fB$.

\end{itemize}
We refer to the operation $\fA \sqcup _w \fB$ as {\em grafting $\fB$ onto $(\fA,w)$,} or simply {\em onto $w$} if $\fA$ is clear from context.
If $\fA$, $w$ and $\fB$ satisfy the assumptions of the definition, we will say that {\em $\fA \sqcup _w \fB$ can be grafted.}

If $\mathfrak B$ is a finite set of canonical sprouts such that $\fA \sqcup _w \fB$ can be grafted for every $\fB\in \mathfrak B$, $\fA\sqcup _w \mathfrak B$ denotes the result of successively grafting each element of $\mathfrak B$ onto $w$.
Note that, up to isomorphism, this operation does not depend on how $\mathfrak B$ is ordered.
\end{definition}

We may write $\fA\sqcup \fB$ (omitting the grafting point $w$) when $w$ is a root of $\fA$.
The operation of grafting and pivotal relations go hand in hand and will work together to resolve defects.

\begin{definition}
Let $\Sigma \subseteq \lanfull$ and suppose that $\fA \sqcup _w  \fB$ can be grafted.
We say that a $\Sigma$-intuitionistic relation $\peq$ is {\em (strictly) pivotal for $\fA\sqcup _w\fB$} if it is (strictly) pivotal with pivot $w$ and, whenever $u\prec u'$, it follows that $u\in\dom\fA$ and $u'\in [w]\cup \dom\fB$.
\end{definition}

Thus, the general strategy is to resolve defects by successively grafting and adding pivotal relations, which are relatively stable when subsequent elements are grafted.
In order to ensure termination, we need to select our witnesses in such a way that they belong to an increasing $\ll$-chain of regions as they are added.
As we mentioned, the order $\lleq$ is Noetherian, so this will ensure that the process will eventually halt.
To make this precise, we need to assign regions to canonical sprouts.

When working with $\sf IGL$, we will restrict our attention to irreflexive regions, which are singletons by Lemma~\ref{lemmregionProps}.\ref{itregionSingle} and hence have relatively simple structure.
It follows in particular that $\mathfrak r=\hat{\mathfrak r}$ for any such region.
We can, however, not restrict to such regions when working with $\ikf$, and we must keep track of which points in a canonical sprout correspond to canonical buds.
Below, if ${\mathfrak r} = \rho(w) $, then $\hat\rho(w) := \hat{{\mathfrak r}}$, and, if $C$ is a $\rel_{\rm c}$-cluster, a {\em set of $\Sigma$-representatives for $C$} is a subset $C'$ of $C$ such that $C'{{\upharpoonright}}  \Sigma = C {{\upharpoonright}} \Sigma$, where $X {{\upharpoonright}} \Sigma = \{\Xi {{\upharpoonright}} \Sigma : \Xi\in X \}$.

\begin{definition}
Fix ${\sf L} \in \{\ikf,{\sf IGL}\}$, $\Sigma\Subset\lanfull$, and let $\Phi\in W_{\rm c}$.
\begin{itemize}

\item If ${\sf L} = \ikf$, define $\rho(\Phi) $ to be the unique region ${\mathfrak r}$ such that $ \Phi \in {\mathfrak r}$ if such a region exists, and otherwise $\rho(\Phi) = \{{\Perp}\}$.

\item If ${\sf L} = \sf IGL$, we define $\rho(\Phi) $ similarly, except it is either the unique {\em irreflexive} region ${\mathfrak r}$ such that $ \Phi \in {\mathfrak r}$ if such a region exists, and otherwise $\rho(\Phi) = \{{\Perp}\}$.

\end{itemize}

If $\fA$ is a canonical sprout and $w\in\dom{\fA}$, we define $\rho(w) = \rho(h_\fA(w))$.
If ${\sf L} = \isf$, say $w \in \dom\fA$ is a {\em bud} if:
\begin{enumerate}

\item $h([w])$ is a set of $\Sigma$-representatives of $\hat\rho(w)$,

\item if $w'\srel w$, then $\rho(w')\neq\rho(w)$.

\end{enumerate}
If ${\sf L} = \sf IGL$, $w $ is a {\em bud} if the above holds and also $\rho(w)$ is an irreflexive region.
\end{definition}

Buds will be the ideal place to perform grafting, since their order properties as stated in Lemma~\ref{lemmregionProps} will allow us to `hack' the confluence conditions.
Note that if $\rho(w) $ is an irreflexive region, then $ w$ will be a bud by default, which will simplify things considerably when working with $\sf IGL$.

\section{Confluence Propagation}

The next step is to explore some consequences of Finality and confluence that will yield useful constructions.
We begin with Lemma~\ref{lemmModalDefects}, which roughly corresponds to how this is done in the classical setting, where $\rel$-Final worlds are used to find witnesses for formulas of the form $\ps\varphi$.
We remark that the notation $[\cdot ]$ always refers specifically to $\rel$-clusters.

Below, recall that the difference between canonical sprouts and canonical $\Sigma$-trees is that the latter have no modal defects in $\Sigma$.

\begin{lemma}\label{lemmModalDefects}
Let ${\sf L} \in \{\ikf,{\sf IGL}\}$.
If $\Sigma \Subset \lanfull$ and $\fA$ is a canonical sprout for $\sf L$, then there exists a canonical $\Sigma$-tree $\fB\supseteq_{\rm c}\fA$ for $\sf L$ such that ${\rm hgt}(\fB)\leq {\rm hgt}(\fA) + \#\Sigma$ and  ${\rm wdt}(\fB)\leq {\rm wdt}(\fA) + \#\Sigma$.
\end{lemma}

\begin{proof}
This is a routine Finality argument.
For each modal defect $\delta = (v,\ps\varphi)$ of $\fA $, we use Proposition~\ref{propFinal} to successively choose $\rel_{\rm c}$-Final $\Psi \ler_{\rm c} h(v)$ for $\varphi$, select a set of $\Sigma$-representatives $C$ for $[\Psi] $, then graft them to $v$.
The process satisfies the stated bounds, because the width of every node $v$ is bounded by its original width on $\fA$ together with the number of modal defects of $v$, while any chain that is added to $\fA$ is bounded in length by $\#\Sigma$, as each defect corresponds to an element of $\Sigma$ and these cannot repeat in a strict chain due to Finality.
When ${\sf L} = \sf IGL$, we use Proposition~\ref{propStrictFinal} and Corollary~\ref{corStrictFin} to ensure that all points added are irreflexive (see also Remark~\ref{remNegDia}).
\end{proof}

Next, we turn our attention to resolving intuitionistic defects.
This will be more involved and require some preparatory results.
First, we show that, in the canonical model, there is a variant of the confluence properties which applies to clusters  rather than individual points. 
This result is needed only for $\ikf$, although technically it holds for both logics.
Below, recall that if $C,D$ are $\rel$-clusters of some frame $\fA$, we write $C \peq_\fA D$ if for every $c\in C$, there is $d\in D$ with $c\peq d$.

\begin{lemma}\label{lemmClusConf}
Let $\Phi,\Phi',\Psi \in W_{\rm c}$.
\begin{enumerate}

\item\label{itClusConfFor} 
If $\Phi' \seq \Phi \rel \Psi$, there is $\Psi'  $ such that $\Phi' \rel \Psi'  $ and $ [\Psi] \peq [\Psi'] $, and 

\item\label{itClusConfBack}  If $\Psi\rel \Phi\llsim \Phi' $, there is $\Psi' $ such that $ \Psi'\rel \Phi' $ and $ [\Psi] \peq [\Psi'] $.

\end{enumerate}
\end{lemma}

\begin{proof}
View $[\Psi] $ as a canonical sprout and let $\fC = [\Psi]^+ $ be its positive part, as in Definition~\ref{defPosPart}, so that $\ell_\fC(\Xi) \subT \Theta $ if and only if $\Xi \peq_{\rm c}\Theta$.

For the first item, by forward confluence, there is $\Psi''$ with $\Phi \rel \Psi'' \seq  \Psi$.
Iteratively using forward and backward confluence, we easily see that $\fC \simu_\Sigma \Psi''$, so by Proposition~\ref{propFinal}\ref{itFinalSim}, there is a simulably $ \rel $-Final $\Psi' \ler_{\rm c}\Psi''$ for $\fC$.
Then, if $E\colon \dom{\fC} \simu W_{\rm c}$ is a simulation with $ \Psi \mathrel E \Psi' $ and $\Upsilon \in [\Psi] $, using $\Psi\rel\Upsilon$ and the confluence of $E$, we find $\Upsilon'$ such that $\Psi' \rel \Upsilon' $ and $\Upsilon \mathrel E\Upsilon'$.

Since $\Upsilon \rel \Psi$, we again use confluence to find $\Psi^*$ with $\Upsilon' \rel \Psi^* $ and $\Psi \mathrel E \Psi^*$.
By transitivity of $\rel$, $\Psi'\rel \Psi^*$ and, by $\rel$-Finality, $\Psi^*\equiv \Psi'$, hence also $\Upsilon '\equiv \Psi' $.
But note that $\Upsilon \mathrel E\Upsilon'$ implies that $\Upsilon \peq_{\rm c} \Upsilon'$ by the way we defined $\fC$, so, indeed, $[\Psi]\peq_{\rm c}[\Psi']$.

The second claim is similar, except that we have the weaker assumption that $\Phi \llsim \Phi'$.
If $\Phi \peq \Phi'$, let $\Phi''=\Phi'$, otherwise, pick $\Phi''$ such that $\Phi \peq \Phi'' \rel \Phi'$.
Then, repeat the above argument with $\Phi''$ instead of $\Phi'$, and $\rel$ and $\ler$ interchanged.
\end{proof}

The next lemma will be essential in showing that many of our constructions are stable under subsequent extensions.

\begin{lemma}\label{lemmGraftExt}
Let $\fA \subseteq_{\rm c} \fA' $ be canonical sprouts for ${\sf L} \in \{\ikf,{\sf IGL}\}$.
Suppose that $\fA \sqcup _w \fB$ can be grafted  and let $\peq $ be an honest, pivotal order for $\cl A \sqcup _w\cl B$.
Then, there are a canonical sprout $\fB' \supseteq_{\rm c} \fB $ and a pivotal order ${\peq'}  $ for $\fA'  \sqcup _w \fB' $ with the property that if $x\in \dom\fA$ or $y\in\dom\fB$, then  $x\peq'y$ if and only if $x\peq y$ and, if $\peq$ is strictly pivotal, then so is $\peq'$.
\end{lemma}

\begin{proof}
As before, we follow an inductive argument, this time on $\#(\dom{\fA'} \setminus \dom{\fA} )$.
Pick a $\rrel$-minimal element $w'$ of $\#(\dom{\fA'} \setminus \dom{\fA} )$, so that $w  \rel^1 w' $ for some $w\in \dom{\fA}$.
Let $v \succ w$ with $v\in \dom{\fB}$.
Using forward confluence when $[w']$ is a singleton or Lemma~\ref{lemmClusConf} when it is a proper cluster, pick a singleton or cluster $D \ler_{\rm c} h(v)$ with $D \seq_{\rm c} h([w'] )$, then pick a set of cluster $\Sigma$-representatives $C$ for $D$ and graft $C$ onto $v$, with $ [w']\peq '' C$, otherwise, set $x\peq'' y$ if and only if $x\peq y$.
Note that strictness is preserved by this operation.

There may be more than one possible choice for $v$; after repeating this operation for every such choice, we obtain $\fB''$ and a pivotal order ${\peq''}\supseteq {\peq}$ on $  \fA '' \sqcup_w \fB''$, where $\fA''$ is the restriction of $\fA'$ to $\dom\fA\cup [w']$.

Note that ${\peq''}\supseteq {\peq}$ and we have only added pairs $(x,y)$ to $\peq''$ when $x\notin\dom\fA$ and $y\notin\dom\fB$.
We can then apply the induction hypothesis to $\fA''\subseteq_{\rm c} \fA'$, given that $\#(\dom{\fA'}\setminus \dom {\fA  '' }) < \#(\dom{\fA'}\setminus \dom \fA)$, and thus obtain the desired $\fB' \supseteq_{\rm c} \fB''\supseteq_{\rm c} \fB $.
\end{proof}

One handy property preserved by this sort of extension is {\em root-preservation:} if $\fA,\fB$ are tree-like transitive frames, say that $R\subseteq \dom\fA\times\dom\fB$ is {\em root-preserving} if, whenever $a$ is a root of $\fA$ and $a\mathrel R b$, then $b$ is a root of $\fB$ (note that we do not require the converse implication).
Then, due to the conservation properties given by Lemma~\ref{lemmGraftExt}, we obtain the following variant.

\begin{corollary}\label{corRootPres}
Under the assumptions of Lemma~\ref{lemmGraftExt}, if $\peq$ is root-preserving, then $\peq'$ may also be chosen to be root-preserving.
\end{corollary}
 
Buds will be useful in order to `collapse' regions into a single cluster, as given by the following two lemmas.
They are only needed for treating $\ikf$.

\begin{lemma}\label{lemmBudCollapse}
Let $\Sigma\Subset\lanfull $, $\mathfrak r$ be a $\Sigma$-region, and $\fA$ be a canonical sprout such that $\rho(x) = \mathfrak r$ for all $x\in\dom\fA$.
Suppose that $C \subseteq W_{\rm c}$ is a cluster with $\hat{\mathfrak r} \peq^\Sigma_{\rm c} C $.
Then, viewing $C$ as a canonical sprout with a single cluster, $\fA \peq C$ via a $\Sigma$-honest relation.
\end{lemma} 

\begin{proof}
For $x\in \dom\fA$ and $\Phi\in C$, define $x \mathrel R \Phi$ if $h(x) \peq^\Sigma_{\rm c} \Phi$.
This relation is total by Lemma~\ref{lemmBudNonEmpty}, since given $x\in \dom\fA$, there is $\hat\Phi\in\hat{\mathfrak r}$ with $h(x) \prec \hat\Phi $ and by assumption, there is $\Phi\in C$ with $\hat\Phi \peq^\Sigma_{\rm c}\Phi$.
Confluence is then trivially checked, as for example if $ x \rel y \mathrel R \Psi $, then any $\Phi $ with $x\mathrel R\Phi$ will serve as a witness of backward confluence, and honesty is immediate by construction.
\end{proof}

\begin{lemma}\label{lemmBudCollapseTwo}
Let $\Sigma\Subset\lanfull $, $\fA$ be a canonical sprout, and $r$ be a root of $\fA$ such that $r$ is a bud.
Let $H=\{x\in\dom{\fA}:\rho(x) = \rho(r) \}$.
Then, there is a $\Sigma$-intuitionistic relation $ E \subseteq \dom\fA\times\dom\fA$ such that, for every $x\in H $, there is $w\equiv r$ with $x\mathrel E w$. 
\end{lemma} 

\begin{proof}
We define $E$ by $x\mathrel E y$ if $x=y$ or $x \in H$, $y\in [r]$, and $h(x) \peq^\Sigma_{\rm c} h(y)$.
It remains to check that $E$ is confluent.
For backward confluence, assume that $ x \rel y \mathrel E y' $.
If $y=y'$, take $x=x'$.
Otherwise, $y' \in [r]$ and thus $\rho(y ) = \rho(x) =\rho(r)$, so $x\in H$.
Applying Lemma~\ref{lemmBudCollapse} to $h([r])$, we find $x' \in [r]$ with $h(x) \peq^\Sigma_{\rm c} h(y)$ and thus $x\mathrel E x' \rel y'$.

Forward confluence is easier, as we can uniformly take $y'=y$.
\end{proof}

This will allow us to conveniently graft new sprouts exclusively onto buds while resolving intuitionistic defects.
The next result applies uniformly to $\ikf$ and $\sf IGL$ and will be used for both.

\begin{lemma}\label{lemmBudGraft}
Let $\Sigma\Subset\lanfull $ and $\fA$, $\fB$, $\fC$ be canonical sprouts.
Suppose that $u',u\in\dom{\fA}$ are such that $u' \rel^1 u$.
Suppose further that $E$ is an intuitionistic relation from $\restr \fA u $  to $ \fB $, $R$ is an intuitionistic relation from  $\fA $ to $ \fC$, and $w\in\dom \fC $ is such that $u'\mathrel R w$.

Then, $ R\cup E $ is an intuitionistic relation from $\fA$ to $\fC\sqcup_w \fB$.
\end{lemma} 

\begin{proof}
Let ${\peq} = R\cup E$; we must check that $\peq$ is expansive.
That $\peq$ is forward confluent is straightforward.
Indeed, assume that $x' \seq x \rel y$.
If $x \mathrel R x'$, we may use the witness given by the forward confluence of $R$, otherwise, use the witness given by the forward confluence of $E$.

For backward confluence, assume that $x \rel y \peq y'$.
If $y\mathrel R y'$, we may again use the witness of backward confluence of $R$.
If $y\mathrel E y'$ and $u \rrel x$, then we use the witness from $E$.
Otherwise, since $\fA$ is tree-like and $u'\rel^1 u \rrel y$, $x \rrel u'$.
Use the backward confuence of $R$ on $x \rrel u' \mathrel R w$ to find $x'$ with $x \mathrel R x' \rrel w$ (take $x'=w$ when $x=u'$).
Since $w \rel_{\fC\sqcup_w \fB } y'$ by the definition of grafting, this $x'$ works by transitivity.
\end{proof}

\section{Defect Resolution}

In Lemma~\ref{lemmModalDefects}, we have already seen how to resolve modal defects.
The rest of the constructions in the previous section will be aimed at resolving intuitionistic defects, which will mostly be referred to as {\em defects} from here on, as we have already seen how to resolve modal defects.
Such defects will be resolved via grafting, according to the following definition.

\begin{definition} 
Let $\delta=(v,\varphi)$ be an intuitionistic defect of a canonical $\Sigma$-quasimodel $\fA$.
We say that a pair $(\fB,w)$ is a {\em $\ikf$ potential resolution} for $\delta$ if $w$ is a bud, $\fA \sqcup_ w \fB$ can be grafted, and there is a pivotal order $\peq $ for $\fA\sqcup _w \fB$ such that $\peq$ resolves $\delta$.
It is moreover an {\em $\sf IGL$ potential resolution} if and only if $\peq $ is strictly pivotal and both $\fA$ and $\fB$ are irreflexive.
If there is $y$ such that $ w\prec y$ and $w\rel y$, the potential resolution is a {\em stub,} otherwise it is {\em live.}

We say that $u$ is a {\em target} for  $(\fB,w)$ if it is a root of $\fB$ and a {\em source} if either $ w\rel^1 u \rrel v$ or $w\equiv v\equiv u$.
\end{definition}

Note that $\sf IGL$ potential resolutions are $\ikf$ resolutions, but not the other way around.
As we will see later, the condition of liveness will allow us to use $w$ to resolve future defects, whereas stubs cannot be reused.
Note that stubs are not created by strictly pivotal orders, and hence all potential resolutions for $\sf IGL$ will be live, but this need not be the case for $\ikf$.

We will construct potential resolutions by a recursive method, starting from the defect itself and climbing down towards the root, until we find a good pivot for the resolution.
The base case of this recursion is an {\em initial potential resolution,} which is a $\Sigma$-tree whose root is a resolution link.
However, we cannot yet obtain a suitable intuitionistic relation to resolve the defect, since, in principle, backward confluence may fail at this root, even though it holds elsewhere.

\begin{definition}
Let $\Sigma\Subset \lanfull$, $\delta=(v,\varphi)$ be a defect of $\fA$, $\fB$ be a $\Sigma$-tree with a single irreflexive root $t$, and $R\subseteq \dom{\fA_v} \times \dom\fB$.
We say that $(\fB,R)$ is an {\em initial potential resolution  for $\delta$} if 
\begin{enumerate}

\item $v\mathrel R t$, $\Sigma_\fA(u) \subseteq h(t)$, and $h(t)$ is a resolution link for $\varphi$,

\item if ${\rel}(v) \neq\varnothing$, then for $\fA' = \fA {{\upharpoonright}} {\rel}(v)$, there is $s\ler^1 t$ such that $R \cap \dom{\fA_v} \times \dom{\fB_s}$ is a $\Sigma$-intuitionistic relation, and

\item if $u$ is a bud, then  $ h(u) \ll h(t)$.

\end{enumerate}
If both $\fA $ and $\fB$ are irreflexive, then it is an {\em $\sf IGL$ initial potential resolution.}
\end{definition}

\begin{lemma}\label{lemmInitialRes}
If ${\sf L} \in \{\ikf,{\sf IGL}\} $ and $\fA$ is a $\Sigma$-tree for $\sf L$ with a defect $\delta$, then $\delta$ has an $\sf L$ initial potential resolution.
\end{lemma}

\begin{proof}
Using Proposition~\ref{propFinal}, let $\Phi\seq h(v)$ be $\peq_{\rm c}$-Final for $\overline\varphi$, so that $\Phi$ is a resolution link for $ \delta : =(v,\varphi)$.
If $v$ is reflexive, use Lemma~\ref{lemmClusConf}\ref{itClusConfFor} to find a cluster $D'$ such that $ h([v]) \peq  D'$ and $  \Phi \rel D'  $, otherwise, set $D ' =\{\Phi\}$, and let $D$ be a set of $\Sigma$-representatives for $D$.
By Lemma~\ref{lemmBudCollapse}, there is an expansive and $\Sigma$-honest relation $R\subseteq  [v]  \times  D $.
Use Lemma~\ref{lemmGraftExt} to extend $D$ to some $\fB'$ such that $\fA_v \peq \fB'$, witnessed by some $\Sigma$-intuitionistic relation $R'\supseteq R$.
Add $\Phi$ to $\fB'$ as an irreflexive root, then use Lemma~\ref{lemmModalDefects} to extend to a canonical $\Sigma$-tree $\fB$.
By adding $(v,\Phi) $ to $R' $, we obtain the desired initial potential resolution, where the point of resolution is $\Phi$ if $\varphi$ is an implication and, if $\varphi = \nec\psi$, is a witness for $\ps\overline \psi$ which was added to $\fB$ given the $\peq_{\rm c}$-Finality of $\Phi$ (see Remark~\ref{remNegDia}).
\end{proof}

Next, we develop the inductive step of climbing down to the root.
The idea is to try to obtain a potential resolution but, if this fails, we instead obtain the next-best-thing, which is a {\em local potential resolution.}
These won't quite do the job of resolving the defect, but can be propagated downward until we find a better opportunity.
Below, we will not explicitly mention the logic $\sf L$, but if ${\sf L} = \sf IGL$, then all sprouts and resolutions will be tacitly assumed to be for $\sf IGL$.

\begin{definition}
Let  $\Sigma\Subset\lanfull$, $\delta=(v,\varphi)$ be a defect of a canonical  $\Sigma$-tree $\fA$, and $u\rrel v$.
Let $\fB$ be a $\Sigma$-tree with root $t$ and $R\subseteq\dom{\restr\fA v}\times\dom\fB$.
We say that $(\fB,R)$ is a {\em local potential resolution  for $\delta$ at $u$} if 
\begin{enumerate}

\item $h(u) \llsim h(t)$,

\item there is a $\Sigma$-intuitionistic relation $R\subseteq \dom{\fA_u} 
\times \dom\fB$ resolving $\delta$.

\end{enumerate}
If moreover $u$ is not a bud but $t$ is, we say that $(\fB,R)$ is {\em budding,} and if $h(u) \ll h(t)$, then it is {\em progressive.}
\end{definition}

The following definition is relevant only for $\ikf$.

\begin{definition}
If $\fA$ is a canonical sprout, $w\in\dom\fA$ is a {\em breaking point} if $w\rrel v$ and either $\rho(w) = \{{\Perp}\}$, or else there is $w'$ with $w'\rel^1 w$ and $\rho(w')\neq\rho(w)$. 

If $v\in\dom\fA$, ${\rm Break}(v)$  is the set of all $w\rrel v$ such that $w$ is a breaking point.
\end{definition}

We specifically wish to construct local potential resolutions at {\em breaking points,} which are roughly points where $\rho$ makes a `jump'.
We remark that, over $\sf IGL$, every point in our constructions will be a breaking point, so there is no restriction on where local potential resolutions can occur.

\begin{definition}\label{defLadder}
Let  $\Sigma\Subset\lanfull$, $\delta=(v,\varphi)$ be a defect of a canonical  $\Sigma$-tree $\fA$,   $u\rrel v$, and $\vec{\fB} := \{(\fB_u,R_u):u \in {\rm Break}(v) \}$ be such that $(\fB_u,R_u)$ is a local potential resolution for $\delta$ at $u$.

For $u \in {\rm Break}(v)$, let $(\fB'_u,R')$ be either $(\fB_{u'},R_{u'})$, where $u' $ is a least element of ${\rm Break}(v) \cap (u,v]$, if such an element exists, or an initial potential resolution for $\delta$ if no such element exists.

Then, $\vec{\fB} $ is an {\em $\isf$ ladder for $\delta$} if for every $u\in {\rm Break}(v)$, either
\begin{enumerate}[label=(\alph*)]

\item\label{itBud} $(\fB'_u,u)$ is a potential resolution for $\delta$, or

\item\label{itStrict} $(\fB_u,R_u) $ is either budding or progressive.

\end{enumerate}
If, moreover, all frames are irreflexive and $(\fB_v,R_v) = (\fB'_u,R')$ (i.e., $(\fB_v,R_v)$ is an initial potential resolution for $\delta$), then $\vec{\fB} $ is an {\em $\sf IGL$ ladder for $\delta$.}
\end{definition}

{\em Ladders} are chains of local potential resolutions, providing our means of climbing down from a defect to the root.

\begin{lemma}\label{lemmResInductive}
If $\Sigma\Subset\lanfull$ and $\delta=(v,\varphi)$ is a defect of a canonical  $\Sigma$-tree $\fA$, then there is a ladder for $\delta$.
\end{lemma}

\begin{proof}
We build $(\fB_u,R_u)$ inductively on the distance from $u$ to $v$ i.e., the maximal $n$ so that there is a chain $u=u_0\srel \ldots \srel u_n $ of elements of ${\rm Break}(v)$.
If $n = 0$, let $(\fB'_u ,R')$ be an initial potential resolution as given by Lemma~\ref{lemmInitialRes}, otherwise, $(\fB'_u ,R') = (\fB_{u'},R_{u'})$, which exists by induction hypothesis.
Let $t$ be a root for $\fB'_u$.

Using Lemma~\ref{lemmregionProps}, we see that $ \hat\rho( u)    \rel h(u') \llsim h(t)$, so that by Lemma~\ref{lemmClusConf}\eqref{itClusConfBack}, there is a cluster $C  \seq_{\rm c} \hat\rho( u)$ with $C \rel h(t)$.
In the $\sf IGL$ case, $\rho( u)$ is an irreflexive region, $\hat \rho( u) = \{u\} $, and we can let $C$ be a singleton chosen by backwards confluence.
Let $H = \{ x\in \dom{\fA_u} : \rho(x) = \rho(u)\} $ unless $\rho(u) = \{{\Perp}\}$, in which case, $H=[u]$, and consider the following cases.

\begin{Cases}
\item ($C\cong \hat\rho(u)$ and $u$ is a bud).\label{caseBud}
By Lemma~\ref{lemmBudCollapseTwo}, there is a $\Sigma$-intuitionistic relation $E $ on $\restr \fA u$ such that for every $x\in H$, there is $y \in [u]$ with $x\mathrel E y $. 
By Lemma~\ref{lemmBudGraft}, $\restr \fA u \peq  \fA \sqcup_u \fB'$ via $R_u = E \cup R'$.
We may thus set $\fB_u := \fA \sqcup_u \fB'_u $, so that \ref{itBud} holds in this case.

\item (otherwise).
We first replace $C$ by a finite $C'$.
If $  C \cong  \hat\rho (u) $, let $C'$ be a set of $\Sigma$-representatives of $\hat \rho(u)$ and we note that $ C' \rrel C \rel h(t)$ by Lemma~\ref{lemmregionProps}.
Otherwise, $  C \gg  \hat\rho (u) $, and we let $C'$ be a set of $\Sigma$-representatives of $C$.
Note that, in either case, $\hat\rho(u) \peq^\Sigma_{\rm c} C' \rel h(t)$.
Note also that if ${\sf L} = \sf IGL$, then both $C$ and $\hat\rho(u)$ are irreflexive singletons, hence, so is $C'$.

By Lemma~\ref{lemmBudCollapse}, there is a $\Sigma$-honest intuitionistic relation $E \subseteq H \times C'$  and hence by Lemma~\ref{lemmGraftExt}, we can extend $C'$ to a canonical sprout $\fC$ such that $\restr\fA u \peq \fC$ via some $E'$, then let $ \fB  := \fC \sqcup_u \fB' $.
By Lemma~\ref{lemmBudGraft}, $\restr \fA u \peq \fB $ via $R:=R'\cup E'$.

Let $r$ be a root of $\fB $ and $v'$ be the point of resolution of $\delta$.
Use Proposition~\ref{propFinal} to find $\Theta \ggsim C'$ which is $\llsim$-final such that $(\fB^+(v'),r) \subT \Theta$, where $\fB^+(v')$ is as given by Definition~\ref{defPosPart}.
By definition, there is a homomorphism $h' \colon \dom{\fB^+(v')} \to W_{\rm c}$ with $h'(r) = \Theta$.
We obtain $\fB_u$ from $\fB^+(v')$ by first replacing $h_{\fB^*}$ by $h'$ and then applying Lemma~\ref{lemmModalDefects} to resolve any modal defects.
Finally, we note that we still have $\restr\fA u \peq \fB_u$ via the same relation $R_u:= R$.

If $\Theta\gg h(u)$, then the local potential resolution $(\fB_u,R_u )$ is progressive.
Otherwise, as we are assuming that \ref{caseBud} does not hold,   $u$ is not a bud.
In this case, we chose $C'$ to be a set of $\Sigma$-representatives of $\hat \rho(u)$, so $r$ is a bud and $(\fB_u,R_u )$ is budding.\qedhere
\end{Cases}
\end{proof}

The conditions \ref{itBud} and \ref{itStrict} we identified are designed so we can always resolve defects in a `strategic' way, i.e., one that, as we will see, will ensure termination of the model-search.

\begin{definition}
Let $\Sigma\Subset\lanfull$ and $\delta=(v,\varphi)$ be a defect of a canonical sprout $\fA$ and $(\fB,w)$ be a potential resolution for $\delta$.
Let $u$ be a source for $(\fB,w)$ and $t$ be a target.
Then, we say that $(\fB,w)$ is {\em strategic} if $\rho (u ) \llsim \rho (t)$ and it is either
\begin{description}

\item[($\diamondsuit$)] {\em progressive:} $u\not\equiv w$ and $ \rho(u ) \ll \rho (t) $,

\item[($\heartsuit$)] {\em budding:} $u \not \equiv w$ and $u$ is not a bud but $t$ is,

\item[($\clubsuit$)] {\em blooming:} $w$ is a reflexive bud and $\rho(w) = \rho(v) = \rho(v') $, or

\item[($\spadesuit$)] {\em hatching:} $w$ is a reflexive bud, $u \equiv w$, and $ \rho(u ) \ll \rho (t) $.

\end{description}

\end{definition}

Note that we have, somewhat arbitrarily, assigned a poker suit to each defect type; this is not too relevant at this point, but will occasionally be used to link certain notations to a defect type.

\begin{proposition}\label{propExistsStrategic}
Let ${\sf L} \in  \{\ikf,{\sf IGL} \}$.
If $\Sigma\Subset\lanfull$ and $\delta=(v,\varphi)$ is a defect of a canonical  $\Sigma$-tree $\fA$, then there is $w\rrel v$ such that $w$ is a strategic potential resolution for $\delta$.
The potential resolution is live unless the defect is hatching, and if ${\sf L} = \sf IGL$, it is progressive or budding.

In particular, if $\vec{\fB}$ is a ladder for $\delta$ and there is any breaking point $u \rrel v$ such that $(\fB_u,R_u)$ is budding or progressive, then $w$ can be chosen so that $w\srel u$.
\end{proposition}

\begin{proof}
First assume that there is such a $u$ and choose it to be $\rrel$-minimal. Let $w$ be the predecessor of $u$ in ${\rm Break}(v)$.
According to Definition~\ref{defLadder}, $w$ is a bud, $\fB_w = {\restr{\fA}{w}} \sqcup_u \fB_u$, and $R_w$ is pivotal for $\fB_w$.
This implies that $(\fB_u,w)$ is a potential resolution for $\delta$ and, since $(\fB_u,R_u)$ is budding or progressive, either $u$ is not a bud and the resolution is budding or $\rho(u) \ll\rho(t) $ and it is progressive.
Note that when ${\sf L} = \sf IGL$, $(\fB_v,R_v)$ itself is always budding or progressive:
by definition of $\sf IGL$ ladders, $t$ is the root of $\fB_v$, and either $v$ is not a bud but $t$ is, so $(\fB_v,R_v)$ is budding, or else, $h(v)\neq h(t)$ (since $\delta $ is not a defect of $t$), which, since $\rho(v) \llsim \rho(t)$ and both regions are irreflexive, implies that $\rho(v) \ll \rho(t)$ and $(\fB_v,R_v)$ is progressive.

Otherwise, there is no such $u$ and ${\sf L} = \ikf$.
Let $w$ be a maximal element of ${\rm Break}(v)$.
Since $(\fB_w,R_w)$ is not strict, $w$ is a bud, $\fB_w = \restr{\fA}{w} \sqcup_u \fB'$, and $R_w$ is pivotal for $\fB_w$, where $(\fB',R')$ is an initial potential resolution for $\delta$, with resolution point $t$.
Note that in this case, $\rho(v) =\rho(w) $.
If also $\rho(t) = \rho(w)$, the resolution is blooming.
Otherwise, $\rho(v ) = \rho(w) \ll \rho(t) $, so it is hatching if $v\equiv w$ or progressive otherwise.

By the definition of a source, in all cases the resolution is witnessed by an intuitionistic order $\peq$ such that $x\prec y$ implies that $u \rrel x$, which moreover implies that $w\srel x$ except in the case where the resolution is hatching, and thus it is live in all other cases. 
\end{proof}

Aside from the general case described above, the following special case will pop up when working with $\ikf$.

\begin{corollary}\label{corStrategicSuccessor}
If $\Sigma\Subset\lanfull$ and $\delta=(v,\varphi)$ is a defect of a canonical $\ikf$ $\Sigma$-tree $\fA$ and $w\rel^1 w' \srel u \rrel v$ are such that $w$ and $u$ are breaking points, $w$ is a bud, and there are no breaking points in $(w,v)$ (hence, in particular, $\rho(w') = \rho(w)$), then there is a strategic potential resolution $w^*\rrel w$ for $\delta$.
\end{corollary}

\begin{proof}
Using Lemma~\ref{lemmResInductive}, let $\vec{\fB}$ be a ladder for $\delta$, so that either $(\fB_w,R_w) $ is strict or it is of the form $\fB_w = \restr{\fA}{w} \sqcup_u \fB_u$, with $R_w$  pivotal for $\fB_w$.
In the first case, we use Proposition~\ref{propExistsStrategic} to find suitable $w^*\srel w$.
Otherwise, using $w'$ as source and a root $t$ of $\fB_u$ as target, since $\rho(w') = \rho(w ) \ll \rho(u) \llsim \rho(t)$, we see that $(\fB_u,w)$ is a progressive potential resolution.
\end{proof}

\section{Stable Quasimodels}

At this point, we know that any defect may be resolved, suggesting a strategy for successively resolving the defects of an initial canonical sprout.
The issue is that defect resolution does not always commute, and, by resolving one defect $\delta$, we may `reactivate' another defect, $\delta'$, given that the intuitionistic order previously resolving it now fails to be expansive.
In order to remedy this, in this section, we will identify combinatorial conditions which ensure that this will never be the case.
Recall that we previously defined resolutions to be live or stubs; after a resolution $(\fB,w)$ is applied, $w$ will also be live or a stub, respectively.

\begin{definition}
Let $\Sigma\Subset\lanfull$ and $\fA$ be a canonical $\Sigma$-quasimodel.
We say that $w\in \dom{\fA}$ is {\em a stub} if there is $y\in\dom\fA$ with $w\prec y$ and $w\rel y$, otherwise, $w$ is {\em live.}
We say that $w $ is {\em admissible for $ \fA$} if every $w'\rrel w$ is live and, whenever $u \prec u'$, it follows that $ w\sqcap u \rrel u'\sqcap u  $.

We say that $\fA$ is {\em stable} if $\peq_\fA$ is pointwise pivotal and whenever $\delta = (v,\varphi)$ is a defect of $\fA$ and $u'\succ u$, it follows that $\delta$ has a  potential resolution $w$ which is admissible for $\peq_\fA$.

If $ (\fB,w)$ is such a potential resolution for $\delta$ with $\peq_\fB$ pivotal for $ \fA \sqcup_w \fB$, we define
${\peq_\fA} \sqcup_w {\peq_\fB} =  \bigcurlyvee \mathfrak R$, where $\mathfrak R$ is the set of all  pivotal relations $R\subseteq {\peq_\fA} \cup {\peq_\fB}$ which are $\Sigma$-intuitionistic on ${ \fA} \sqcup_w { \fB} $ (see Lemma~\ref{lemmExpJoin}) and equip $ \fA \sqcup_w \fB$ with ${\peq_\fA} \sqcup_w {\peq_\fB}$. 
\end{definition}

By always resolving defects using admissible resolutions, we can ensure that the original intuitionistic order continues to be expansive, as witnessed by the following.
Below, $\mathbb I_X$ denotes the identity on a set $X$.

\begin{lemma}\label{lemmIsOrd}
Let $\Sigma\Subset\lanfull$ and suppose that $\fA \sqcup_ w \fB$ can be grafted.
Suppose that $\peq_\fA$ is a pivotal relation on $\fA$ such that $w$ is admissible for $\peq_\fA$.
Then, $ {\peq_\fA} $ can be extended to a pivotal relation on $\fA \sqcup_ w \fB$.
\end{lemma}

\begin{proof}
Let $v$ be a pivot of $\peq_\fA$.
Let ${\peq} :={{\peq_\fA} \cup { \mathbb I_{\dom{\fB} }}}$.
We check that $\peq $ is a pivotal relation on ${\peq_\fA} \sqcup_w {\peq_\fB}  $.

To see that it is forward confluent, suppose that $x'\succ x \rel y$.
Then, $x\prec_\fA x'$, so $x',x\in \dom{\fA}$.
If $y\in\dom\fA$, we use the assumption that $\peq_v$ was forward confluent and use the same witness.
If $y\in\dom\fB$, then $x\rrel w$, so $ x = w\sqcap x \rrel x'\sqcap x \rrel x'  $; but, $x\neq x'$, so $x\rel x'$.
This implies that $x$ is a stub, contradicting the assumption that $w$ is admissible.

For backward confluence, suppose that $x \rel y \prec y'$.
If all three elements are in $\dom\fA$, we use the witness for $\peq_\fA$.
Now, for this to fail, we must have $y\in\dom\fA$ but $y'\in\dom\fB$, but this cannot occur by the way we defined $\peq$.
\end{proof}

When grafting a new sprout $\fB$ onto $\fA$, we can take the join of the preexisting $\peq_\fA$ with the new pivotal order $\peq_\fB$.
As the following lemma shows, this join can be computed in a straightforward way.

\begin{lemma}\label{lemmCompose}
Let $\Sigma\Subset\lanfull$ and suppose that $\fA \sqcup_ w \fB$ can be grafted.
Suppose that $\peq_\fA$ is a pointwise pivotal relation on $\fA$ such that $w$ is admissible for $\peq_\fA$ and $\peq_\fB$ is pivotal for $\fA \sqcup_w\fB$.
Then, ${\peq_\fA} \sqcup_w {\peq_\fB} = {\peq_\fB} \circ {\peq_\fA}$.
\end{lemma}

\begin{proof}
Since $\peq_\fA$ is a $\Sigma$-intuitionistic order on $\fA \sqcup_ w \fB$ by Lemma~\ref{lemmIsOrd}, it follows from Lemma~\ref{lemmCurlyvee} that $\peq_\fB \circ \peq_\fA $ is expansive.
So, in order to see that ${\peq } = { \peq_\fA \curlyvee \peq_\fB  }: = (  \peq_\fB \circ \peq_\fA )^*$, it suffices to show that $\peq_\fB \circ \peq_\fA  $ is transitive.

If $x_0 \mathrel{(\peq_\fB \circ \peq_\fA) ^2} x_2$, by definition, there is a sequence
\[   x _0 \peq_\fA  y_0 \peq_\fB x_1 \peq_\fA y_1 \peq_\fB  x_2. \]
From $y_0 \peq_\fB x_1$, we obtain $x_1 = y_0$, $x_1\in[w]$, or  $x_1\in \dom{\fB}$.

If $x_1 = y_0$, then
\[   x _0 \peq_\fA  y_0   \peq_\fA y_1 \peq_\fB  x_2. \]
By transitivity of $\peq_\fA$, $ x _0 \peq_\fA    y_1 \peq_\fB  x_2$ and $x_0 \mathrel{ \peq_\fB \circ \peq_\fA   } x_2$.
Since $w$ is admissible, $x_1$ is not a stub, so $x_1\in[w]$ entails $x_1 = y_1$.
Similarly, if $x_1\in \dom{\fB}$, we must have that $x_1=y_1 $.
We thus obtain
\[   x _0 \peq_\fA  y_0 \peq_\fB x_1   \peq_\fB  x_2 \]
which, now by transitivity of $\peq_\fB$, yields $x_0 \mathrel{ \peq_\fB \circ \peq_\fA   } x_2$ once again.
We conclude that $  \mathrel{ \peq_\fB \circ \peq_\fA   } $ is transitive, as claimed.
\end{proof}

It remains to check that $  \mathrel{ \peq_\fB \circ \peq_\fA   } $ is pivotal.
For this, it will be convenient to first establish some order-theoretic considerations.

\begin{lemma}\label{lemmJoins}
Let $\Sigma\Subset\lanfull$ and suppose that $\fA \sqcup_ w \fB$ can be grafted.
Suppose that $\peq_\fA$ is a pivotal relation on $\fA$ with pivot $w' $ such that $w$ is admissible for $\peq_\fA$ and $\peq_\fB$ is pivotal for $\fA \sqcup_w\fB$.
Suppose that $a,b,c$ are such that $a \peq _\fA b \peq_\fB c$ and $a\sqcap b = w'$.

Then, $w\sqcap w' \rrel a\sqcap c  $.
If, moreover, $a \prec _\fA b \prec _\fB c$, then either
\begin{enumerate}

\item $w\sqcap w' = w' = a \sqcap b = a\sqcap w = a \sqcap c$, or

\item $w\sqcap w' = w = b \sqcap c = w' \sqcap c  = a \sqcap c$.

\end{enumerate}

\end{lemma}

\begin{proof}
If $a=b$ then $w\sqcap w' \rrel w \rrel a \sqcap c  $ using the fact that $\peq_\fB$ is pivotal and if $b=c$ then $w\sqcap w' \rrel w' \rrel a \sqcap c  $ using the fact that $\peq_\fA$ is, so we may assume that $a \prec _\fA b \prec _\fB c$.

We have that $w'\rrel a\sqcap b$ and hence $w'\rrel b$ by the definition of a pivot.
Similarly, since  $b \prec_\fB c$, we also have $w\rrel b$, hence since $\fA\sqcup_w\fB $ is tree-like, $w \rrel w' $ or $w' \rrel w$.
If $w' \rel w$, then
\[w' = a \sqcap b = a\sqcap w = a \sqcap c. \] 			
Here, we are using a general property that on a tree, if $x\sqcap y \rel z\rrel y$, then $x\sqcap y = x\sqcap z$, and, similarly, if $x\sqcap z \rel z\rrel y$, then $x\sqcap y = x\sqcap z$.

By similar reasoning, if $w \rel w' $, then
\[w = b \sqcap c = w' \sqcap c  = a \sqcap c.\]
If $w' = w$, then  $ w = a\sqcap c$ because $a\in\dom\fA$ and $c\in\dom\fB$, so there can be no greater joint lower bound.
\end{proof}

\begin{lemma}\label{lemmCurlyvee}
Let $\Sigma\Subset\lanfull$ and suppose that $\fA \sqcup_ w \fB$ can be grafted.
Suppose that $\peq_\fA$ is a pointwise pivotal relation on $\fA$ such that $w$ is admissible for $\peq_\fA$ and $\peq_\fB$ is pivotal for $\fA \sqcup_w\fB$.
Then, ${\peq_\fA} \sqcup_w {\peq_\fB} $ is pointwise pivotal.
\end{lemma}

\begin{proof}
By Lemma~\ref{lemmCompose}, ${\peq}:= {\peq_\fA} \sqcup_w {\peq_\fB} = {\peq_\fB} \circ {\peq_\fA}$
Now, assume that $x \prec y$; we must find a relation with pivot $x\sqcap y$ witnessing $x \prec y$.
By Lemma~\ref{lemmIsOrd}, if $x \prec_\fA y$, we can continue using the pivot they had on $\fA$.
Similarly, if $x \prec_\fB y$, we may use the pivot $w$.
So, it remains to consider the case where there is some $z$ such that $x \prec_\fA z \prec_\fB y$.
Let $v= x\sqcap z$ and $\peq_v$ be a pivotal order witnessing for $x \prec_\fA z $ with pivot ${ v}$.
Let $ {\peq'} := {\peq_\fB} \circ{\peq _v}  $.
This is a $\Sigma$-intuitionistic order by Lemma~\ref{lemmCompose}.
By Lemma~\ref{lemmJoins},  $ w \sqcap v  =   x\sqcap y$, and clearly, $w\sqcap v$ is a pivot for $\peq'$.
\end{proof}

Now, we must pinpoint exactly which defect to resolve first.
It will be one with a {\em maximal admissible} potential resolution $w$, in the sense that $w$ is admissible and there is no admissible potential resolution $w' \sler_\fA w$.

\begin{proposition}\label{propStillStable}
Let $\fA$ be a stable canonical quasimodel and $w$ be a maximal admissible potential resolution for some defect $\delta$.
Let $(\fB ,  w)$ be a live potential resolution for $\delta$ and $\fA':=\fA\sqcup_w \fB$.
Then:
\begin{enumerate}

\item\label{itStillStableOne} If $\eta = (v ,\varphi )$ is any defect of $\fA'$ with $v\in\dom{\fA} $, then $\eta$ is a defect of $\fA$.

\item\label{itStillStableTwo} $\fA' $ is stable.
 
\end{enumerate}
\end{proposition}

\begin{proof}
Let ${\peq} = {\peq_\fA} \sqcup _w {\peq_\fB}$ which, by Lemma~\ref{lemmCurlyvee}, contains $  {\peq_\fA}$ and is pointwise pivotal.
The first item is then immediate, since any resolution for a defect in $\fA$ is also a resolution in $\fA'$.

Now, let $\delta = (v,\varphi)$ be any defect of $ \fA'$.
We must check that $\delta$ has an admissible resolution,  $c$.

If $v\in \dom{\fA}$ then, by the first item, $\delta$ was already a defect of $\fA$ and we let $(\fC,c)$ be an admissible potential resolution for $\delta$ on $\fA$.
By Lemma~\ref{lemmGraftExt}, there is a potential resolution $(\fC',c)$ for $\delta$ on $\fA'$, so $c$ is a potential resolution in this case.
Note that by choice of $w$, $w \srel c$ does not hold.
In the case where $v\in\dom{\fB}$, by Proposition~\ref{propExistsStrategic}, $\delta$ has at least one potential resolution, say $c$.
In this case, $c \rrel v$.
So, we have identified a potential resolution $c$ such that either $w  \srel c$ does not hold and $c$ is admissible on $\fA$, or else $c\rrel v$.

Now, suppose that $u \prec u' $, so that there is $u''$ with $u\peq_\fA u'' \peq_\fB u'$, with one of the two being strict.
We must check that $c\sqcap u \rrel  u'\sqcap u$
Let $\peq'_\fA$ be a pivotal relation with pivot $w':= u\sqcap u''$ witnessing $u\peq_\fA u''$ (with $\peq'_\fA$ being the identity when $u = u''$).

We must have that $u\in \dom{\fA}$, since no elements of $\fB$ are moved by $\peq_\fA$ nor $\peq_\fB$.
Similarly, $u\prec _\fA u''$ and $ u'' \prec_\fB u'$ both imply that $u''\in\dom\fA$, so $u'' \in \dom\fA$.

First assume that $w  \srel c$ does not hold and $c  $ is admissible for $\fA$.
If $u ' = u'' $, then $u\prec_\fA u''$ and since $c$ is admissible on $\fA$,  $c\sqcap u \rrel u' \sqcap u  $.
Otherwise, $u''\prec_\fB u'$, so $u'\in\dom\fB \cup [w]$ and $w\rrel u''$.
If $w\rrel u$, then $w\sqcap u =  w$ and thus $ w \sqcap u \srel c\sqcap u $ would yield $w \srel c $, a contradiction.
But $w \sqcap u , c\sqcap u \rrel u $, so the two are linearly ordered, and we must have $c\sqcap u  \rrel w\sqcap u \rrel u'\sqcap u$, the latter inequality following from $w\rrel u'$.
Note that this covers the case $ u=u'' $, since $w\rrel u'' $.
We are thus left with the case where $u\prec_\fA u'' \prec_\fB u'$.
By Lemma~\ref{lemmJoins}, either
$w\sqcap w' = w'  = u\sqcap u'  = u\sqcap u'' $ and $c\sqcap u \rrel u\sqcap u'' $
since $c$ is admissible for $\fA$, hence $c\sqcap u \rrel u\sqcap u' $, or else 
$w\sqcap w' = w = u' \sqcap u $.
In the latter case, we cannot have $ u' \sqcap u \srel c \sqcap u$, since this would entail $w\srel c$.
But again, $ u' \sqcap u $ and $ c \sqcap u$ are comparable, so $c \sqcap  u \rrel u'\sqcap u  $.

For the case where $v\in\dom\fB$ and $c\rrel v$, since $u\in\dom\fA$, $c\sqcap u \rrel v\sqcap u\rrel w$ and thus $c\sqcap u \rrel  w\sqcap u$.
Since $w$ was admissible for $\fA$, $w \sqcap u \rrel u''\sqcap u =  w'$.
Thus, $c \sqcap u  \rrel w$ and $c \sqcap u  \rrel w'$, so we may once again use Lemma~\ref{lemmJoins} to see that $c \sqcap u  \rrel w\sqcap w' \rrel u \sqcap u'$.
\end{proof}

\section{The Blooming Lemma}\label{sectBloom}

In this and the next section, we show how to resolve defects that occur when applying our selection method to $\ikf$; these two sections may be skipped by readers interested only in $\sf IGL$.
Progressive and budding defects can be resolved successively, since they introduce new regions that are above previous ones with respect to $\ll$ and, as we will see in Theorem~\ref{theoNoetherian}, we can take advantage of the fact that this order is Noetherian.
However, blooming resolutions do not have this feature.
Our strategy will thus be to resolve all of them simultaneously.
Below, $\overline\Sigma = \{\overline\psi:\psi\in \Sigma\}$.

\begin{definition}
Let $\Sigma\Subset\lanfull$, $\fA$ be a canonical $\Sigma$-tree, $w$ a bud of $\fA$, and $n\in\mathbb N$.
An {\em $n$-petal on $w$} is a canonical $\Sigma$-tree $\fB$ such that:
\begin{enumerate}[label=(\alph*)]

\item $ \fA\sqcup _w \fB $ can be grafted.

\item The width, height, and density of $\fB$ are bounded by $n$.

\item If the height of $\fB$ is $n$, then $\fB$ has a unique irreflexive root $r$.
In this case, $r$ has a representative set of successors $V$ such that there is an injection $\cdot_\psi\colon V \to \Sigma\cup\overline\Sigma\cup \{\top\}$ with $\psi \in \ell (u_\psi)$.
In this case, we say that $\fB$ is a {\em tall $n$-petal.}

\end{enumerate}
\end{definition}

Petals are useful for resolving blooming defects, although they must work together to do so, since any such defect occurring in a petal is resolved in another petal, which in turn may have defects resolved in the first.

\begin{lemma}\label{lemmPetal}
Let $\Sigma\Subset\lanfull$, $\fA$ be a canonical $\Sigma$-tree, $w$ a bud of $\fA$, and $\fB$ an $n$-petal with $n > 2\#\Sigma$ such that $\fA':= \fA\sqcup_w \fB$ can be grafted.

Suppose that there is an intuitionistic defect $\delta=(w,\varphi)$ of $\fA'$ such that $w$ is a blooming potential resolution for $\delta $.
Then, there is an $n$-petal $\fB'$ such that $(\fB',w)$ resolves $\delta$.

We can moreover resolve $\delta$ via a live $\Sigma$-intuitionistic order $\peq$ with the property that if $u\prec u'$, then $\rho(u) = \rho(w)$.
\end{lemma}

\begin{proof}
Choose a resolution link $\Phi \seq_{\rm c} h (v)$ for $(h(v),\varphi)$ and such that $\Phi \cong h(v)$ (i.e., $\rho(\Phi) = \rho(h(v))$).
Let $v'$ be a fresh node with $h(v') = \Phi$ and recall that $\restr\fB v = \fB{{\upharpoonright}} {\rel} (v) $.
For each modal defect $\ps\psi$ of $v'$ with $\psi\neq\top$, find $\Phi_\psi \ler \Phi$ with $\psi\in \Phi_\psi$.
Add a fresh node $v_\psi$ with $h(v_\psi)=\Phi_\psi$ and use Lemma~\ref{lemmModalDefects} to extend the singleton $\{v_\psi\}$ to a canonical $\Sigma$-tree $\fB_\psi$ of height, width, and density bounded by $2\# \Sigma$.

Now, consider two cases.
First, assume that $\restr\fB v$ has height less than $n$.
Then, graft $\restr\fB v$ and each $\fB_\psi$ to $ v' $ and call the result $\fB'$.
For $x\in\dom{\restr\fB v}$, we let $\tilde x$ be the copy of $x$ that was grafted to $v' $.
Using Lemma~\ref{lemmregionProps}, we see that $h(w) \rel_{\rm c} \Phi \rel_{\rm c} h(r)$, so we have a pivotal order $\peq$ on $\fA'\sqcup_w \fB'$ with $x\peq y$ if either
\begin{enumerate}

\item $x=y$ or

\item  $x\in\dom{\fB}$ and either
\begin{enumerate}

\item $x=v$ and $y=v'$,

\item $x\in\dom{\restr\fB v} $ and $y=\tilde x$, or

\item $ \rho(x) = \rho(w)$, $ y \equiv w$, and $x \mathrel E h  (y) $, where $E$ is an expansive, $\Sigma$-honest relation as given by Lemma~\ref{lemmBudCollapse}.

\end{enumerate}

\end{enumerate}
Clearly, $\peq$ resolves $\delta$.

\begin{claim}
The relation $\peq$ is expansive.
\end{claim}

\begin{proof}[Proof of claim.]
Let us check that it is forward confluent and suppose that $x'\seq x \rel y$.
We may assume that $x'\neq x$, since otherwise we can take $y'=y$, which moreover implies that $x\in\dom\fB$ and hence also $y\in\dom\fB$.
If $y\in\dom{\restr\fB v} $, we can uniformly take $y'=\tilde y$, which works if $x'=\tilde x$ but also works if $x'=v'$ or $x' \equiv w$.
If $y\notin\dom{\restr\fB v} $, then $y\mathrel E h(y') $.
But then, we also have $x \notin\dom{\restr\fB v} $, so $x\mathrel E x'$ and thus $x'\equiv y'$.

For backward confluence, suppose that $x\rel y\peq y'$.
If $x\rrel w$, then $x'=x$ will always work, so we assume otherwise.
If $\rho(x) \neq \rho(w)$, then also $\rho(y)\neq \rho(w)$ and so $y'=\tilde y$, so we may take $x'=\tilde x$.
If $\rho(x) = \rho(w)$, then any $x'$ with $x\mathrel E x'$ will do.
Thus $\peq$ is expansive, concluding the proof of the claim.
\end{proof}

We moreover note that $\fB'$ either has height at most $n$ or else it satisfies the properties required of a tall $n$-petal, setting $v_\top:= \tilde v$.

Otherwise, $\restr\fB v$ has height $n$, which means that it is a tall $n$-petal and hence has unique root $v$.
Let $\fB'$ be the result of replacing $v$ by $v'$ and then grafting $\fB_\psi$ onto $v'$ whenever $(v,\ps \psi) $ is a modal defect (omitting those that are already resolved in $\restr\fB v$).
In this case, we can let $\peq$ be the identity together with $v\prec v'$.
\end{proof}

Here, the reader may be concerned that, while one petal resolves the defects of another, there is no real `progress', as all petals share a single upper bound and they are all in the same region.
This is true, but what we {\em can} do is to add all petals at once and produce a  `flower' resolving each other's blooming defects.

\begin{proposition}\label{propBloom}
Let $\Sigma\Subset\lanfull$, $\fA$ be a stable canonical $\Sigma$-tree, and $w$ an admissible bud of $\fA$.
Then, there is a stable extension $\fA' $ of $\fA$ which is local to $w$ and such that $w$ is not a blooming potential resolution for any defect of $\fA'$.
\end{proposition}

\begin{proof}
Let $n = \max \{{\rm hgt}(\restr\fA w),{\rm wdt}(\restr\fA w),{\rm dns}(\restr\fA w), 2\#\Sigma+1 \}$, $\mathfrak G$ be the a set of representatives of all $n$-petals over $w$ up to isomorphism, and set $\fA' := \fA\sqcup_w \mathfrak G$.

Suppose that $\delta = (v,\varphi)$ is a defect of $\fA'$ such that $w$ is a blooming potential resolution for $\delta$.
By Lemma~\ref{lemmPetal}, $\delta$ is resolved in another $n$-petal, but this is just an element of $\mathfrak G$, so $\delta$ is already resolved on $\fA'$.

It remains to check that $\fA'$ is stable.
In view of Proposition~\ref{propStillStable}, we only need to consider defects $\delta=(v,\varphi)$ with $v\in\dom{\mathcal G}$ for some $\mathcal G \in \mathfrak G$ and show that there is an admissible potential resolution  $w'\rrel w$ for such $\delta$.

By Proposition~\ref{propExistsStrategic}, $\delta$ has at least one potential resolution, say, $w'$.
Suppose that $u\peq_{\fA'} u'$.
If $ u\sqcap v \rrel w$, then forcibly $u\sqcap w' \rrel w \rrel u\sqcap u'$, where the latter relation holds since $w$ was admissible if $u\peq_{\fA} u'$ and, otherwise, since we only added pairs to the intuitionistic preorder when both elements belonged to a petal or to $[w]$.
It follows that $w'$ is admissible unless $u,u'$ can be chosen with $w \srel u\sqcap v$.
 
According to Lemma~\ref{lemmPetal}, $\rho(u) =\rho(w) $, and thus, if $r$ is the root of $\mathcal G$, $w\rel^1 r$ and $\rho(r) =\rho(w) $.
If $\vec{\fB}$ is a ladder to $\delta$ as given by Lemma~\ref{lemmResInductive}, by  
Corollary~\ref{corStrategicSuccessor}, $w$ is a progressive potential resolution for $\delta$, as desired.
\end{proof}

We will call an extension as given by Proposition~\ref{propBloom} a {\em flower.}
With this, we now can remove every blooming defect that arises in our model search.

\section{The Hatching Lemma}\label{sectHatch}

The last type of defect we need to resolve are hatching defects.
These will be the last defects resolved at $w$ and, as with blooming defects, will all be resolved simultaneously and only occur when working with $\ikf$, so this section can be ignored by readers interested in $\sf IGL$.
We begin with a characterisation of hatching resolutions.

\begin{lemma}\label{lemmBinaural}
Let $\fA$ be a canonical tree and $ w $ be a bud of $\fA$.
Suppose that $\delta = (w,\varphi)$ is a defect of $\fA$ and $(\fB,w)$ is a hatching potential resolution for $\delta$.
Then, $h(w) \rel_{\rm c} h(\fB)$ and $h(w) \peq_{\rm c} h(\fB)$.
\end{lemma}

\begin{proof}
We have that $h(w) \peq_{\rm c} h(\fB)$ by definition of a potential resolution and $ h(w) \rel_{\rm c} h(\fB)$ by the definition of $\fA\sqcup_w\fB$.
\end{proof}

Let $\fA$ be a canonical quasimodel.
When resolving a hatching defect $(v,\varphi)$ of $\fA$, we thus do so by adding some $v'$ with $h(v) \rel_{\rm c} h(v')$ and $h(v) \peq_{\rm c} h(v')$.
It will also be convenient for $v'$ to be irreflexive.
Thus, we define $v \rel^\spadesuit_\fA v'$ (or $v \rel^\spadesuit v'$ when $\fA$ is clear from context) if $h_\fA(v) \rel_{\rm c} h_\fA(v')$, $h_\fA(v) \peq_{\rm c} h_\fA(v')$, $v\rel^1_\fA v'$, and $v'$ is irreflexive.
We let $\spadesuit_\fA(v)$, or simply $\spadesuit(v)$, be the set of $v'$ such that $[v]\rel^\spadesuit_\fA v'$ (in the sense that $v''\rel^\spadesuit_\fA v'$ for some $v''\in [v]$).
A hatching defect $  (v,\varphi)$ is {\em incubating} if it there is $v' \ler^\spadesuit v$ such that $h(v')$ is a resolution link for $ (h(v),\varphi) $; note, however, that incubating defects are not necessarily resolved by $\peq_\fA$.

Unlike other defect types, hatching defects are resolved in two phases: first, we incubate them, and only later  do we resolve them.
As with blooming defects, the resolution is done for all hatching defects at once, but, unlike blooming defects, we do not do so immediately after introducing the resolution links.
The following construction describes how these defects are resolved in this second phase.

\begin{lemma}\label{lemmnuclear}
Let $\fA$ be a canonical tree and $ w $  be a bud of $\fA$.
Then, there is a cluster $C^* \subseteq W_{\rm c}$ such that $[h(w)] \peq_{\rm c} C^*$ and, for every $v\in\spadesuit(w) $, $h(v)\rel C^*$ and $h(v)\peq C^*$.
\end{lemma}

\begin{proof}
If $\spadesuit(w) =\varnothing $, we may take $C^* = h([w])$, so we assume there is at least one $v_0\in \spadesuit(w)$, and let $w_0 \in [w]$ be such that $w_0 \peq v_0$.
Using Lemma~\ref{lemmClusConf}\ref{itClusConfFor}, let $C_0$ be such that $ h(v_0) \rel C_0 \seq h([w])$ and, using Proposition~\ref{propFinal}, let $C^*$ be $\peq$-Final with $ h(v_0) \rel C^* \seq C_0 $.

Now, consider possibly distinct $w_1,v_1$   such that $w\equiv w_1 \peq v_1 \in\spadesuit(w)$.
Since $C^* \seq h(w_1 ) \rel h(v_1)$, by forward confluence, there is $\Psi_1 $  such that $C^* \rel \Psi_1 \seq h(v_1)$.
We have that $h(w_1 ) \peq h(v_1) \peq \Psi_1$ and $h(w_1) \rel h(w_0) \rel h(v_0) \rel C^* $, so $\Psi_1 \seq h(w_1) \rel C^*$ and once again by Lemma~\ref{lemmClusConf}\ref{itClusConfFor}, there is $C_1$  such that $\Psi_1 \rel C_1 \seq C^*$.
But $h(v_0)  \rel C^* \rel \Psi_1 \rel C_1 $, so $ h(v_0) \rel C_1 \seq C^*$, and since $C_*$ was chosen $\peq$-Final, $C_1 \equiv C^*$.
It follows that $C^* \rel \Psi_1 \rel C^*$, which implies that $\Psi_1 \equiv C^*$, and hence we may assume $\Psi_1 \in C^*$, so that $h(v_1) \peq C^*$.

It remains to check that $h(v_1) \rel C^*$.
Since $ h(v_1) \seq h(w_1) \rel C^*$, by Lemma~\ref{lemmClusConf}\ref{itClusConfFor}, there is a cluster $D$ such that $  h(v_1) \rel D \seq C^*$.
Since $h(v_0) \seq h(w_0 ) \rel h(w_1) \rel h(v_1) \rel D $, there is $C''$ such that $ h(v_0) \rel C'' \seq D \seq C^*$.
By $\peq$-finality of $C^*$, $C'' \equiv C^* $, whence also $C^* \equiv D$, and $ h(v_1) \rel C^*$.
\end{proof}

This cluster $C^*$ we have found will help us resolve hatching defects in a very uniform way.
Below, recall that $\fA\sqcup \fB$ (without a subindex) indicates that the grafting is performed at the root.

\begin{definition}\label{defSpade} 
Let $\Sigma\Subset\lanfull$, $\fA$ be a canonical tree, and $ w $ be a reflexive bud of $\fA$.
A {\em nest for $w$} is a  canonical tree  $\fC^*$ such that, for every $ t \in \{w\} \cup \spadesuit(w)$,
\begin{enumerate}

\item $\restr{\fA}{t} \peq \fC^*$ via a root-preserving relation $E_t$, and

\item $ \fA \sqcup_t \fC^*$ can be grafted.

\end{enumerate}
\end{definition}

\begin{lemma}\label{lemmBouquetExists}
If $\Sigma\Subset\lanfull$, $\fA$ is a canonical tree, and $ w $ is a reflexive bud of $\fA$, then there is a nest for $w$.
\end{lemma}

\begin{proof}
Choose $C^*$ as in Lemma~\ref{lemmnuclear}.
By picking a set of $\Sigma$-representatives, we may assume that $C^*$ is finite.
We have that $h([w])\peq  C^* $ by Lemma~\ref{lemmBudCollapse}.
By Corollary~\ref{corRootPres}, we can extend $C^*$ to some $\fC_0$ such that $\restr{\fA}{w} \peq \fC_0$ via a root-preserving relation $E_0$.
If we enumerate $\spadesuit(w)$ by $(t_i)_{i<I}$, from $ h(t_0) \peq C^*$, we may similarly extend $\fC_0$ to some $\fC_1$ so that $ \restr{\fA}{t_0} \peq \fC_1 $, and continue inductively to obtain $\fC^* = \fC_I$ with the property that $\restr{\fA}{w} \peq \fC^*$ and, for each $t\in \spadesuit(w)$, $\restr{\fA}{t } \peq \fC^* $ via a root-preserving relation $E_t$.
\end{proof}

Nests are used to resolve all hatching defects simultaneously, as witnessed by the following lemma.
It is worth pointing out that $w$ will become a stub after the construction below.

\begin{lemma}\label{lemmCoreRes}
Let $\Sigma\Subset\lanfull$, $\fA$ is a canonical tree, $ w $ be a reflexive bud of $\fA$, $\fC^*$ be a nest for $w $, and $\fA'$ be the result of successively grafting $ \fC^*$ onto each $t\in \spadesuit(w)$.
Then, whenever for each $ t \in \spadesuit(w)$, there is an $\Sigma$-intuitionistic relation ${\peq}={\peq_t}$ on $\fA '$ with pivot $w$ and such that if $v\in[w]$ is such that $h(v) \peq h(t)$, then $v\peq t$.
\end{lemma}

\begin{proof}
Fix $t$ and $v$ as in the statement of the lemma.
For $s \in \spadesuit(w) $, let $\fC^*_s$ be the copy of $\fC^*$ that is grafted onto $s$ and, for $x\in\dom{\fC^*}$, let $x_s$ be the copy of $x$ in $\fC^*_s$.

By definition of a nest, for each $x\in \{w\}\cup \spadesuit(w)$, $ \restr{\fA}{x} \peq \fC^*$ via a root-preserving relation $E_x$.
Let $f_t$ be a function that maps an element of $\fC^*_s$ to its copy in $\fC^*_s$.

We define $ {\peq } = {\peq_\delta} $ on $\fA' $ by
\[{\peq} = \{(v,t)\} \cup \mathbb I_{\dom{\fA'}} \cup f_t \cup \bigcup_{s\in \{w\} \cup \spadesuit(w)} f_t\circ E_s.\]
Clearly, $\peq$ resolves $\delta$ and has pivot $w$, provided it is expansive.
Note that each of $\mathbb I_{\dom{\fA'}}$, $f_t$, and $E_s$ are expansive over their domain and, by Proposition~\ref{propExpandingClos}, so is $f_t \circ E_s$.

Let us check that $\peq$ is forward confluent.
Suppose that $x'\seq x \rel y$.
If $x\mathrel R x'$ for $R\in \{\mathbb I_{\dom{\fA'}} ,  f_t ,  f_t \circ E_s \}$ and $y$ is also in the domain of $R$, then we appeal to the forward confluence of $R$, so we assume otherwise.
By inspection, this can occur if $x\in [w] \cup \spadesuit (w)$, in which case, $x'\in [w] \cup \{t \} \cup   C^*  $, where $C^*$ is the root cluster of $\fC^*_t$.
Then, if $y\in\dom{\fC^*_\eta}$, we take $y' = f_t(y)$, and if $y\in \dom{\restr\fA s}$, $y' = f_t(z)$, where $z$ is chosen so that $ y\mathrel E_s z $.

For backward confluence, suppose that $x\rel y\peq y'$.
Reasoning as above, we may ignore the cases where $y\mathrel R y'$ for $R\in \{\mathbb I_{\dom{\fA'}} ,  f_t ,  f_t \circ E_s \}$ and $x$ is also in the domain of $R$.

If $x \rrel w$ fails, we can take $x = x'$.
By inspection, we are thus left with the case where $y \in \dom{\fC^*_s}$, $y'=f_t(y)$, and $x = s \in \spadesuit (w)$.
We can then choose $x'\in C^*_t$ such that $x\peq x'$.
\end{proof}

This uniform resolution of hatching defects will allow us to resolve them all simultaneously, `when the time is right'.
One way to think about this construction is that, given a bud $w$ with some hatching defects, we place the resolution points of all such defects as immediate $\prec$-successors of $w$, then place them all beneath $\fC^*$.
As it is generally more convenient to work with trees, we will instead give each one its own copy of $\fC^*$, but then have $\peq$ identify equivalent points in all of these copies.
Note that this step will force $\peq$ to be a preorder, rather than a partial order.

\begin{proposition}\label{propnuclear}
Let $\Sigma\Subset\lanfull$ and $\fA$ be a stable, canonical $ \Sigma$-tree.
Let $w\in\dom\fA$ be a bud and assume that if $\delta$ is any defect of $\fA$ with potential resolution $w'\rler w$, then $\delta$ is an incubating hatching defect of $[w]$.

Then, there exists a stable extension $\fA'$ of $\fA$ local to $w$.
\end{proposition}

\begin{proof}
Using Lemma~\ref{lemmBouquetExists}, let  $\fC^*$ be a nest for $w$, and let $\fA'$ be obtained by grafting $\fC^* $ onto every element of $\spadesuit (w)$.
If $\delta = (v,\varphi)$ is a defect with resolution $w'\rler w$, then by assumption, $v \equiv w$ and $\delta$ is an incubating hatching defect, hence it has a resolution link $t \ler^\spadesuit v$.
The relation $\peq_t $ as given by Lemma~\ref{lemmCoreRes} witnesses that $\delta$ is not a defect of $\fA'$.
 
What remains to be shown is that any defect of $\fA'$ has an admissible potential resolution, which essentially boils down to showing that there is always a potential resolution distinct from $w$, which is now a stub.

Let $\delta=(v,\varphi)$ be a defect of $\fA'$ such that $w \rrel v$.
If $v\in\dom\fA$, then by assumption, $\delta$ has a potential resolution $w'\srel w$, which, by Proposition~\ref{propStillStable} (applied iteratively for each $s \in \spadesuit(w)$),  is still a potential resolution for $\delta$ in $\fA'$.
Otherwise, $c \rrel v$, where $c$ is the root of some $\fC^*_{s}$.
Thus the ladder for $\delta$ involves some local potential resolution $\fD$ for $\delta$ at $c$.
For the root $r$ of $\fD $, we have $h( w )  \llsim h(c) \llsim h(r) $ and $\restr\fA w \peq \fD$ via the relation $E_w$ of Definition~\ref{defSpade}, so $\fD$ is a strict local resolution for $\delta$ at $w$.
By Proposition~\ref{propExistsStrategic}, there is a potential resolution $w'\srel w$ for $\delta$.
\end{proof}

We call an extension as given by Proposition~\ref{propnuclear} a {\em nest extension.}

\section{Well Foundedness}

The process of resolving defects typically leads to ever-larger canonical trees.
We thus cannot use standard measures such as the size of the tree or the number of defects to prove that eventually we will reach a $\Sigma$-model, i.e., a canonical tree without intuitionistic defects.
Instead, we will need to assign a transfinite number to each canonical sprout and show that this number decreases during the resolution process.
This is possible due to some deep order-theoretic results, which we review in this section before applying them to canonical trees.

\begin{definition}
Let $\Lambda= (\dom \Lambda, \leq )$ be a preorder.
\begin{enumerate}

\item $\Lambda$ is a {\em well quasiorder} if, whenever $(\lambda_i)_{i<\omega}$ is a sequence of elements of $\Lambda$, there are $i\leq j$ such that $\lambda_i\leq\lambda_j$.

\item  $\Lambda$ is {\em Noetherian} if there is no infinite, strictly increasing sequence on $\Lambda$.

\item $\Lambda$ is a {\em well order} if $\leq$ is a linear order and either $\Lambda$ is a well quasiorder or, equivalently, $(\dom\Lambda,\geq)$ is Noetherian.

\end{enumerate}
\end{definition}

The crucial property of $\subT$ is that it defines a well quasiorder on the class of $\Sigma$-sprouts.

\begin{theorem}\label{theoKruskalSprout}
Fix finite $\Sigma\subseteq \lanfull$.
If  $(\fA_i)_{i<\omega}$ is a sequence of $\Sigma$-sprouts, then there are $i<j$ such that $\fA_i \subT \fA_j$.
\end{theorem}

\begin{proof}
This is by now standard (see e.g.~\cite{pml}).
We may view $\Sigma$-labelled trees as trees of clusters, where each cluster $C$ is labelled by $\{\ell(w):w\in C\}$.
Then, apply Kruskal's theorem.
\end{proof}

As a consequence of this, the set of $\Sigma$-regions is Noetherian.

\begin{theorem}\label{theoNoetherian}
Fix $\Sigma\Subset \lanfull$ and let $\mathfrak R_\Sigma$ be the set of $\Sigma$-regions of $\cl M_{\rm c}$.
Then, $(\mathfrak R_\Sigma,\lleq)$ is Noetherian.
\end{theorem}

\begin{proof}
Let $(\mathfrak r_i)_{i<\omega}$ be an infinite increasing sequence of regions.
For each $i$, let $\fA_i$ be a $\Sigma$-sprout so that $\mathfrak r_i$ is $\Sigma$-Final.
Using Theorem~\ref{theoKruskalSprout}, there are $i<j$ such that $\fA_i \subT \fA_j$; but this contradicts $\fA_i$ being $\llsim $-Final for $\fA_i$, since any $\Phi\in \mathfrak r_j$ such that $\fA_j\subT \Phi$ also satisfies $\fA_i \subT \Phi$ by the transitivity of $\subT$.
\end{proof}

If $(N,\leq)$ is any Noetherian partial order, we can assign to every $x\in N$ an ordinal $\| x\|$ such that $x<y$ implies that $\| x\| > \| y\|$.
To this end, we recall that $\sf Ord$ is a class of linear orders with the property that every well order is isomorphic to a unique $\xi\in \sf Ord$.
The class $\sf Ord$ is itself well ordered by embeddability, which we write $ \leq  $.
Moreover, $\varnothing $ is an ordinal (usually denoted $0$), every ordinal $\xi$ has a unique succesor $\xi+1$, and, if $\Xi$ is a set of ordinals then $\bigcup \Xi$ is an ordinal.
With this in mind, we may define the following.

\begin{definition}
Let $\Lambda $ be a Noetherian partial order.
For $\lambda\in \dom\Lambda$, we define
\[\|\lambda\|=\bigcup_{\lambda<\eta} (\|\eta\|+1).\]
\end{definition}

Using the assumption that $\Lambda$ is Noetherian, it is readily checked that this definition is sound and $\|\lambda\|$ is always an ordinal.
We are in particular interested in the case where $\Lambda=\mathfrak R_\Sigma$.

Let us review the ordinal arithmetic we will use.
We will in particular need to extend elementary operations to the ordinals.
Given ordinals $\xi,\zeta$, we define $\xi+\zeta$ by recursion on $\zeta$ as follows:
\begin{enumerate}

\item $\xi+0=\xi$

\item $\xi+(\zeta+1)=(\xi+\zeta)+1$

\item $\xi+\zeta=\displaystyle\bigcup_{\vartheta<\zeta}(\xi+\vartheta)$, for $\zeta$ a limit ordinal.

\end{enumerate}
Likewise, we define $\xi\cdot\zeta$ by
\begin{enumerate}

\item $\xi\cdot 0=0$,

\item $\xi\cdot(\zeta+1)=\xi\cdot\zeta+\xi$, and

\item $\xi\cdot\zeta=\displaystyle\lim_{\vartheta<\zeta}\xi\cdot\vartheta$, for $\zeta$ a limit ordinal.

\end{enumerate}
Similarly, we define $\xi^ \zeta$ by:

\begin{enumerate}

\item $\xi^0=1$,

\item $\xi^{\zeta+1}=\xi^\zeta\cdot\xi$, and

\item $\xi^\zeta=\displaystyle\lim_{\vartheta<\zeta}\xi^\vartheta$, for $\zeta$ a limit ordinal.

\end{enumerate}

The order-type of the natural numbers is denoted $\omega$ and is the smallest infinite ordinal.
This will typically be our base for exponentiation.
Numbers of the form $\omega^\xi$ are {\em additivey indecomposable,} in the sense that $\alpha,\beta<\omega^\xi$ implies that $\alpha+\beta<\omega^\xi$.
The {\em Cantor normal form} theorem states that every ordinal may be uniquely written in the form $\omega^{\alpha }    + \beta $, where $\beta<\omega^ {\alpha}$.
The {\em natural sum} $\oplus$ is a variant of $+$ which is always computed in Cantor normal form.
Formally, it is given by $\al \oplus 0 = 0 \oplus \al = \al $ and
\[
(\om^\al+\be) \oplus (\om^\ga+\de)=
\begin{cases}
 \om^\al+(\be \oplus (\om^\ga+\de)) &\text{if $\ga \leq \al$,}\\
\om^\ga+ ((\om^\al+\be) \oplus  \de) &\text{otherwise,}
\end{cases}
\]
where $\om^\al+\be$, $\om^\ga+\de$ are in Cantor normal form.

We will make use of   {\em multiplicativey indecomposable numbers,} which are those of the form $\omega^{\omega^\xi}$ and have the additional property that $\alpha,\beta<\omega^{\omega^\xi}$ implies that $\alpha\cdot\beta<\omega^{\omega^\xi}$.
As these numbers are quite important to us, we introduce the notation $\mu(\xi):= \omega^{\omega^\xi}$.
We say that $\xi$ is a {\em multiple} of $\zeta$ if there is $\eta$ such that $\xi = \zeta\cdot \eta$ and note that if $\xi<\zeta$, then $\omega^\zeta$ is a multiple of $\omega^\xi$ and hence $\mu(\zeta)$ is a multiple of $\mu(\xi)$.

\section{Termination}

We finally have all the tools needed to prove that the defect-resolution procedure terminates.
Our argument will be by transfinite induction, using the order-theoretic machinery of the last section.

\begin{definition}
Let $\Sigma\Subset\lanfull$ and $\fA$ be a canonical $\Sigma$-tree.
For $w\in \dom{\fA}$, define $\|w\| =  \min \{ \|\rho(x)\|: x\rrel w\}$.

We let 
\begin{itemize}

\item $\partial ^\clubsuit _\fA  w  $ (or $\partial^ \clubsuit  w $ if $\fA$ is clear from context) be the number of blooming defects of the form $(w,\varphi)$ with $\varphi\in\Sigma$ which are not resolved by $\peq_\fA$,

\item $\partial ^\spadesuit _\fA  w  $ or $\partial^\spadesuit  w $ be the number of hatching defects of the form $(w,\varphi)$ with $\varphi\in\Sigma$ which are not incubating,

\item and $\partial ^\circ _\fA  w  = \partial _\fA  w  $ be the number of progressive, budding, or hatching defects that are not resolved.

\end{itemize}
Then, set $\partial_\fA w  =\partial w  =  \omega \partial ^\clubsuit   w  +\partial ^\spadesuit   w + \partial ^\circ _\Sigma  w $.
Recall that $ x \rrel^\bullet y $ if and only if $ x \rrel  y $ but $\rho(x) \ggeq \rho(y)$, then set $\nabla_\fA(w) = \nabla (w) := \bigoplus_{w\rrel^\bullet_\fA v} \partial  (v)$.

Define $ \beta(w) = 0 $ if there is $w' \rrel^\bullet w$ such that $w'$ is a bud, otherwise, $ \beta(w) = 1 $.

Then, define $\alpha_\fA(w) =  \alpha(w) $ and $o_\fA(w) =  o(w) $ by
\begin{align*}
  \alpha(w)  & : =    \omega^3 \|w\|+\omega^2 \beta(w) + \nabla (w) \\ 
    \gamma(w)  & : =   1 + \bigoplus_{w\srel v}   o(v) \\ 
 o(w)  & : =  \mu  ( \alpha(w) ) \cdot   \gamma(w)
\end{align*}
and $  o(\fA) : = o(r)$, where $r$ is a root of $\fA$.
\end{definition}

The intuition behind the definition of $\partial w$ is as follows.
First, blooming defects are all resolved simultaneously but, in doing so, we may introduce an arbitrary number of new (non-blooming) defects in the same region.
The term $\omega \partial ^\clubsuit   w$ will thus disappear when resolving blooming defects, but be replaced by an arbitrarily large natural number contributing to $ \partial ^\circ   w$; the factor $\omega$ ensures that the resulting ordinal will be smaller.
Meanwhile, hatching defects are `counted twice': as long as they are not resolved, they are counted by $ \partial ^\circ   w$, but when they are not even incubating, they are counted by $ \partial ^\clubsuit   w$, thus contributing $2$ to the sum.
This allows us to resolve them in two phases, first incubating them, and then resolving them with a nest, while decreasing $\partial w$ in each phase.

Next, we list some useful properties of our ordinal mapping.

\begin{lemma}\label{lemmOmon}
If $\Sigma\Subset\lanfull$, $\fA$ is a canonical $\Sigma$-tree, and $w \rel _\fA v\in \dom{\fA }$, then

\begin{enumerate}[label=(\arabic*)]

\item if $w \equiv v$ then $o(w) = o(v)$,

\item \label{itOMuMult} $o(w)$ is a multiple of $\mu(\omega^3 \|w\| )$,

\item\label{itAlphMon} if $w \srel v$ then $\alpha (w) \geq \alpha (v)$,

\item \label{itOPlusOne} $o(w) <\mu(\alpha (w) +1)$,

\item\label{itOmMono} if $w \srel v$, then $o(w) > o(v)$, and

\item \label{itOMonMu} if $\|w\| > \|v\|$, then $\mu(\omega^3 \|w\| ) > o(w)$.

\end{enumerate}
\end{lemma}

\begin{proof}
The first two items follow easily from the definitions and properties of $\mu$.
Claim \ref{itAlphMon} follows readily from the definition of $\alpha (w)$, and we observe that \ref{itOMonMu} follows from \ref{itOPlusOne}, so it remains to prove \ref{itOmMono} and \ref{itOPlusOne}.

We proceed to prove them by simultaneous induction on the depth of $w$.
If $w\srel v$, then by \ref{itAlphMon}, $\alpha(v) \leq \alpha(w)$, and, by the induction hypothesis applied to \ref{itOPlusOne}, $o(v)<\mu(\alpha(v) + 1) \leq \mu(\alpha(w) + 1)$.
Since $\mu(\alpha(w) + 1)$ is additively indecomposable, it follows that
\[\gamma(w) = 1 + \bigoplus_{w\srel v}   o(v)< \mu(\alpha(w) + 1).\]
Since $\mu(\alpha(w) + 1)$ is the least multiplicatively indecomposable above $\mu(\alpha(w)  )$, we conclude that
\[o(v) \leq \gamma(w) < \mu(\alpha(w)  ) \cdot \gamma (w)<\mu(\alpha(w) + 1) ,\]
thus establishing both claims.
\end{proof}

It remains to show that we can resolve defects in a way that decreases the ordinal assignment.

\begin{lemma}\label{lemmInductive}
Let $\fA$ be a stable, canonical $\Sigma$-quasimodel and suppose that $\fA $ has at least one defect.
Then, there is an extension $\fA'$ of $\fA$ such that $o(\fA') < o(\fA)$.
\end{lemma}

\begin{proof} 
Since there is at least one defect, we can pick $\delta=(v,\varphi)$ so that $w$ is maximal admissible and such that $\delta$ is an incubating hatching defect only if no defects of other types are available.
By Proposition~\ref{propExistsStrategic}, there is a strategic potential resolution $(\fB',w)$ for $\delta$.
Let us define $\fB$ according to the following cases:
\begin{enumerate}

\item If $(\fB',w)$ is progressive, budding, or hatching but not incubating, then $\fB = \fB'$.

\item If $(\fB',w)$ is blooming, then $\fB$ is a flower resolving all blooming defects as given by Proposition~\ref{propBloom}.

\item If $(\fB',w)$ is an incubating hatching defect, then $\fB$ is obtained by grafting a nest to all elements of $\spadesuit(w)$, as given by Proposition~\ref{propnuclear}.

\end{enumerate}
Let $\fA'=\fA\sqcup_w \fB$, $o :=o_\fA$, $o' := o_{\fA'}$, and define $\nabla,\nabla',\beta,\beta'$ analogously.
By Proposition~\ref{propStillStable},~\ref{propBloom}, or~\ref{propnuclear}, $\fA'$ is stable.
Recall that
\[  o(w)   =  \mu \big ( \omega^3\cdot \|w\|+\omega^2 \beta(w) + \nabla (w) \big ) \cdot    \big (1 + \bigoplus_{w \srel v}   o(v) \big ). \]
We prove by induction on the depth of $x \in \dom{\fA}$ that $o'(x) \leq o(x) $, $\beta'(x) \leq \beta(x) $, $\nabla'(x) \leq \nabla (x) $  and if $x \rrel v$, then $o'(x) < o(x) $.
Note that $\nabla' x  \leq \nabla' x$, since $\fA'$ does not introduce new defects to elements of $\dom{\fA}$ by Proposition~\ref{propStillStable}\eqref{itStillStableOne}.
If $x \equiv v $ and $v$ is not a hatching defect, then $\partial ' v  < \partial  v$, since we have resolved the defect $\delta$.
From this, we obtain $\nabla ' x  < \nabla   v$, which together with the induction hypothesis for $y\srel v$ yields $o'(x)<o(x)$.

If $x\srel v$, then choosing $u$ so that $x \rel^1 u \rrel v$, we can apply the induction hypothesis to obtain $o'(u)<o(u)$.
This yields $o'(x)<o(x)$, except in the critical case where $x\equiv w$, in which we have grafted one or more new subtrees.

Here, we consider four cases.
We remark that only the first two may occur when the logic is $\sf IGL$.
\begin{Cases}

\item ($(\fB,w)$ is progressive).
In this case, we have added a subtree $\fB$ with $\rho( u) \ll \rho(\fB)$, so $\|\fB\|< \|u\|$ (where, as usual, $\|\fB\|$ indicates that it is evaluated at a root).
By Lemma~\ref{lemmOmon}\ref{itOMonMu},  $o(\fB) < \mu( \om^3\| u\|) $, and by Lemma~\ref{lemmOmon}\ref{itOMuMult}, both $o(u)$ and $o'(u)$ are multiples of $\mu( \om^3\| u\|)$, so $o'(u)\oplus o'(\fB) < o(u)$, yielding $o'(x)<o(x)$.

\item ($(\fB,w)$ is budding).
As before, $o'(u)<o(u)$ and both are multiples of $\mu( \om^3\| u\|)$.
We have once again grafted a new subree $\fB$ with $\beta(\fB) = 0$, whereas $\beta(u) = 1$, and since $\mu:= \mu \big ( \omega^3\cdot \|w\|+\omega^2 \beta(w) + \nabla (w) \big )$ is multiplicatively indecomposable, $o'(\fB) <\mu $.
We thus have, once again, that $o'(u)\oplus o'(\fB) < o(u)$, yielding $o'(x)<o(x)$.

\item ($(\fB,w)$ is blooming).
According to Lemma~\ref{lemmPetal}, $x$ does not have any blooming resolutions in $\fA'$, so whenever $y\rler^\bullet_{\fA'} w$, it follows that $\partial' (y)<\omega$ and hence $\nabla' x <\omega$.
We thus have that
\[o'(x) < \mu \big (\om^3 \|x\| + \omega  \big ) \leq \mu \big (\om^3  \|x\| + \nabla x  \big  ) \leq o(x) .\]

\item ($(\fB,w)$ is hatching).
Here, we consider two sub-cases.

\begin{Cases} 

\item ($(\fB,w)$ is not incubating).
In this case, $v\equiv w$ and we have added a subtree $\fB$ with $\rho(v) \ll \rho(\fB)$, so $\|\fB\|< \|v\| = \| x \|$.
Meanwhile, $\partial^\clubsuit v $ has decreased  and hence $\nabla ' x < \nabla  x $.
By Lemma~\ref{lemmOmon}\ref{itOMonMu},
\[o(\fB) < \mu( \om^3\| x \|) \leq \mu(\alpha'(x)) < \mu(\alpha(x)).\]
Note that if $y\sler_\fA x$, then $\|y\| \leq \|x\|$ and, if equality holds, then $\beta(y) \leq \beta(x)$ by definition and $ \nabla y<\nabla x$ since $\delta$ does not contribute to $ \nabla y$. Thus, $\alpha(y) <  \alpha (x)  $  and   $o(y) <\mu(\alpha(x))$ by Lemma~\ref{lemmOmon}\ref{itOPlusOne}.
Since the latter is multiplicatively (hence, additively) indecomposable, $\gamma(x) <\mu(\alpha(x))$ and therefore 
\[o'(x) = \mu(\alpha'(x)) ( \gamma(x) \oplus o'(\fB) ) <  \mu(\alpha (x)) \leq o(x). \]

\item ($(\fB,w)$ is incubating).
We have grafted subtrees $\fC^*_t$ to each $t\in\spadesuit(w)$.
We have that $\rho(w)\ll \rho(t) \lleq \rho(\fC^*_t)$, so these new subtrees do not contribute to $\nabla ' x$ and, since there was at least one hatching defect resolved, $\nabla ' x < \nabla x$.
Thus, we have that
\[o'(x) <  \mu \big (\om^3  \|x\| + \nabla x  \big  ) \leq o(x) . \qedhere \]

\end{Cases}

\end{Cases}
\end{proof}

As long as we can reduce the ordinal assignment, we can guarantee that the defect resolution process will terminate in finite time, since there are no infinite decreasing sequences of ordinals, thus yielding the following.

\begin{proposition}\label{propRemoveDefect}
If $\Sigma\Subset\lanfull$ and $\fA$ is any stable, canonical $\Sigma$-quasimodel, then $\fA$ can be extended to a $\Sigma$-model $\fA^*$.
\end{proposition}

\begin{proof}
By induction on $o(\fA)$; if $\fA$ has no defects then we can take $\fA^* = \fA$, otherwise, by Lemma~\ref{lemmInductive}, there is  $\fA' \supseteq_{\rm c} \fA $ such that $\fA'$ is stable and $o(\fA') < o(\fA)$.
By the induction hypothesis, $\fA'$ (and hence $\fA$) can be extended to a $\Sigma$-model, as desired.
\end{proof}

\begin{theorem}
$\sf IK4$ has the finite model property.
\end{theorem}

\begin{proof}
Suppose that ${\sf IK4} \not \vdash \varphi$ and let $\Phi \in W_{\rm c}$ be such that $\overline \varphi \in \Phi $, which exists by the Lindenbaum lemma (\ref{lem:lindenbaum}) and Lemma~\ref{lemCompleteType}.
Define a canonical sprout with root $\Perp$ followed by a node labelled by $\Phi$, then extend it to a canonical $\Sigma$-tree $\fA$ using Lemma~\ref{lemmModalDefects}, and note that $\fA$ is rooted.
Letting $\peq_\fA$ be the identity, we see that $\fA$ is trivially stable.
Hence by Proposition~\ref{propRemoveDefect}, $\fA$ can be extended to a $\Sigma$-model $\fA^*$, thus providing a finite countermodel for $\varphi$.
\end{proof}

\begin{theorem}
$\sf IGL$ is complete for the class of $\sf IGL$ models.
\end{theorem}

\begin{proof}
The argument is similar to the previous, except that we have to ensure that the resulting model is an $\sf IGL$ model.
Assume that  ${\sf IGL}\nvdash \varphi $ and let $\fA_0$ be a canonical $\Sigma$-tree falsifying $\varphi$, then let $ (\fA_i)_{i\leq n} $ be a sequence of extensions as given by Lemma~\ref{lemmInductive}, where $\fA_n$ is a model falsifying $\varphi$.
In view of Proposition~\ref{propExistsStrategic}, all defects are progressive or budding, and hence we have the property that if $x\prec y'$, then there is some $i$ such that $x\in \dom{\fA_i} $ and $y\in \dom{\fA_{i+1}} \setminus \dom{\fA_i} $, since all resolutions are strict.
Now, assume that $x_0\ll x_1\ll\ldots$ is an infinite sequence of elements of $\fA_n$.
Let $k(i)$ be the least $k$ so that $x_i \in \dom{\fA_k}$ and note that $k(i)$ is non-decreasing, since we never add new predecessors with respect to $\peq$ or $\rel$.
We cannot have that $x_i \rel x_{i+1}$ for all $i>i_0$, since $\dom{\fA_n}$ is finite and $\rel$ is an irreflexive partial order.
We thus have that $x_i \prec x_{i+1} $ for infinitely many $i$.
But, as we have observed, this implies $k(i) < k(i+1)$, which cannot be, since $n$ is finite.
\end{proof}

We remark that the finite model property for $\sf IGL$ also follows from this argument, although this had already been proven previously~\cite{FDIGL}.

\subsection*{Acknowledgments}
I am grateful to Oriola Gjetaj for convincing me to pursue this work after more than a decade of failed attempts and to Ian Shillito for his interest in formally verifying the proof.
I am also grateful to Juan Aguilera, Nick Bezhanishvili, Anupam Das, Alakh Dhruv Chopra, Martín Diéguez, Marianna Girlando, Tim Lyon, Brett McLean, George Metcalfe, Mojtaba Mojtahedi, Leonardo Pacheco, and Lutz Stra{\ss}burger for enlightening discussions.
Many of these discussions were long ago and may have already been forgotten by some of you.
I apologise if I've forgotten some discussions in turn.
But I do look forward to many more in the future, albeit hopefully on a different topic!

\subsection*{Funding information}
This work was supported by grant CEX2021-001169-M, funded by MICIU\slash{}AEI\slash{}10.13039\slash{}501100011033, as well as the Acadèmia d'Excel{·}lència scheme funded by the Generalitat de Catalunya and grant PID2023-149556NB-I00 of the Spanish Ministry of Science and Innovation.

\bibliographystyle{plainurl}
\bibliography{biblio}

\begin{thebibliography}{10}

\bibitem{Aguilera2025IntuitionisticGW}
Juan~P. Aguilera and Leonardo Pacheco.
\newblock Intuitionistic {G}{\"o}del-{L}{\"o}b without sharps.
\newblock {\em ACM Transactions on Computational Logic}, 26:1 -- 14, 2025.
\newblock URL: \url{https://api.semanticscholar.org/CorpusID:280254585}.

\bibitem{AlechinaMPR01}
Natasha Alechina, Michael Mendler, Valeria de~Paiva, and Eike Ritter.
\newblock Categorical and {Kripke} semantics for constructive {S4} modal logic.
\newblock In Laurent Fribourg, editor, {\em Proceedings of 15th International
  Workshop on Computer Science Logic (CSL 2001), 10th Annual Conference of the
  EACSL, Paris, France}, volume 2142 of {\em Lecture Notes in Computer
  Science}, pages 292--307, Heidelberg, Germany, 2001. Springer-Verlag.
\newblock \href {https://doi.org/10.1007/3-540-44802-0_21}
  {\path{doi:10.1007/3-540-44802-0_21}}.

\bibitem{balbiani2026}
Philippe Balbiani, Martín Diéguez, David Fernández-Duque, and Brett McLean.
\newblock Constructive {S}4 modal logics with the finite birelational frame
  property.
\newblock {\em arXiv}, cs.LO(2403.00201), 2026.

\bibitem{DasIGL}
Anupam Das, Iris van~der Giessen, and Sonia Marin.
\newblock Intuitionistic g{\"{o}}del-l{\"{o}}b logic, {\`{a}} la simpson:
  Labelled systems and birelational semantics.
\newblock In Aniello Murano and Alexandra Silva, editors, {\em 32nd {EACSL}
  Annual Conference on Computer Science Logic, {CSL} 2024, Naples, Italy,
  February 19-23, 2024}, volume 288 of {\em LIPIcs}, pages 22:1--22:18. Schloss
  Dagstuhl - Leibniz-Zentrum f{\"{u}}r Informatik, 2024.
\newblock URL: \url{https://doi.org/10.4230/LIPIcs.CSL.2024.22}, \href
  {https://doi.org/10.4230/LIPICS.CSL.2024.22}
  {\path{doi:10.4230/LIPICS.CSL.2024.22}}.

\bibitem{me:simulability}
D.~Fern{\'a}ndez-Duque.
\newblock On the modal definability of simulability by finite transitive
  models.
\newblock {\em Studia Logica}, 98:347--373, August 2011.
\newblock \href {https://doi.org/http://dx.doi.org/10.1007/s11225-011-9339-x}
  {\path{doi:http://dx.doi.org/10.1007/s11225-011-9339-x}}.

\bibitem{FDIGL}
David Fernández-Duque.
\newblock Classical companions for intuitionistic {G}ödel-{L}öb logic, 09
  2026.
\newblock URL:
  \url{https://www.researchgate.net/publication/414929178_Classical_Companions_for_Intuitionistic_Godel-Lob_Logic},
  \href {https://doi.org/10.13140/RG.2.2.34019.72488}
  {\path{doi:10.13140/RG.2.2.34019.72488}}.

\bibitem{Fine1974}
K.~Fine.
\newblock An ascending chain of {S}4 logics.
\newblock {\em Theoria}, 40(2):110--116, 1974.

\bibitem{Fin74c}
K.~Fine.
\newblock Logics containing {$K4$}. {I}.
\newblock {\em J. Symbolic Logic}, 39:31--42, 1974.

\bibitem{servi1977modal}
Gis{\`e}le {Fischer Servi}.
\newblock On modal logic with an intuitionistic base.
\newblock {\em Studia Logica}, 36(3):141--149, 1977.

\bibitem{servi1984axiomatizations}
Gis{\`e}le {Fischer Servi}.
\newblock Axiomatizations for some intuitionistic modal logics.
\newblock {\em Rend. Sem. Mat. Univers. Politecn. Torino}, 42(3):179--194,
  1984.

\bibitem{pml}
D.~Gabelaia, A.~Kurucz, F.~Wolter, and M.~Zakharyaschev.
\newblock Non-primitive recursive decidability of products of modal logics with
  expanding domains.
\newblock {\em Annals of Pure and Applied Logic}, 142(1-3):245--268, 2006.

\bibitem{GKMMS23}
Marianna Girlando, Roman Kuznets, Sonia Marin, and Lutz Straßburger.
\newblock A decision procedure for intuitionistic modal logic {IS4} (and
  {IK4}).
\newblock {\em arXiv}, cs.LO(2609.24922), 2026.

\bibitem{krusty}
J.~B. Kruskal.
\newblock Well-quasi-ordering, the tree theorem, and vazsonyi's conjecture.
\newblock {\em Transactions of the American Mathematical Society},
  95(2):210--225, 1960.

\bibitem{Lyon}
Tim~S. Lyon.
\newblock Nested sequents for intuitionistic multi-modal logics: Modularity,
  cut-elimination, and undecidability.
\newblock {\em CoRR}, abs/2511.22174, 2025.
\newblock URL: \url{https://doi.org/10.48550/arXiv.2511.22174}, \href
  {https://arxiv.org/abs/2511.22174} {\path{arXiv:2511.22174}}, \href
  {https://doi.org/10.48550/ARXIV.2511.22174}
  {\path{doi:10.48550/ARXIV.2511.22174}}.

\bibitem{piazza}
Mario Piazza.
\newblock A {K}ruskal decision procedure for intuitionistic modal logic ik4.
\newblock {\em arXiv}, cs.LO(2608.10283), 2026.

\bibitem{santiago2026}
Sofía Santiago-Fernández, David Fernández-Duque, and Joost~J. Joosten.
\newblock The complexity of the constructive master modality.
\newblock {\em arXiv}, cs.LO(2603.05131), 2026.

\bibitem{Simpson94}
A.~Simpson.
\newblock {\em The proof theory and semantics of intuitionistic modal logic}.
\newblock PhD thesis, University of Edinburgh, Edinburgh, UK, 1994.

\bibitem{Strassburger}
Lutz Stra{\ss}burger.
\newblock Cut elimination in nested sequents for intuitionistic modal logics.
\newblock In Frank Pfenning, editor, {\em Foundations of Software Science and
  Computation Structures}, pages 209--224, Berlin, Heidelberg, 2013. Springer
  Berlin Heidelberg.

\end{thebibliography}

%
%
%
%
%

\end{document}